\documentclass[final,1p,times]{elsarticle}

\usepackage{amssymb,color}
\usepackage{amsmath}
\usepackage{amsthm}
\newcommand\dela[1]{}%%13 December 2013
\newcommand{\embed}{\hookrightarrow }

\newcommand{\lb}{\langle}
\newcommand{\rb}{\rangle}

\numberwithin{equation}{section}
\newtheorem{theorem}{Theorem}[section]

\newtheorem{lemma}[theorem]{Lemma}
\newtheorem{cor}[theorem]{Corollary}

\newtheorem{definition}[theorem]{Definition}

\newcommand{\nat}{\mathbb{N}}
\newcommand{\rzecz}{\mathbb{R}}
\newcommand{\eps}{\varepsilon }

\newcommand{\rd}{{\rzecz }^{d}}

\newcommand{\modulus}{\mbox{\rm m}}
\newcommand{\qvar}[1]{\bigl< \! \bigl< #1  \bigr> \! \bigr> }

\newcommand{\tr}{\mbox{\rm Tr}}

\newcommand{\diver}{\mbox{\rm div}\,}
\newcommand{\supp}{\mbox{\rm supp}\, }

\newcommand{\acal}{\mathcal{A}}

\newcommand{\ccal}{\mathcal{C}}

\newcommand{\fcal}{\mathcal{F}}

\newcommand{\hcal}{\mathcal{H}}
\newcommand{\kcal}{\mathcal{K}}
\newcommand{\lcal}{\mathcal{L}}

\newcommand{\ocal}{\mathcal{O}}

\newcommand{\scal}{\mathcal{S}}
\newcommand{\tcal}{\mathcal{T}}
\newcommand{\vcal}{\mathcal{V}}
\newcommand{\xcal}{\mathcal{X}}

\newcommand{\zcal}{\mathcal{Z}}

\newcommand{\fmath}{\mathbb{F}}

\newcommand{\smath}{\mathbb{S}}

\newcommand{\nlim}{\lim_{n \to \infty }}
\newcommand{\kinf}{k \to \infty }
\newcommand{\ninf}{n \to \infty }

\newcommand{\ball}{\mathbb{B}}
\newcommand{\ind}[1]{{1\! \!  1 }_{#1}  }

\newcommand{\norm}[3]{\|  #1 {\| }_{#2}^{#3}}

\newcommand{\ilsk}[3]{{\bigl( #1 , #2 \bigr)}_{#3}}
\newcommand{\Ilsk}[3]{\Bigl( #1 , #2 {\Bigr)}_{#3}}

\newcommand{\dual}[3]{{\lb #1 , #2 \rb}_{#3}}
\newcommand{\Dual}[3]{{\Bigl< #1  , #2 \Bigr>}_{#3}}

\newcommand{\dirilsk}[3]{{\bigl( \! \bigl( #1 , #2 \bigr) \! \bigr)}_{#3}}

\newcommand{\ddual}[4]{{}_{#1}\lb #2 ,#3 {\rb }_{#4}}

\newcommand{\p}{\mathbb{P}}
\newcommand{\e}{\mathbb{E}}

\newcommand{\fn}{{f}_{n}}

\newcommand{\ft }{{\fcal }_{t}}

\newcommand{\tOmega}{\tilde{\Omega}}

\newcommand{\tfcal}{\tilde{\fcal}}
\newcommand{\tp}{\tilde{\p}}

\newcommand{\te}{\tilde{\e}}

\newcommand{\Xn}{{X}_{n}}

\newcommand{\tMn}{{\tilde{M}}_{n}}

\newcommand{\Pn}{{P}_{n}}

\newcommand{\un}{{u}_{n}}
\newcommand{\Jn}[1]{{J}^{n}_{#1}}

\newcommand{\unk}{{u}_{{n}_{k}}}
\newcommand{\taun}{{\tau}_{n}}
\newcommand{\lhs}{{\mathcal{T}_2}}

\newcommand{\Bn}{{B}_{n}}

\newcommand{\tM}{\tilde{M}}

\newcommand{\ttfcal}{\tilde{\tilde{\fcal}}}

\newcommand{\ttOmega}{\tilde{\tilde{\Omega }}}
\newcommand{\ttp}{\tilde{\tilde{\p }}}

\newcommand{\ttW}{\tilde{\tilde{W }}}

\newcommand{\hi}{{h}^{(i)}}
\newcommand{\bi}{{b}^{(i)}}
\newcommand{\bij}{{b}^{(i)}_{j}}
\newcommand{\bik}{{b}^{(i)}_{k}}
\newcommand{\ci}{{c}^{(i)}}

\newcommand{\tun}{{\tilde{u}}_{n}}
\newcommand{\tunk}{{\tilde{u}}_{{n}_{k}}}

\newcommand{\tu}{\tilde{u}}
\newcommand{\ttu}{\tilde{\tilde{u}}}

\newcommand{\vn}{{v}_{n}}

\journal{Journal of Differential Equations}

\begin{document}

\begin{frontmatter}
\title{Existence of a martingale solution of the stochastic Navier-Stokes equations
in  unbounded 2D and 3D domains.}

\author[ZB]{Zdzis\l aw Brze\'{z}niak}
\ead{zdzislaw.brzezniak@york.ac.uk}
\author[EM]{El\.zbieta Motyl}
\ead{emotyl@math.uni.lodz.pl}

\address[ZB]{Department of Mathematics, University of York, Heslington, York,
YO105DD, United Kingdom}
\address[EM]{Department of Mathematics and Computer Sciences, University of \L\'{o}d\'{z}, ul. Banacha 2,
91-238 \L \'{o}d\'{z}, Poland}

\begin{abstract}
Stochastic Navier-Stokes equations in 2D and 3D possibly unbounded domains driven by a multiplicative Gaussian noise are considered. The noise term depends on the unknown velocity and its spatial derivatives.
The existence of a martingale solution is proved.
The construction of the solution is based on the classical Faedo-Galerkin approximation,
 the compactness method and the Jakubowski version of the Skorokhod Theorem for nonmetric spaces.
 Moreover, some compactness and tightness criteria in nonmetric spaces are proved. Compactness results are based on a certain generalization of the classical Dubinsky Theorem.
\end{abstract}


\end{frontmatter}

\section{Introduction.}

\bigskip  \noindent
Let $\ocal \subset \rd $ be an open connected possibly unbounded subset with smooth boundary $\partial \ocal $, where $d=2,3$.
We consider the stochastic Navier-Stokes equations
$$
   \frac{\partial u }{\partial t } + (u \cdot \nabla ) u
  -  \nu \Delta u  + \nabla p =  f(t) +  G(t,u) \, d W(t) , \qquad t \in [0,T] ,
$$
in $\ocal $, with the incompressibility condition
$$
  \diver u =0
$$
and with the homogeneous boundary condition ${u}_{| \partial \ocal }=0$.
In this problem $u=u(t,x)$

\noindent
$=({u}_{1}(t,x), ...,{u}_{d}(t,x))$ and $p=p(t,x)$
represent the velocity and the pressure of the fluid. Furthermore, $f$ stands for the deterministic
external forces and $G(t,u)\, d W(t)$, where $W$ is a cylindrical Wiener process, stands for the random forces.

\bigskip  \noindent
The above problem can be written in an abstract form as the following initial value problem
$$
\begin{cases}
& du(t)+\acal u(t) \, dt+B\bigl( u(t)\bigr)\, dt = f(t) \, dt+G\bigl( u(t)\bigr) \, dW(t),
 \qquad t \in [0,T] , \\
& u(0) = {u}_{0} .
\end{cases}
$$
Here $\acal $ and $B$ are appropriate maps corresponding to the Laplacian and the nonlinear term, respectively in the Navier-Stokes equations, see Section \ref{S:Functional_setting}.
We impose rather general assumptions (G), (\ref{E:G*}) and  (\ref{E:G**}) on the noise $G(t,u ) d W(t)$ formulated in (A.3) in Section \ref{S:Navier_Stokes}.
These assumptions cover the following special case
$$
 G(t, u ) d W(t) := \sum_{i=1}^{\infty } \bigl[
 \bigl( {b}^{(i)}(x) \cdot \nabla  \bigr) u (t,x) + {c}^{(i)} (x) u (t,x) \bigr]
d {\beta }^{(i)}(t) ,
$$
where  $\{ {\beta }^{(i)} { \} }_{i\in \nat }$ are independent real standard Brownian motions,
see Section \ref{S:Example}.

\bigskip  \noindent
We prove the existence of a martingale solution.
The construction of a solution is based on the classical Faedo-Galerkin approximation, i.e.
$$
\begin{cases}
  & d \un (t) =  - \bigl[ \Pn \acal \un (t)  + \Bn  \bigl(\un (t)\bigr)  - \Pn f(t) \bigr] \, dt + \Pn G\bigl( \un (t)\bigr) \, dW(t),   \quad t \in [0,T] , \\
  &  \un (0) = \Pn {u }_{0}
\end{cases}
$$
given in  Section \ref{S:Existence}.
The crucial point is to prove suitable uniform \it a priori \rm estimates on $\un $. Analogously to \cite{Flandoli+Gatarek_1995}, we prove that the following estimates hold
$$
  \sup_{n \ge 1 }  \e \bigl[ \int_{0}^{T} \norm{ \un (s)}{V}{2} \, ds \bigr] < \infty .
$$
and
$$
 \sup_{n \ge 1 } \e \bigl( \sup_{0 \le s \le T } |\un (s){|}_{H}^{p} \bigr) < \infty  .
$$
for $p\in  \bigl[ 2, 2+ \frac{\eta }{2-\eta } \bigr)$, where  $\eta \in(0,2)$ is  given parameter, see Section \ref{S:Navier_Stokes}.
Here, $V$ and $H $ denote the closures  in the Sobolev space ${H}^{1}(\ocal , \rd )$ and ${L}^{2}(\ocal , \rd )$, respectively of the divergence-free vector fields of class ${\ccal }^{\infty }$ with compact supports contained in $\ocal $. The solutions $\un $ to the Galerkin scheme generate a sequence of laws $\{ \lcal (\un), n \in \nat \} $ on appropriate functional spaces. To prove that this sequence of probability measures is weakly compact we need appropriate tightness criteria.

\bigskip  \noindent
In Section \ref{S:Compactness} we prove certain deterministic compactness results, see Lemmas
\ref{L:Dubinsky_ext} and \ref{L:comp}.
If $\ocal $ is unbounded, then the embedding $V \hookrightarrow H $ is not compact.
However, using Lemma  2.5 from \cite{Holly_Wiciak_1995}
(see Lemma  C.1 in Appendix C), we can find a separable Hilbert space $U$
such that
$$
    U \subset V \subset H
$$
the embedding $\iota: U\hookrightarrow V$ being dense and compact.
Then we have
$$
  U
 \underset{\iota }{\hookrightarrow } H \cong H^{\prime }
  \underset{{\iota }^{\prime }}{\hookrightarrow } U^{\prime },
$$
where $H^{\prime }$ and $U^{\prime }$ are the dual spaces of $H$ and $U$, respectively, $H^{\prime }$ being identified with $H$
and ${\iota }^{\prime }$ is the dual operator to the embedding $\iota $.
Moreover, ${\iota }^{\prime }$ is compact as well.
Modifying the proof of the classical Dubinsky Theorem, \cite[Theorem IV.4.1]{Vishik_Fursikov_88}, we obtain
a certain deterministic compactness criterion,
see Lemma \ref{L:Dubinsky_ext}. Namely, we will show that a set $\kcal $ is relatively compact in the intersection
$$
  \tilde{\zcal }:= \ccal ([0,T]; U^{\prime }) \cap {L}_{w}^{2}(0,T;V)  \cap {L}^{2}(0,T;{H}_{loc})
$$
if the following two conditions hold:
\begin{itemize}
\item $ \sup_{u\in \kcal } \int_{0}^{T} \norm{u(s)}{V}{2} \, ds < \infty  $,
    i.e. $\kcal $ is bounded in ${L}^{2}(0,T;V)$,
\item $ \lim{}_{\delta \to 0 } \sup_{u\in \kcal }
  \sup_{\overset{s,t \in[0,T]}{|t-s|\le \delta }} {|u(t)-u(s)|}_{U^{\prime }} =0 $.
\end{itemize}
Here $\ccal ([0,T];U^{\prime })$ denotes the space of $U^{\prime }$-valued continuous functions, ${L}_{w}^{2}(0,T;V) $ is the spaces ${L}^{2}(0,T;V) $ equipped with the weak topology and  ${L}^{2}(0,T;{H}_{loc})$ is a Fr\'{e}chet space defined in (\ref{E:seminorms}) in Section \ref{S:Compactness}.
Let us notice that the second condition implies the equicontinuity of the family $\kcal$ of $U^{\prime }$-valued functions.
Thus the above two conditions are the same as in the Dubinsky Theorem. However, since the embedding $V \hookrightarrow H$ is not compact, then in comparison to the Dubinsky Theorem we have the space ${L}^{2}(0,T;{H}_{loc})$ instead of ${L}^{2}(0,T;H)$.

\bigskip \noindent
Next, using this version of the Dubinsky Theorem,  we will prove another deterministic compactness criterion, see Lemma \ref{L:comp}. Namely, we will show that a set $\kcal $ is relatively compact in the intersection
$$
  \zcal : = \ccal ([0,T]; U^{\prime }) \cap {L}_{w}^{2}(0,T;V)  \cap {L}^{2}(0,T;{H}_{loc})
  \cap \ccal ([0,T];{H}_{w})
$$
if  the following three conditions hold:
\begin{description}
\item[(a)] $\sup_{u\in \kcal } \sup_{s \in[0,T]} {|u(s)|}_{H}^{2} < \infty  $,
\item[(b)] $ \sup_{u\in \kcal } \int_{0}^{T} \norm{u(s)}{}{2} \, ds < \infty  $,
\item[(c)] $ \lim{}_{\delta \to 0 } \sup_{u\in \kcal }
  \sup_{\overset{s,t \in[0,T]}{|t-s|\le \delta }} {|u(t)-u(s)|}_{U^{\prime }} =0 $.
\end{description}

\bigskip  \noindent
These results were inspired by Lemma 2.7 due to  Mikulevicius and  Rozovskii \cite{Mikulevicius+Rozovskii_2005}, where the case of $\ocal := \rd $, $d \ge 2$ is considered.
In \cite{Mikulevicius+Rozovskii_2005} the space ${L}^{2}(\rd )$ is compactly embedded in the Fr\'{e}chet space ${H}^{-{k}_{0}}_{loc}(\rd )$ for sufficiently large ${k}_{0}$. Then, the authors prove the deterministic compactness criterion in the intersection
$$
  \ccal ([0,T];{H}^{-{k}_{0}}_{loc}(\rd )) \cap \ccal ([0,T]; {L}^{2}_{w}(\rd ))
  \cap {L}^{2}_{w}(0,T; {H}^{1}(\rd )) \cap {L}^{2}(0,T; {L}^{2}_{loc}(\rd )).
$$
The main difference with the approach of  Mikulevicius and  Rozovskii, is that instead of the
Fr\'{e}chet space ${H}^{-{k}_{0}}_{loc}(\rd )$, we consider the space $U^{\prime }$ dual to the Hilbert space $U$ constructed in a special way by  Holly and  Wiciak, see \cite[Lemma  2.5]{Holly_Wiciak_1995}.
This allows us to prove the above mentioned modification of the Dubinsky Theorem.
The space $U$ will be also of crucial importance in further construction of a martingale solution.

\bigskip  \noindent
Using Lemma \ref{L:comp} and the Aldous condition in the form given by  M\'{e}tivier \cite{Metivier_88}, we obtain a new tightness criterion for the laws on the space $\zcal $, see Corollary
\ref{C:tigthness_criterion}. Next, we prove that the set of laws $\{ \lcal (\un), n \in \nat \} $ is tight on $\zcal $. The next step in our construction of a martingale solution differs from the approach of  Mikulevicius and  Rozovskii.
We apply the method used by  Da Prato and  Zabczyk in \cite[Chapter 8]{DaPrato+Zabczyk_1992}. This method is based on the Skorokhod Theorem and the Martingale Representation Theorem.
However, we will apply the Jakubowski's version of the Skorokhod Theorem for nonmetric spaces in the form given by Brze\'{z}niak and Ondrej\'{a}t \cite{Brz+Ondrejat_2011},
\cite{Jakubowski_1998}. In \cite[Chapter 8]{DaPrato+Zabczyk_1992} the authors impose the linear growth conditions on the nonlinear term and assume the compactness of the appropriate semigroup.
The assumptions considered in \cite[Chapter 8]{DaPrato+Zabczyk_1992} do not cover the stochastic Navier-Stokes equations, however, we can use the ideas introduced there.
This method seems to us more direct.
In the case of 2D domains we prove moreover the existence and uniqueness of strong solutions.

\bigskip  \noindent
Stochastic Partial Differential Equations (SPDEs)  can  be  viewed  as  an intersection    of  the  infinite-dimensional Stochastic
Analysis and Partial Differential Equations. The  theory of SPDEs  began in the early  $70$'s with
works of Bensoussan and Temam  \cite{Bens+Temam_1973}, Dawson and Salehi \cite{Daws+Salehi_1975}
and many others. Due to contributions of several authors such as Pardoux \cite{Pardoux79}
  Krylov and Rozovskii \cite{Kryl+Ros_79}, Da Prato and Zabczyk \cite{DaPrato+Zabczyk_1992} many aspects  of this new theory  are now well developed and understood. The study  of stochastic NSEs   initiated in \cite{Bens+Temam_1973}  was continued by many, for instance  Brze{\'z}niak et al. \cite{BCF3,BCF4}, Flandoli and G\c{a}tarek \cite{Flandoli+Gatarek_1995}, Capi\'{n}ski and Peszat \cite{Capinski_Peszat_1997}, Hairer and Mattingly \cite{Hairer+Matt_2006} and  Mikulevicius and Rozovskii \cite{Mikulevicius+Rozovskii_2005}. In the last paper the authors study the existence of a martingale solution of the stochastic Navier-Stokes equations for turbulent flows in
$\rd $, ($d \ge 2 $) corresponding to the Kraichnan model of turbulence.

\bigskip  \noindent
Stochastic Navier-Stokes equations in unbounded 2D and 3D domains with the noise independent on $\nabla u$
were considered by Capi\'{n}ski and Peszat \cite{Capinski_Peszat_2001} and Brze\'{z}niak and Peszat \cite{Brz+Peszat_2001}.
The solutions are constructed in weighted spaces. Invariant measures for stochastic Navier-Stokes equations with an additive noise in some unbounded 2D domains are investigated by Brze\'{z}niak and Li \cite{Brz+Li_2006}. Our results generalize the corresponding existence results from \cite{Capinski_Peszat_2001}, \cite{Brz+Li_2006} \cite{Flandoli+Gatarek_1995},  and \cite{Mikulevicius+Rozovskii_2005}.
What concerns modelling of noise, we have tried to be as general as possible and include the roughest noise possible. One should bear in mind that  the rougher the noise the closer the model is to
reality. Moreover,  Landau and Lifshitz in
 their fundamental 1959 work \cite[Chapter 17]{LL_1987}  proposed
to study NSEs under additional stochastic  small fluctuations.
Consequently these authors considered the classical balance laws for mass,
energy and momentum forced
by a random noise, to describe the fluctuations, in particular local
stresses and temperature,
which are not related to the gradient of the corresponding quantities.
In \cite[Chapter 12]{LL_1968} the same authors  found
correlations for the random forcing by following the general theory of
fluctuations. One of the requirements imposed on the noise is that
it is either spatially uncorrelated or correlated as little as
possible.

\bigskip  \noindent
The present paper is organized as follows. In Section \ref{S:Functional_setting} we recall basic definitions and introduce some auxiliary operators.
In Section \ref{S:Compactness} we are concerned with the compactness results. In Section
\ref{S:Navier_Stokes}, we formulate the Navier-Stokes problem as an abstract stochastic evolution equation. The main theorem about existence and construction of the martingale solution is in Section \ref{S:Existence}.
Some auxiliary results connected with the proof are given in Appendices A and B. In Section \ref{S:Example}, we consider some example of the noise. For the convenience of the reader, in Appendix C we recall Lemma  2.5 from \cite{Holly_Wiciak_1995} with the proof.
Section \ref{S:2D_NS} and Appendix D are devoted to  2D Navier-Stokes equations.

\subsection*{Acknowledgments} The authors would like to thank an anonymous referee for most useful queries and comments which led to the improvement of the paper.

%%%%%%%%%%%%%%%%%%%%%%%%%%%%%%%%%%%%%%%%%%%%%%%%%%%%%%%%%%%%%%%%%%%%%%%%%%%%%%%%%%%%%%%%%%%%%%%%%%%%
\bigskip
\section{Functional setting} \label{S:Functional_setting}
%%%%%%%%%%%%%%%%%%%%%%%%%%%%%%%%%%%%%%%%%%%%%%%%%%%%%%%%%%%%%%%%%%%%%%%%%%%%%%%%%%%%%%%%%%%%%%%%%%%%%

%\bigskip
%%%%%%%%%%%%%%%%%%%%%%%%%%%%%%%%%%%%%%%%%%%%%%%%%%%%%%%%%%%%%%%%%%%%%%%%%%%%%%%%%%%%%%%%%%%%%%%%%%%%
\subsection{Notations.}
%%%%%%%%%%%%%%%%%%%%%%%%%%%%%%%%%%%%%%%%%%%%%%%%%%%%%%%%%%%%%%%%%%%%%%%%%%%%%%%%%%%%%%%%%%%%%%%%%%%%%%
\noindent
Let $(X,{|\cdot |}_{X}), (Y,{|\cdot |}_{Y})$ be two real normed spaces. The symbol $\lcal (X,Y)$ stands for the space of all bounded linear operators from $X$ to $Y$. If $Y=\rzecz $, then $X^\prime :=\lcal (X,\rzecz )$ is called the dual space of $X$. The symbol
$\ddual{{X}^{\prime }}{\cdot }{\cdot }{X}$
denotes the standard duality pairing. If no confusion seems likely we omit the subscripts $X^\prime ,X$ and write $\dual{\cdot }{\cdot }{}$.
If both spaces $X$ and $Y$ are separable Hilbert, then by $\lhs(Y,X)$ we will denote the Hilbert space of all Hilbert-Schmidt operators from $Y$ to $X$ endowed with the standard norm.

\bigskip  \noindent
Assume that $X,Y$ are Banach spaces. Let $A$ be a densely defined linear operator from $X$ to $Y$ and let $D(A)$ denote the domain of $A$. Let
$$
   D(A^{\prime }):= \{ {y}^{\prime }\in Y^\prime: \, \, \, \mbox{the linear functional } {y}^{\prime } \circ A : D(A) \to \rzecz
   \mbox{ is bounded.} \}
$$
Note that $D(A^\prime)=Y^\prime$ if $A$ is bounded. Let ${y}^\prime \in D(A^\prime)$. Since $A$ is densely defined, the functional ${y}^\prime \circ A $ can be uniquely extended to the linear bounded functional $\overline{{y}^\prime \circ A } \in X^\prime $.
The operator $A^\prime$ defined by
$$
   A^\prime   {y}^\prime := \overline{{y}^\prime \circ A } , \qquad {y}^\prime \in D(A^\prime ),
$$
is called the dual operator of $A$.

\bigskip  \noindent
Assume that $X,Y$ are Hilbert spaces with scalar products $\ilsk{\cdot }{\cdot }{X}$ and $\ilsk{\cdot }{\cdot }{Y}$, respectively.
Let $A:X \supset D(A) \to Y$ be a densely defined linear operator. By ${A}^{\ast }$ we denote the adjoint operator of $A$. In particular, $D({A}^{\ast }) \subset Y $, ${A}^{\ast }: D({A}^{\ast }) \to X $ and
$$
   \ilsk{Ax}{y}{Y} = \ilsk{x}{{A}^\ast y}{X} , \qquad x \in D(A), \quad y \in D({A}^{\ast }) .
$$
Note that $D({A}^\ast)=Y$ if $A$ is bounded.

\bigskip
%%%%%%%%%%%%%%%%%%%%%%%%%%%%%%%%%%%%%%%%%%%%%%%%%%%%%%%%%%%%%%%%%%%%%%%%%%%%%%%%%%%%%%%%%%%%%%%%%%
\subsection{Basic definitions}
%%%%%%%%%%%%%%%%%%%%%%%%%%%%%%%%%%%%%%%%%%%%%%%%%%%%%%%%%%%%%%%%%%%%%%%%%%%%%%%%%%%%%%%%%%%%%%%%%

%\bigskip
\noindent
Let $\ocal \subset \rd $ be an open subset with smooth boundary $\partial \ocal $, $d=2,3$.
Let $p\in (1,\infty )$ and let ${L}^{p}(\ocal , \rd )$ denote the Banach space of Lebesgue measurable $\rd $-valued $p$-th power integrable functions on the set $\ocal $. The norm in ${L}^{p}(\ocal , \rd )$ is given by
$$
      \norm{u}{{L}^{p}}{} := \biggl( \int_{\ocal } |u(x){|}^{p} \, dx {\biggr) }^{\frac{1}{p}} , \qquad u \in {L}^{p}(\ocal , \rd ).
$$
By ${L}^{\infty } (\ocal , \rd )$ we denote the Banach space of Lebesgue measurable essentially bounded $\rd $-valued functions defined on $\ocal $. The norm is given by
$$
     \norm{u}{{L}^{\infty }}{}:= \mbox{\rm esssup} \,\{  |u(x)| , \, \, x \in \ocal  \}  , \qquad u \in {L}^{\infty } (\ocal , \rd ).
$$
If $p=2$, then ${L}^{2}(\ocal , \rd )$ is a Hilbert space with the scalar product given by
$$
  \ilsk{u}{v}{{L}^{2}} := \int_{\ocal } u(x) \cdot v (x) \, dx , \qquad u,v \in {L}^{2}(\ocal ,\rd ).
$$
Let ${H}^{1}(\ocal ,\rd )$ stand for the Sobolev space of all $u \in {L}^{2}(\ocal ,\rd )$ for which
there exist weak derivatives $\frac{\partial u}{\partial {x}_{i}} \in {L}^{2}(\ocal ,\rd )$, $i=1,2,...,d$.
It is a Hilbert space with the scalar product given by
$$
   \ilsk{u}{v}{{H}^{1}} := \ilsk{u}{v}{{L}^{2}} +  \dirilsk{u}{v}{} , \qquad u,v \in
   {H}^{1}(\ocal ,\rd ),
$$
where
\begin{equation} \label{E:il_sk_Dir}
  \dirilsk{u}{v}{} :=
\ilsk{\nabla u}{\nabla v}{{L}^{2}}
=\sum_{i=1}^{d}\int_{\ocal }\frac{\partial u}{\partial {x}_{i}} \cdot \frac{\partial v}{\partial {x}_{i}} \, dx  , \qquad u,v \in  {H}^{1}(\ocal ,\rd ).
\end{equation}
Let ${\ccal }^{\infty }_{c} (\ocal , \rd )$ denote the space of all $\rd $-valued functions of class ${\ccal }^{\infty }$ with compact supports contained in $\ocal $ and
let
\begin{eqnarray*}
 &  &\vcal := \{ u \in {\ccal }^{\infty }_{c} (\ocal , \rd ) : \, \, \diver u= 0 \}  ,\\
 & & H := \mbox{the closure of $\vcal $ in ${L}^{2}(\ocal , \rd )$} , \\
 & & V := \mbox{the closure of $\vcal $ in ${H}^{1}(\ocal , \rd )$} .
\end{eqnarray*}
In the space $H$ we consider the scalar product and the norm inherited from ${L}^{2}(\ocal , \rd )$ and
denote them by $\ilsk{\cdot }{\cdot }{H}$ and $|\cdot {|}_{H}$, respectively, i.e.
$$
\ilsk{u}{v}{H}:= \ilsk{u}{v}{{L}^{2}} , \qquad
 | u {|}_{H} := \norm{u}{{L}^{2}}{} , \qquad u, v \in H.
$$
In the space $V$ we consider the scalar product
inherited from ${H}^{1}(\ocal , \rd )$, i.e.
\begin{equation} \label{E:V_il_sk}
  \ilsk{u}{v}{V} := \ilsk{u}{v}{H} + \dirilsk{u}{v}{} ,
\end{equation}
where $\dirilsk{\cdot }{\cdot }{}$ is defined in (\ref{E:il_sk_Dir}),
and the norm  induced by $\ilsk{\cdot }{\cdot }{V}$, i.e.
\begin{equation} \label{E:norm_V}
  \norm{u}{V}{2} := |u {|}_{H}^{2} + \norm{u}{}{2},
\end{equation}
where $\norm{u}{}{2} := \norm{\nabla u}{{L}^{2}}{2}$.

\bigskip  \noindent
Let us consider the following tri-linear form
\begin{equation}  \label{E:form_b}
     b(u,w,v ) = \int_{\ocal  }\bigl( u \cdot \nabla w \bigr) v \, dx .
\end{equation}
We will recall the fundamental properties of the form $b$. Since usually one considers the bounded domain case we want to recall only those results that are valid in unbounded domains as well.

\bigskip \noindent
By the Sobolev Embedding Theorem, see \cite{Adams}, and the H\"{o}lder inequality, we obtain the following estimates
\begin{eqnarray}
    |b(u,w,v )|
& \le & {\| u\| }_{{L}^{4}} \norm{w }{V}{}{\| v\| }_{{L}^{4}}   %\label{E:b_estimate_L^4}
 \\
& \le & c \norm{u }{V}{} \norm{w }{V}{} \norm{v }{V}{} , \qquad u,w,v \in V   \label{E:b_estimate_V}
\end{eqnarray}
for some positive constant $c$. Thus the form $b$ is continuous on $V$, see also \cite{Temam79}.
Moreover, if we define a bilinear map $B$ by $B(u,w):=b(u,w, \cdot )$, then by inequality (\ref{E:b_estimate_V}) we infer that $B(u,w) \in {V}_{}^{\prime }$ for all $u,w\in V$ and that the following inequality holds
\begin{equation}  \label{E:estimate_B}
 |B(u,w) {|}_{V^{\prime }} %\le {c}_{1} {\| u\| }_{{L}^{4}} \norm{w }{V}{}
   \le c  \norm{u }{V}{}\norm{w }{V}{},\qquad u,w \in V .
\end{equation}
Moreover, the mapping $B: V \times V \to V^{\prime } $ is bilinear and continuous.

\bigskip  \noindent
Let us also recall the following properties of the form $b$, see Temam \cite{Temam79}, Lemma II.1.3,
\begin{equation}  \label{E:antisymmetry_b}
b(u,w, v ) =  - b(u,v ,w), \qquad u,w,v \in V .
\end{equation}
In particular,
\begin{equation}  \label{E:wirowosc_b}
b(u,v,v) =0   \qquad u,v \in V.
\end{equation}

\bigskip  \noindent
Let us, for any $s>0$ define the following standard scale of Hilbert spaces
$$
  {V}_{s} := \mbox{the closure of $\vcal $ in ${H}^{s}(\ocal , \rd )$} .
$$
If $s > \frac{d}{2} +1$ then by the Sobolev Embedding Theorem,
$$
    {H}^{s-1}(\ocal  , \rd ) \hookrightarrow  {\ccal }_{b}(\ocal , \rd )
   \hookrightarrow {L}^{\infty } (\ocal , \rd ).
$$
Here ${\ccal }_{b}(\ocal , \rd )$ denotes the space of continuous and bounded $\rd $-valued functions defined on $\ocal $.
If $u,w \in V$ and $v \in {V}_{s}$ with $s > \frac{d}{2} +1$ then
$$
 |b(u,w,v)|  =  |b(u,v,w)|
   \le   \norm{u}{{L}^{2}}{} \norm{w}{{L}^{2}}{} \norm{\nabla v}{{L}^{\infty }}{}
 \le  {c}_{} \norm{u}{{L}^{2}}{} \norm{w}{{L}^{2}}{} \norm{v}{{V}_{s}}{}
$$
for some constant ${c}_{} >0 $. Thus, $b$ can be uniquely extended to the tri-linear form (denoted by the same letter)
$$
   b : H \times H \times {V}_{s} \to \rzecz
$$
and $|b(u,w,v)| \le  {c}_{} \norm{u}{{L}^{2}}{} \norm{w}{{L}^{2}}{} \norm{v}{{V}_{s}}{}$
for $u,w \in H$ and $v \in {V}_{s}$. At the same time the operator $B$ can be uniquely extended
to a bounded bilinear operator
$$
    B : H \times H \to {V}_{s}^{\prime } .
$$
In particular, it satisfies the following estimate
\begin{equation}  \label{E:estimate_B_ext}
 |B(u,w) {|}_{{V}_{s}^{\prime }} \le c {|u|}_{H}  {|w|}_{H} ,\qquad u,w \in H.
\end{equation}
We will also use the following notation, $B(u):=B(u,u)$.

\bigskip
\begin{lemma} \ \label{L:estimate_B}
The map $B:V \to V^{\prime }$ is locally Lipschitz continuous, i.e. for every $r>0$ there exists a constant ${L}_{r}$ such that
\begin{equation}
   \bigl| B(u) - B(\tilde{u}) {\bigr| }_{V^{\prime }} \le {L}_{r} \norm{u - \tilde{u}}{V}{} ,
   \qquad u , \tilde{u } \in V , \quad \norm{u}{V}{}, \norm{\tilde{u}}{V}{} \le r  .
\end{equation}
\end{lemma}

\bigskip
\begin{proof}
This is classical but for completeness we provide the proof. The assertion follows from the following estimates
\begin{eqnarray*}
\bigl| B(u,u) -  B(\tilde{u},\tilde{u}) {\bigr| }_{V^{\prime }}
 & \le &\bigl| B(u,u-\tilde{u})  {\bigr| }_{V^{\prime }}
   + \bigl| B(u-\tilde{u},\tilde{u}) {\bigr| }_{V^{\prime }} \\
&  \le & \norm{B}{}{}  (\norm{u}{V}{} +\norm{\tilde{u}}{V}{} ) \norm{u-\tilde{u}}{V}{}
  \le 2r  \norm{B}{}{}  \cdot \norm{u-\tilde{u}}{V}{} .
\end{eqnarray*}
Thus the Lipschitz condition holds with ${L}_{r}= 2r \norm{B}{}{} $, where $\norm{B}{}{}$ stands for the norm of the bilinear map $B:V\times V \to V^{\prime }$. The proof is thus complete.
\end{proof}

\bigskip
%%%%%%%%%%%%%%%%%%%%%%%%%%%%%%%%%%%%%%%%%%%%%%%%%%%%%%%%%%%%%%%%%%%%%%%%%%%%%%%%%%%%%%%%%%%%%%%%%%
\subsection{Some operators} \label{S:Some_operators}
%%%%%%%%%%%%%%%%%%%%%%%%%%%%%%%%%%%%%%%%%%%%%%%%%%%%%%%%%%%%%%%%%%%%%%%%%%%%%%%%%%%%%%%%%%%%%%%%%

\bigskip  \noindent
Consider the natural embedding $j: V \hookrightarrow H $ and its adjoint ${j}^\ast : H \to V $.
Since the range of $j$ is dense in $H$, the map ${j}^\ast$ is one-to-one.
Let us put
\begin{eqnarray}
 & & D(A) := {j}^\ast(H) \subset V , \nonumber \\
 & & Au := \bigl( {j}^\ast {\bigr) }^{-1} u , \qquad u \in D(A) . \label{E:op_A}
\end{eqnarray}
Notice that for all $u \in D(A)$ and $v \in V $
\begin{equation} \label{E:op_A_coerciv}
     \ilsk{Au}{v}{H} = \ilsk{u}{v}{V}.
\end{equation}
Indeed, this follows immediatelly from the following equalities
\begin{equation*} \label{E:op_A_coerciv_proof}
(Au|v{)}_{H} = (( {j}^{*} {)}^{-1} u |v{)}_{H} = (( {j}^{*} {)}^{-1} u |jv{)}_{H}
        = \ilsk{{j}^{*}( {j}^{*} {)}^{-1}u}{v}{V} = \ilsk{u}{v}{V} .
\end{equation*}
Let
$$
     \acal u :=\dirilsk{u}{\cdot }{} , \qquad u \in V,
$$
where $\dirilsk{\cdot }{\cdot }{}$ is defined by (\ref{E:il_sk_Dir}). Let us notice that
if $u \in V$, then $\acal u \in V^\prime $ and
\begin{equation} \label{E:Acal_V'_norm}
 |\acal u{|}_{V^\prime }\le \norm{u}{}{}.
\end{equation}
Indeed,  from (\ref{E:norm_V}) and the following inequalities
$$
  |\dirilsk{u}{v}{}| \le \norm{u}{}{} \cdot \norm{v}{}{}
   \le \norm{u}{}{} (\norm{v}{}{2} + |v{|}_{H}^{2} {)}^{\frac{1}{2}}
   = \norm{u}{}{} \cdot \norm{v}{V}{}, \quad v \in V,
$$
it follows that $\acal u \in V^{\prime }$ and inequality (\ref{E:Acal_V'_norm}) holds. Denoting by $\dual{\cdot }{\cdot }{}$ the dual pairing between $V$ and $V^{\prime }$, i.e.
$\dual{\cdot }{\cdot }{}:=\ddual{V^\prime }{\cdot }{\cdot }{}{}_{V}$,
 we have the following equality
\begin{equation}  \label{E:Acal_ilsk_Dir}
   \dual{\acal u}{v}{} = \dirilsk{u}{v}{} , \qquad u,v \in V.
\end{equation}

\bigskip
\begin{lemma} \label{L:A_acal_rel}\
\begin{description}
\item[(a)] For any $u \in D(A)$ and $v \in V$:
$$
((A-I)u|v{)}_{H} =  \dirilsk{u}{v}{} = \dual{\acal  u }{v}{} ,
$$
where $I$ stands for the identity operator on $H$.
In particular,
\begin{equation}
         |\acal u{|}_{V^{\prime }}\le |(A-I)u{|}_{H}.
\end{equation}
\item[(b)] $D(A)$ is dense in $H$.
\end{description}
\end{lemma}

%\bigskip
\begin{proof}
To prove assertion (a), let $u\in D(A)$ and $v\in V$. By (\ref{E:op_A_coerciv}), (\ref{E:V_il_sk}) and (\ref{E:Acal_ilsk_Dir}), we have
\begin{eqnarray*}
(Au|v{)}_{H} = \ilsk{u}{v}{V}
= \ilsk{u}{v}{H} + \dirilsk{u}{v}{}
        = \ilsk{Iu}{v}{H} + \dual{ \acal  u}{v}{}.
\end{eqnarray*}
Let us move to the proof of part (b). Since $V$ is dense in $H$, it is sufficient to prove that $D(A)$ is dense in $V$. Let $w\in V$ be an arbitrary element orthogonal to $D(A)$ with respect to the scalar product in $V$. Then
$$
     \ilsk{u}{w}{V} = 0 \qquad \mbox{for} \quad u \in D(A).
$$
On the other hand, by (a) and (\ref{E:V_il_sk}),
$\ilsk{u}{w}{V} = \ilsk{Au}{w}{H}$ for $u \in D(A)$.
Hence $\ilsk{Au}{w}{H}=0$ for $u\in D(A)$. Since $A:D(A) \to H$ is onto, we infer that $w=0$,
which completes the proof.
\end{proof}

\bigskip  \noindent
Let us assume that $s > 1$. It is clear that ${V}_{s}=D(A^s)$ is dense in $V$ and the embedding ${j}_{s}: {V}_{s} \hookrightarrow V$ is continuous. Then by Lemma C.1 in Appendix C,
there exists a Hilbert space $U$ such that $U \subset  {V}_{s}$, $U$ is dense in $ {V}_{s}$ and
\begin{equation} \label{E:U_comp_V_s}
 \mbox{\it the natural embedding ${\iota }_{s}: U \hookrightarrow  {V}_{s}$ is compact \rm }.
\end{equation}
Then we have
\begin{equation} \label{E:embeddings}
  U \underset{{\iota }_{s}}{\hookrightarrow } {V}_{s}
  \underset{{j }_{s}}{\hookrightarrow }V
 \underset{{j}_{}}{\hookrightarrow} H \cong H^{\prime }
\underset{j^{\prime }}{\hookrightarrow } V^{\prime }
\underset{{j}_{s}^{\prime }}{\hookrightarrow } {V}_{s}^{\prime }
  \underset{{\iota }_{s}^{^{\prime }}}{\hookrightarrow }U^{\prime } .
\end{equation}
Since the embedding  ${\iota }_{s}$ is compact, ${\iota }_{s}^{\prime }$ is compact as well.
Consider the composition
$$
   \iota := j \circ {j}_{s} \circ {\iota }_{s} : U \hookrightarrow H
$$
and its adjoint
$$
   {\iota }^\ast := (j \circ {j}_{s} \circ {\iota }_{s} )^\ast
  ={\iota }_{s} ^\ast \circ {j}_{s}^\ast \circ {j}^\ast :  H \to U .
$$
Note that $\iota $ is compact and since the range of $\iota $ is dense in $H$, ${\iota }^\ast : H \to U $ is one-to-one. Let us put
\begin{eqnarray}
& & D(L) := {\iota }^\ast(H) \subset U , \nonumber \\
& & Lu := \bigl( {\iota }^\ast {\bigr) }^{-1} u , \qquad u \in D(L) .  \label{E:op_L}
\end{eqnarray}
It is clear that $L:D(L) \to H $ is onto. Let us also notice that
\begin{equation} \label{E:op_L_ilsk}
    \ilsk{Lu}{w}{H} = \ilsk{u}{w}{U}, \qquad u \in D(L), \quad w \in U .
\end{equation}
Indeed, by (\ref{E:op_L}) we have for all $u \in D(L)$ and $ w \in U$
$$
 \ilsk{Lu}{w}{H}=\ilsk{({\iota }^\ast{)}^{-1}u}{\iota w}{H} =
 \ilsk{{\iota }^\ast({\iota }^\ast{)}^{-1}u}{ w}{U} = \ilsk{u}{w}{U} \, ,
$$
which proves (\ref{E:op_L_ilsk}).
By equality (\ref{E:op_L_ilsk}) and the density of $U$ in $H$, we infer similarly as in the proof of assertion (b) in Lemma \ref{L:A_acal_rel} that $D(L)$ is dense in $H$.

\bigskip  \noindent
Moreover, for $u \in D(L)$,
$$
Lu = \bigl( {\iota }^\ast {\bigr) }^{-1} u
   = \bigl( {\iota }_{s} ^\ast \circ {j}_{s}^\ast \circ {j}^\ast  {\bigr) }^{-1} u
   = \bigl(  {j}^\ast {\bigr) }^{-1} \circ \bigl( {j}_{s}^\ast {\bigr) }^{-1} \circ   \bigl( {\iota }_{s} ^\ast {\bigr) }^{-1} u
   = A \circ \bigl( {j}_{s}^\ast {\bigr) }^{-1} \circ   \bigl( {\iota }_{s} ^\ast {\bigr) }^{-1} u,
$$
where $A$ is defined by (\ref{E:op_A}).

Let us also put
\begin{eqnarray}
& & D({L}_{s}) := {\iota }_{s}^\ast({V}_{s}) \subset U , \nonumber \\
& & {L}_{s}u := \bigl( {\iota }_{s}^\ast {\bigr) }^{-1} u , \qquad u \in D({L}_{s}) \label{E:op_L_s}
\end{eqnarray}
and
\begin{eqnarray}
& & D({A}_{s}) := {j}_{s}^\ast(V) \subset {V}_{s},\nonumber \\
& & {A}_{s}u := \bigl( {j }_{s}^\ast {\bigr) }^{-1} u , \qquad u \in D({A}_{s}) . \label{E:op_A_s}
\end{eqnarray}
The operators ${L}_{s}:D({L}_{s})\to {V}_{s}$ and ${A}_{s}:D({A}_{s})\to V$ are densely defined and onto.
Let us also notice that
\begin{equation}  \label{E:L=AAsLs}
L  = A \circ {A}_{s} \circ {L}_{s} .
\end{equation}

\bigskip  \noindent
Since $L$ is self-adjoint and ${L}^{-1}$ is compact, there exists an orthonormal basis $\{ {e}_{i} {\} }_{i \in \nat }$ of $H$ composed of the eigenvectors of operator $L$. Let ${\lambda }_{i}$ be the eigenvalue corresponding to ${e}_{i}$, i.e.
\begin{equation} \label{E:eigen_L}
    L{e}_{i} = {\lambda }_{i} {e}_{i} , \qquad i \in \nat .
\end{equation}
Notice that ${e}_{i} \in U $, $i \in \nat $, because $D(L) \subset U$.
Let us fix $n \in \nat $ and let
$span \{ {e}_{1},..., {e}_{n}\} $ denote the linear space spanned by the vectors ${e}_{1},...,{e}_{n}$.
Let $\Pn $ be the operator from $U^{\prime }$ to $span \{ {e}_{1},..., {e}_{n}\} $ defined by
\begin{equation} \label{E:tP_n}
  \Pn {u}^\ast := \sum_{i=1}^{n}  \ddual{U^{\prime }}{{u}^{\ast }}{{e}_{i}}{U} {e}_{i}, \qquad {u}^\ast \in U^{\prime }.
\end{equation}
We will consider the restriction  of the operator  $\Pn $ to the space $H$ denoted still by $\Pn $. More precisely,
we have $H \hookrightarrow U^{\prime }$, i.e. every element $u \in H $ induces a functional ${u}^\ast \in U^{\prime }$ by the formula
\[
    \ddual{{U}^{\prime }}{{u}^{\ast }}{v}{U}:= \ilsk{u}{v}{H} , \qquad v \in U.
\]
Thus  the  restriction of $\Pn $ to $H$ is given by
\begin{equation} \label{E:P_n}
  \Pn u = \sum_{i=1}^{n} \ilsk{ u}{ {e}_{i}}{H} {e}_{i}, \qquad u \in H .
\end{equation}
Hence in particular, $\Pn $  is the $\ilsk{\cdot }{\cdot }{H}$-orthogonal projection from $H$
onto $span \{ {e}_{1},..., {e}_{n}\} $. Restrictions of $\Pn $ to other spaces considered in (\ref{E:embeddings}) will also be denoted by $\Pn $.

\bigskip
\begin{lemma}
The following equality holds
\begin{equation}  \label{E:tP_n-P_n}
   \ilsk{\Pn {u}^\ast}{v}{H} = \ddual{{U}^{\prime }}{{u}^{\ast }}{\Pn v}{U} , \qquad {u}^\ast \in U^{\prime }, \quad v \in U.
\end{equation}
\end{lemma}

\bigskip
\begin{proof}
Let us fix ${u}^\ast \in U^{\prime }$ and  $v \in U $. By (\ref{E:tP_n}) and (\ref{E:P_n})
the assertion follows from the following equalities
\begin{eqnarray*}
 \ilsk{\Pn {u}^{\ast} }{v}{H}
& = & \Ilsk{\sum_{i=1}^{n} \ddual{{U}^{\prime }}{{u}^{\ast }}{{e}_{i}}{U} {e}_{i}}{v}{H}
= \sum_{i=1}^{n} \ddual{{U}^{\prime }}{{u}^{\ast }}{{e}_{i}}{U}
 \ilsk{v}{{e}_{i}}{H} \\
& = &
\ddual{{U}^{\prime }}{{u}^{\ast }}{\sum_{i=1}^{n}\ilsk{v}{{e}_{i}}{H}{e}_{i}}{U}
= \ddual{{U}^{\prime }}{{u}^\ast}{\Pn v }{U}.
\end{eqnarray*}
The proof of (\ref{E:tP_n-P_n}) is thus complete.
\end{proof}

\bigskip  \noindent
Let us denote
$$
   {\tilde{e}}_{i}:= \frac{{e}_{i}}{\norm{{e}_{i}}{U}{}} , \qquad i \in \nat .
$$

\bigskip
\begin{lemma}  \label{L:P_n|U}
\begin{itemize}
\item[(a) ] The system $\{ {\tilde{e}}_{i} {\} }_{i \in \nat }$ is the  orthonormal basis in the space $\bigl( U,\ilsk{\cdot }{\cdot }{U}\bigr) $. Moreover,
\begin{equation}  \label{E:lambda_norm_U}
   {\lambda }_{i} = \norm{{e}_{i}}{U}{2} , \qquad i \in \nat .
\end{equation}
\item[(b) ] For every $n \in \nat $ and $u \in U $
\begin{equation} \label{E:P_n|U}
  \Pn u = \sum_{i=1}^{n} \ilsk{u}{{\tilde{e}}_{i}}{U} {\tilde{e}}_{i},
\end{equation}
i.e. the restriction of $\Pn $ to $U$ is the $\ilsk{\cdot }{\cdot }{U}$-orthogonal projection onto $span \{ {e}_{1},...,{e}_{n} \} $.
\item[(c) ] For every $u\in U$ and $s>0$ we have
\begin{itemize}
\item[(i) ]  $\lim_{n \to \infty }  \norm{\Pn u -u }{U}{} =0$,
\item[(ii) ]  $\lim_{n \to \infty }  \norm{\Pn u -u }{{V}_{s}}{} =0$,
\item[(iii) ]  $\lim_{n \to \infty }  \norm{\Pn u -u }{V}{} =0$.
\end{itemize}
\end{itemize}
\end{lemma}

\bigskip
\begin{proof}
By (\ref{E:op_L_ilsk}) and (\ref{E:eigen_L})
\begin{equation} \label{E:ei_ej_U}
   \ilsk{{e}_{i}}{{e}_{j}}{U} = \ilsk{L{e}_{i}}{{e}_{j}}{H} = {\lambda }_{i} \ilsk{{e}_{i}}{{e}_{j}}{H}
    = {\lambda }_{i} {\delta }_{ij}, \qquad i,j \in \nat ,
\end{equation}
where ${\delta }_{ij}:= 1 $ if $  i=j$ and ${\delta }_{ij}:= 0 $ if $ i \ne j $.
In particular, equality (\ref{E:lambda_norm_U}) holds.
By (\ref{E:ei_ej_U}) and (\ref{E:lambda_norm_U}) the system $\{ {\tilde{e}}_{i} { \} }_{i \in \nat }$ is orthonormal in $U$. To prove that it is a basis in $U$ let $h\in U$ be an arbitrary vector $\ilsk{\cdot }{\cdot }{U}$-orthogonal to each ${e}_{i}$, $i \in \nat $. By (\ref{E:op_L_ilsk}) and (\ref{E:eigen_L}) we obtain
the following equalities
$$
  0= \ilsk{h}{{e}_{i}}{U} = \ilsk{h}{L{e}_{i}}{H} = {\lambda }_{i} \ilsk{h}{{e}_{i}}{H} , \qquad i \in \nat .
$$
Since ${\lambda }_{i} >0 $,
\begin{equation}
   \ilsk{h}{{e}_{i}}{H} =0 , \qquad i \in \nat .
\end{equation}
Since $\{ {e}_{i} {\} }_{i \in \nat }$ is the $\ilsk{\cdot }{\cdot }{H}$-orthonormal basis in $H$, we infer that $h=0$. This means that the system $\{ {\tilde{e}}_{i} { \} }_{i \in \nat }$ is the $\ilsk{\cdot }{\cdot }{U}$-orthonormal basis in $U$.

\bigskip  \noindent
To prove assertion (b) let us fix $u \in U$. By (\ref{E:lambda_norm_U}), (\ref{E:eigen_L}) and (\ref{E:op_L_ilsk}), we have
\begin{equation} \label{E:P_n|U_ilsk}
   \ilsk{u}{{e}_{i}}{H} {e}_{i}
 = \Ilsk{u}{\frac{{\lambda }_{i}{e}_{i}}{\norm{{e}_{i}}{U}{}}}{H} \frac{{e}_{i}}{\norm{{e}_{i}}{U}{}}
 = \Ilsk{u}{L\frac{{e}_{i}}{\norm{{e}_{i}}{U}{}}}{H} \frac{{e}_{i}}{\norm{{e}_{i}}{U}{}}
 = \ilsk{u}{{\tilde{e}}_{i}}{U} {\tilde{e}}_{i} , \quad i \in \nat.
\end{equation}
By (\ref{E:P_n}) and (\ref{E:P_n|U_ilsk})
$$
  \Pn u = \sum_{i=1}^{n}  \ilsk{u}{{e}_{i}}{H} {e}_{i}
    = \sum_{i=1}^{n} \ilsk{u}{{\tilde{e}}_{i}}{U} {\tilde{e}}_{i},
 \qquad n \in \nat .
$$
Since  $\{ {\tilde{e}}_{i} {\} }_{i \in \nat }$ is the $\ilsk{\cdot }{\cdot }{U}$-orthonormal basis in $U$, the restriction of $\Pn $ to the space $U$ is the $\ilsk{\cdot }{\cdot }{U}$-orthogonal projection onto $span \{ {e}_{1},...,{e}_{n} \} $.

\bigskip  \noindent
Assertion (c) follows immediately from (b) and the continuity of the embeddings
$$
     U \hookrightarrow {V}_{s} \hookrightarrow V .
$$
This completes the proof of the Lemma.
\end{proof}

\bigskip
%%%%%%%%%%%%%%%%%%%%%%%%%%%%%%%%%%%%%%%%%%%%%%%%%%%%%%%%%%%%%%%%%%%%%%%%%%%%%%%%%%%%%%%%%%%
\section{Compactness results. }  \label{S:Compactness}
%%%%%%%%%%%%%%%%%%%%%%%%%%%%%%%%%%%%%%%%%%%%%%%%%%%%%%%%%%%%%%%%%%%%%%%%%%%%%%%%%%%%%%%%%%%%%

%\bigskip
%%%%%%%%%%%%%%%%%%%%%%%%%%%%%%%%%%%%%%%%%%%%%%%%%%%%%%%%%%%%%%%%%%%%%%%%%%%%%%%%%%%%%%%%%%%%%
\subsection{Deterministic compactness criteria}
%%%%%%%%%%%%%%%%%%%%%%%%%%%%%%%%%%%%%%%%%%%%%%%%%%%%%%%%%%%%%%%%%%%%%%%%%%%%%%%%%%%%%%%%%%%%%%%

%\bigskip
\noindent
Let us choose $s > \frac{d}{2} +1 $. We have the following sequence of Hilbert spaces
\begin{equation}
  U \underset{{\iota }_{s}}{\hookrightarrow } {V}_{s}
  \hookrightarrow V \hookrightarrow H \cong H^{\prime } \hookrightarrow V^{\prime } \hookrightarrow {V}_{s}^{\prime }
  \underset{{\iota }_{s}^{\prime }}{\hookrightarrow } U^{\prime } ,
\end{equation}
with the embedding ${\iota }_{s}$ being compact.
In particular,
\begin{equation}
     H \cong H^{\prime } \underset{{\iota }^{\prime }}{\hookrightarrow } {U}^{\prime }
\end{equation}
and ${\iota }^{\prime }$ is compact as well.
Let $\bigl( {\ocal }_{R} {\bigr) }_{R \in \nat } $ be a sequence of bounded open subsets of $\ocal $ with
regular boundaries $\partial {\ocal }_{R}$  such that
${\ocal }_{R} \subset {\ocal }_{R+1}$ and $\bigcup_{R=1}^{\infty } {\ocal }_{R} = \ocal $.
We will consider the following spaces of restrictions of functions defined on $\ocal $ to subsets ${\ocal }_{R}$, i.e.
\begin{equation}  \label{E:HR_VR}
 {H}_{{\ocal }_{R}} := \{ {u}_{|{\ocal }_{R}} ; \, \, \,  u \in H  \}
 \qquad {V}_{{\ocal }_{R}} := \{ {v}_{|{\ocal }_{R}} ; \, \, \,  v \in V  \}
\end{equation}
with appropriate scalar products and norms, i.e.
\begin{eqnarray*}
 & & \ilsk{u}{v}{{H}_{{\ocal }_{R}}} := \int_{{\ocal }_{R}} u v \, dx , \qquad  u,v \in {H}_{{\ocal }_{R}}, \\
& &\ilsk{u}{v}{{V}_{{\ocal }_{R}}} :=\int_{{\ocal }_{R}}uv\, dx+\int_{{\ocal }_{R}}\nabla u\nabla v \, dx,
 \qquad  u,v \in {V}_{{\ocal }_{R}}
\end{eqnarray*}
and $|u{|}_{{H}_{{\ocal }_{R}}}^{2}:=\ilsk{u}{u}{{H}_{{\ocal }_{R}}}$ for
$u \in {H}_{{\ocal }_{R}}$ and $\norm{u}{{V}_{{\ocal }_{R}}}{2}:=\ilsk{u}{u}{{V}_{{\ocal }_{R}}}$ for
$u \in {V}_{{\ocal }_{R}}$.
The symbols  ${H}_{{\ocal }_{R}}^{\prime }$ and ${V}_{{\ocal }_{R}}^{\prime }$ will stand for the corresponding dual spaces.

\bigskip \noindent
Since the sets ${\ocal }_{R}$ are bounded,
\begin{equation} \label{E:comp_VR_HR}
\mbox{\it the embeddings ${V}_{{\ocal }_{R}} \hookrightarrow  {H}_{{\ocal }_{R}}$ are compact. \rm }
\end{equation}

\bigskip  \noindent
Let us consider the following three functional spaces, analogous to those considered in \cite{Mikulevicius+Rozovskii_2005}:
\begin{eqnarray}
\ccal ([0,T],U^{\prime })&:=& \mbox{ the space of  continuous functions } u:[0,T] \to U^{\prime }
                          \mbox{ with the topology } {\tcal }_{1} \nonumber \\
                &         & \mbox{ induced by the norm }
                          \norm{u}{\ccal ([0,T],U^{\prime })}{}:=\sup_{t\in[0,T]} |u(t){|}_{U^{\prime }}, \nonumber \\
{L}_{w}^{2}(0,T;V)&:=& \mbox{ the space } {L}^{2} (0,T;V) \mbox{ with the weak topology }
                      {\tcal}_{2},  \nonumber  \\
{L}^{2}(0,T;{H}_{loc})&:=& \mbox{ the space of measurable functions } u:[0,T] \to H
                          \mbox{ such that for all } R \in \nat
                       \nonumber \\
  &   & \, \, {q}_{T,R}^{}(u):=\norm{u}{{L}^{2}(0,T;{H}_{{\ocal }_{R}})}{} :=
     \biggl( \int_{0}^{T}\int_{{\ocal }_{R}}|u(t,x){|}^{2}\, dxdt {\biggr) }^{\frac{1}{2}} <\infty  ,
      \label{E:seminorms} \\
  &   &  \mbox{ with the topology $ {\tcal }_{3} $ generated by the seminorms }
         ({q}_{T,R}{)}_{R\in \nat } . \nonumber
\end{eqnarray}

\bigskip  \noindent
The following lemma was inspired by the classical Dubinsky Theorem, see for instance the monograph \cite{Vishik_Fursikov_88},  and the compactness result due to Mikulevicius and Rozovskii  contained in \cite[Lemma 2.7]{Mikulevicius+Rozovskii_2005}.
The proof will be a certain modification of the proof of  \cite[Theorem IV.4.1]{Vishik_Fursikov_88},
see also \cite{Lions_69}.

\bigskip  \noindent
\begin{lemma} \rm \label{L:Dubinsky_ext} \it
Let
\begin{equation} \label{E:tilde_Z}
  \tilde{\zcal }:=\ccal ([0,T];U^{\prime })\cap {L}_{w}^{2}(0,T;V)\cap {L}^{2}(0,T;{H}_{loc})
\end{equation}
and let $\tilde{\tcal }$ be the supremum of the corresponding topologies. Then a set $\kcal \subset
\tilde{\zcal }$ is $\tilde{\tcal }$-relatively compact if  the following two conditions hold:
\begin{description}
\item[(i) ] $ \sup_{u\in \kcal } \int_{0}^{T} \norm{u(s)}{V}{2} \, ds < \infty  $,
    i.e. $\kcal $ is bounded in ${L}^{2}(0,T;V)$,
\item[(ii)] $ \lim{}_{\delta \to 0 } \sup_{u\in \kcal }
  \sup_{\overset{s,t \in[0,T]}{|t-s|\le \delta }} {|u(t)-u(s)|}_{U^{\prime }} =0 $.
\end{description}
\end{lemma}

\bigskip
\begin{proof}
We can  assume that $\kcal $ is closed in $\tilde{\tcal }$.
Because of the assumption (i), the restriction to $\kcal $ of the weak topology
in  ${L}_{w}^{2}(0,T;V)$ is metrizable. Since the topology in ${L}^{2}(0,T;{H}_{loc})$ is defined by the countable family of seminorms (\ref{E:seminorms}), this space is also metrizable.
Thus the compactness of a subset of $\tilde{\zcal }$ is equivalent to its sequential compactness.
Let $(\un )$ be a sequence in $\kcal $. By the Banach-Alaoglu Theorem condition (i) yields that $\kcal $ is compact in $ {L}_{w}^{2}(0,T;V) $.

\bigskip  \noindent
Arguing analogously to the proof of the classical  Arzel\'{a}-Ascoli Theorem, we will prove that $(\un )$ is compact in  $ \ccal ([0,T]; U^{\prime })$.

\bigskip  \noindent
Let us consider the following  set
\begin{equation} \label{E:I_infty}
    {I}_{\infty }:=\{ t \in [0,T]: \, \, \lim_{n \to \infty }\norm{\un (t)}{V}{}=\infty \} .
\end{equation}
The set ${I}_{\infty }$ is Lebesgue measurable because
$$
   {I}_{\infty }= \bigcap_{n=1}^{\infty } \bigcup_{k=n}^{\infty } \bigcap_{l=k}^{\infty }
    \{ \norm{ {u}_{l} (t)}{V}{2} \ge n  \} .
$$
Moreover, its measure is equal to zero.
Indeed, let us notice that otherwise
$$
  \int_{0}^{T} \norm{\un (t)}{V}{2} \, dt \ge \int_{{I}_{\infty }} \norm{\un (t)}{V}{2} \, dt
  \ge n \, \mbox{\rm{Leb}} ({I}_{\infty })
  \to \infty  \quad \mbox{as } n \to \infty ,
$$
which contradicts assumption (i).

\bigskip \noindent
By (\ref{E:I_infty}), for every $t \in [0,T] \setminus {I}_{\infty }$, the sequence $(\un (t){)}_{n \in \nat }$ contains a subsequence bounded in $V$. Furthermore, since the embedding $V \hookrightarrow U^{\prime } $ is compact, this subsequence contains a subsequence convergent in $U^{\prime }$.

\bigskip  \noindent
Let $\{ {t}_{i} {\} }_{i \in \nat } \subset [0,T] \setminus {I}_{\infty }$ be a dense subset of $[0,T]$.
Using the diagonal method we can choose a subsequence, still denoted by $(\un {)}_{n \in \nat }$
such that
\begin{equation}
     (\un ({t}_{i}){)}_{n \in \nat } \mbox{ is convergent in $U^{\prime }$ for each $i \in \nat $}.
\end{equation}
We will prove that the sequence $(\un )$ is Cauchy in $\ccal ([0,T];U^{\prime })$. To this end let us fix $\eps >0$.
By (ii) there exists $\delta > 0$ such that
$$
\sup_{v \in \kcal } \sup_{\overset{s,t\in [0,T],}{ |s-t|\le \delta }} |v(t)-v(s){|}_{U^{\prime }} < \frac{\eps }{3} .
$$
Let us fix $t \in [0,T]$. There exists $i \in \nat $ such that $|t-{t}_{i}| \le \delta $.
Then for sufficiently large $m,n \in \nat $ we have the following estimates
$$
   |{u}_{n}(t)-{u}_{m}(t) {|}_{U^{\prime }}
  \le   |{u}_{n}(t)-{u}_{n}({t}_{i}) {|}_{U^{\prime }}
      + |{u}_{n}({t}_{i})-{u}_{m}({t}_{i}) {|}_{U^{\prime }}
       + |{u}_{m}({t}_{i})-{u}_{m}(t) {|}_{U^{\prime }} \le \eps .
$$
Since $t \in [0,T]$ was chosen in an arbitrary way we infer that
$$
   \sup_{t\in [0,T]} |{u}_{n}(t)-{u}_{m}(t) {|}_{U^{\prime }} \le \eps
$$
which means that the sequence $({u}_{n})$ is Cauchy in $\ccal ([0,T];U^{\prime })$.

\bigskip \noindent
Therefore there exists a subsequence $(\unk ) \subset (\un )$ and
$u \in {L}_{}^{2}(0,T;V) \cap \ccal ([0,T]; U^{\prime }) $ such that
$$
  \unk \to u \quad \mbox{in} \quad {L}_{w}^{2}(0,T;V) \cap \ccal ([0,T]; U^{\prime })
\quad \mbox{as } \quad \kinf .
$$
In particular, for all $p \in [1,\infty )$
$$
  \unk \to u \quad \mbox{in} \quad {L}^{p}(0,T;U^{\prime })
\quad \mbox{as } \quad \kinf .
$$
\noindent
We claim that
\begin{equation} \label{E:unk_u_L^2_H_loc}
  \unk \to u \quad \mbox{in} \quad  {L}^{2}(0,T; {H}_{loc})
\quad \mbox{as } \quad \kinf .
\end{equation}
In order to prove (\ref{E:unk_u_L^2_H_loc}) let us fix $R >0$.
Since, by (\ref{E:comp_VR_HR}) the embedding ${V}_{{\ocal }_{R}} \hookrightarrow {H}_{{\ocal }_{R}}$ is compact and the embeddings ${H}_{{\ocal }_{R}} \hookrightarrow H^{\prime } \hookrightarrow U^{\prime }$ are continuous,  by the Lions Lemma, \cite{Lions_69}, for every $\eps >0 $ there exists a constant $C = {C}_{\eps , R}>0 $ such that
$$
  {|u|}_{{H}_{{\ocal }_{R}}}^{2} \le \eps \norm{u}{{V}_{{\ocal }_{R}}}{2}  + C {|u|}_{U^{\prime }}^{2} ,
\qquad u \in V .
$$
Thus for almost all $s \in [0,T]$
$$
  {|\unk (s) - u(s )|}_{{H}_{{\ocal }_{R}}}^{2} \le \eps \norm{\unk (s) - u(s)}{{V}_{{\ocal }_{R}}}{2}
 + C {|\unk (s) - u(s)|}_{U^{\prime }}^{2} ,
\qquad k \in \nat ,
$$
and so
$$
  {\| \unk  - u \| }_{{L}^{2}(0,T;{H}_{{\ocal }_{R}})}^{2}
  \le \eps \norm{\unk  - u }{{L}^{2}(0,T;{V}_{{\ocal }_{R}})}{2}
 + C {\| \unk  - u \| }_{{L}^{2}(0,T;U^{\prime })}^{2} ,
\qquad k \in \nat .
$$
Passing to the upper limit as $\kinf $ in the above inequality and using the estimate
$$
 \norm{\unk  - u }{{L}^{2}(0,T;{V}_{{\ocal }_{R}})}{2}
   \le 2 \bigl( \norm{\unk   }{{L}^{2}(0,T;{V}_{{\ocal }_{R}})}{2}
    + \norm{ u }{{L}^{2}(0,T;{V}_{{\ocal }_{R}})}{2} \bigr)
 \le 4 {c}_{2} ,
$$
where ${c}_{2}= \sup_{u \in \kcal } \norm{u}{{L}^{2}(0,T;V)}{2}$, we infer that
$$
  \limsup_{\kinf }  {\| \unk  - u \| }_{{L}^{2}(0,T;{H}_{{\ocal }_{R}})}^{2} \le 4 {c}_{2} \eps  ,
$$
By the arbitrariness of $\eps $,
$$
  \lim_{\kinf }  {\| \unk  - u \| }_{{L}^{2}(0,T;{H}_{{\ocal }_{R}})}^{2} =0 .
$$
The proof of the lemma is thus complete.
\end{proof}

\bigskip \noindent
Let ${H}_{w}$ denote the Hilbert space $H$ endowed with the weak topology.
Let
\begin{eqnarray}
& \ccal ([0,T];{H}_{w}) : = & \mbox{the space of weakly continuous functions }
                          u : [0,T] \to H \mbox{ endowed with }\nonumber \\
  &                     & \mbox{the weakest topology ${\tcal}_{4}$ such that for all
                          $h \in H $   the  mappings } \nonumber \\
  &                     &  \ccal ([0,T];{H}_{w}) \ni u \mapsto \ilsk{u(\cdot )}{h}{H}
                           \in \ccal  ([0,T];\rzecz )
                        \mbox{ are continuous. }  \label{E:C([0,T];H_w)}
\end{eqnarray}
In particular,
$\un \to u $ in $\ccal ([0,T];{H}_{w}) $ iff  for all $ h \in H $:
$$
\lim_{n\to \infty } \sup_{t \in [0,T]} \bigl| (\un (t)-u(t)|h{)}_{H} \bigr| =0 .
$$
Let us consider the ball
$$
    \ball := \{ x \in H : \, \, \, |x{|}_{H} \le r \} .
$$
Let ${\ball }_{w}$ denote the ball $\ball $ endowed with the weak topology.
It is well known that ${\ball }_{w}$  is metrizable, see \cite{Brezis}. Let $q$ denote the metric compatible with the weak topology on $\ball $. Let us consider the following subspace of the space
$\ccal ([0,T];{H}_{w})$
\begin{equation} \label{E:C([0,T];B_w)}
\ccal ([0,T];{\ball }_{w})  =  \{ u \in \ccal ([0,T];{H}_{w}): \, \,
                        \sup_{t \in [0,T]} |u(t){|}_{H} \le r  \} .
\end{equation}
The space $\ccal ([0,T];{\ball }_{w})$ is metrizable with
\begin{equation} \label{E:C([0,T];B_w)_metric}
     \varrho (u,v) = \sup_{t\in [0,T]} q(u(t),v(t)).
\end{equation}
Since by the Banach-Alaoglu Theorem  ${\ball }_{w}$ is  compact,
$(\ccal ([0,T];{\ball }_{w}),\varrho )$ is a complete metric space.
The following lemma says that any  sequence $(\un ) \subset \ccal ([0,T];\ball )$  convergent in $\ccal ([0,T];U^{\prime })$ is also convergent in the space $\ccal ([0,T];{\ball }_{w})$.

%\bigskip
\begin{lemma} \label{L:C(0,T,{H}_{w})_conv} \rm (see Lemma 2.1 in \cite{Brz+Motyl_2010}) \it
Let ${u}_{n}:[0,T] \to H $, $n \in \nat $ be functions such that
\begin{description}
\item[(i)] $\sup_{n \in \nat } \sup_{s \in [0,T]} {|\un (s)|}_{H} \le r  $,
\item[(ii)] $\un \to u $ in $\ccal ([0,T];U^{\prime })$.
\end{description}
Then  $u, \un \in \ccal ([0,T];{\ball }_{w}) $ and $\un \to u$ in
$\ccal ([0,T];{\ball }_{w})$ as $n \to \infty $.
\end{lemma}

\bigskip \noindent
The proof of Lemma \ref{L:C(0,T,{H}_{w})_conv} is the same as the proof of Lemma 2.1 in \cite{Brz+Motyl_2010}.

\bigskip  \noindent
The following lemma is essentially due to Mikulevicius and Rozovskii, see \cite{Mikulevicius+Rozovskii_2005}.
We prove a version of it which we will use in the forthcoming tightness criterion for the set of measures induced on the space $\zcal $, see Corollary \ref{C:tigthness_criterion}. The main difference in comparison to \cite{Mikulevicius+Rozovskii_2005} is that instead of the whole space $\rd $ we consider any open subset $\ocal $ and instead of the Fr\'{e}chet space ${H}^{-{k}_{0}}_{loc}(\rd )$ we take the  Hilbert space $U^{\prime }$.

\bigskip  \noindent
\begin{lemma} \rm \label{L:comp}
(see Lemma 2.7 in \cite{Mikulevicius+Rozovskii_2005}) \it
Let
\begin{equation} \label{E:Z}
  \zcal :=\ccal ([0,T];U^{\prime })\cap {L}_{w}^{2}(0,T;V)\cap {L}^{2}(0,T;{H}_{loc}) \cap \ccal ([0,T];{H}_{w})
\end{equation}
and let $\tcal $ be  the supremum of the corresponding topologies. Then a set $\kcal \subset \zcal $ is $\tcal $-relatively compact if  the following three conditions hold:
\begin{description}
\item[(a)] $ \sup_{u\in \kcal } \sup_{s \in[0,T]} {|u(s)|}_{H}^{} < \infty  $,
\item[(b)] $ \sup_{u\in \kcal } \int_{0}^{T} \norm{u(s)}{}{2} \, ds < \infty  $,
\item[(c)] $ \lim{}_{\delta \to 0 } \sup_{u\in \kcal }
  \sup_{\overset{s,t \in[0,T]}{|t-s|\le \delta }} {|u(t)-u(s)|}_{U^{\prime }} =0 $.
\end{description}
\end{lemma}

\bigskip
\begin{proof}
Let us notice that $\zcal = \tilde{\zcal }\cap \ccal ([O,T],{H}_{w})$, where $\tilde{\zcal }$ is defined by (\ref{E:tilde_Z}). Without loss of generality we can assume that $\kcal $ is a closed subset of $\zcal $.
Because of the assumption (a) we may consider the metric subspace
$\ccal ([0,T];{\ball }_{w}) \subset \ccal ([0,T];{H}_{w})$ defined by (\ref{E:C([0,T];B_w)})
and (\ref{E:C([0,T];B_w)_metric}). From assumptions (a), (b) and  definition (\ref{E:norm_V}) of the norm in $V$, it follows that
the set $\kcal $ is bounded in ${L}^{2}(0,T;V)$.
Therefore the restrictions to $\kcal $ of the four topologies considered in $\zcal $ are metrizable. Let $(\un )$ be a sequence in $\zcal $. By Lemma \ref{L:Dubinsky_ext} the boundedness of the set $\kcal $ in ${L}^{2}(0,T;V)$ and assumption (c) imply that $\kcal $ is compact in $\tilde{\zcal }$. Hence in particular, there exists a subsequence, still denoted by $(\un )$, convergent in $\ccal ([0,T];U^{\prime })$.  Thus by Lemma \ref{L:C(0,T,{H}_{w})_conv} the sequence $(\un )$ is convergent in $\ccal ([0,T]; {\ball }_{w})$ as well.
The proof of the lemma is thus complete.
\end{proof}

%%%%%%%%%%%%%%%%%%%%%%%%%%%%%%%%%%%%%%%%%%%%%%%%%%%%%%%%%%%%%%%%%%%%%%%%%%%%%%%%%%%%%%%%%%%%%%%%%%%
\subsection{Tightness criterion}
%%%%%%%%%%%%%%%%%%%%%%%%%%%%%%%%%%%%%%%%%%%%%%%%%%%%%%%%%%%%%%%%%%%%%%%%%%%%%%%%%%%%%%%%%%%%%%%%%%%

%\bigskip
\noindent
Let $(\smath ,\varrho )$ be a separable and complete metric space.

%\bigskip
\begin{definition}  \rm
Let $u \in \ccal ([0,T],\smath )$.
The modulus of continuity of $u$ on $[0,T]$ is defined by
$$
     \modulus (u,\delta ):= \sup_{s,t \in [0,T] , |t-s|\le \delta } \varrho (u(t),u(s)),
 \qquad \delta >0 .
$$
\end{definition}

%\bigskip
\noindent
Let $(\Omega , \fcal ,\p , )$ be a probability space with filtration $\mathbb{F}:=({\fcal }_{t}{)}_{t \in [0,T]}$
satisfying the usual conditions, see \cite{Metivier_82},
and let $(\Xn {)}_{n \in \nat }$ be a sequence of continuous $\mathbb{F}$-adapted  $\smath $-valued processes.

\begin{definition}  \rm
We say that the sequence $(\Xn )$ of $\smath $-valued random variables
 satisfies  condition [$\mathbf{\tilde{T}}$] iff $  \,  \forall \, \eps >0 \quad \forall \, \eta >0 \quad \exists \, \delta  >0 $:
\begin{equation} \label{E:cond_modulus}
    \sup_{n \in \nat } \, \p \bigl\{  \modulus (\Xn ,\delta )  > \eta \bigr\}  \le \eps .
\end{equation}
\end{definition}
%%%%%%%%%%%%%%%%%%%%%

%%%%%%%%%%%%%%%%%%%
\begin{lemma}  \label{L:modulus_conv}
Assume that $(\Xn )$ satisfies  condition $\mathbf{[\tilde{T}]}$.
Let ${\p }_{n}$ be the  law of $\Xn $ on $\ccal ([0,T], \smath )$, $n \in \nat $.
Then for every $\eps >0 $ there exists a subset
$ {A}_{\eps } \subset \ccal ([0,T], \smath ) $
such that
$$ \sup_{n\in \nat } {\p }_{n} ({A}_{\eps }) \ge 1 - \eps $$
and
\begin{equation} \label{E:modulus_conv}
   \lim_{\delta \to 0 }  \sup_{u \in {A}_{\eps } }  \modulus (u,\delta ) =0.
\end{equation}
\end{lemma}
%%%%%%%%%%%%%%%%

\bigskip  \noindent
Now, we recall the Aldous condition which is connected with condition $\mathbf{[\tilde{T}]}$,
see \cite{Metivier_88} and \cite{Aldous}. This condition allows to investigate the modulus of continuity
for the sequence of stochastic processes by means of stopped processes.

%\bigskip
\begin{definition}  \rm
A sequence $({X}_{n}{)}_{n\in \nat }$  satisfies  condition \rm \bf [A] \rm
iff
 $ \,   \forall \, \eps >0 \quad \forall \, \eta >0 \quad \exists \, \delta >0
$ such that for every sequence $({{\tau}_{n} } {)}_{n \in \nat }$ of $\mathbb{F}$-stopping times with
${\tau }_{n}\le T$ one has
$$
    \sup_{n \in \nat} \, \sup_{0 \le \theta \le \delta }  \p \bigl\{
   \varrho \bigl( {X}_{n} ({\tau }_{n} +\theta ),{X}_{n} ( {\tau }_{n}  ) \bigr) \ge \eta \bigr\}  \le \eps .
$$
\end{definition}

%\bigskip
\begin{lemma} \label{L:Aldous_equiv}
(See \cite{Metivier_88}, Th. 3.2 p. 29)
Conditions \rm [\bf A \rm ] \it and $\mathbf{[\tilde{T}]}$ are  equivalent.
\end{lemma}

\bigskip  \noindent
Using the deterministic compactness criterion formulated in Lemma  \ref{L:comp} we obtain the following corollary which we will use to prove tightness of the laws defined by the Galerkin approximations.
Let us first recall that $U$ is a Hilbert space such that
$$
    U \hookrightarrow  V \hookrightarrow  H
$$
and the embedding $U\hookrightarrow V$ is dense and  compact. Moreover, we consider the following space
$$
  \zcal :=\ccal ([0,T];U^{\prime })\cap {L}_{w}^{2}(0,T;V)\cap {L}^{2}(0,T;{H}_{loc}) \cap \ccal ([0,T];{H}_{w}) ,
$$
equipped with the topology $\tcal $, see (\ref{E:Z}).

\bigskip
\begin{cor} \bf (tightness criterion) \it \label{C:tigthness_criterion}
Let $(\Xn {)}_{n \in \nat }$ be a sequence of continuous $\mathbb{F}$-adapted $U^{\prime }$-valued processes
 such that
\begin{description}
\item[(a)] there exists a positive constant ${C}_{1}$ such that
$$
         \sup_{n\in \nat} \e \bigl[ \sup_{s \in [0,T]} | \Xn (s) {|}_{H}^{2}  \bigr]  \le {C}_{1} .
$$
\item[(b)] there exists a positive constant ${C}_{2}$ such that
$$
          \sup_{n\in \nat} \e \biggl[  \int_{0}^{T} \norm{\Xn (s)}{}{2} \, ds    \biggr]  \le {C}_{2} .
$$
\item[(c)]  $(\Xn {)}_{n \in \nat }$ satisfies the Aldous condition \bf [A] \rm in $U^{\prime }$.
\end{description}
Let ${\tilde{\p }}_{n}$ be the law of $\Xn $ on $\zcal $.
Then for every $\eps >0 $ there exists a compact subset ${K}_{\eps }$ of $\zcal $ such that
$$
  \sup_{n\in \nat} {\tilde{\p }}_{n}({K}_{\eps })\ge 1-\eps .
$$
\end{cor}

\begin{proof}
Let $\eps >0$. By the  Chebyshev inequality and (a), we infer that for any $n \in \nat $ and any $r>0$
$$
  \p \biggl( \sup_{s \in [0,T]} |\Xn (s){|}_{H}^{2} > r  \biggr)
  \le \frac{ \e \bigl[ \sup_{s \in [0,T]} |\Xn (s) {|}_{H}^{2} \bigr]}{r}
  \le \frac{{C}_{1}}{r}.
$$
Let ${R}_{1}$ be such that $\frac{{C}_{1}}{{R}_{1}} \le \frac{\eps }{3}$. Then
$$
  \sup_{n\in \nat } \p \biggl( \sup_{s \in [0,T]}|\Xn (s){|}_{H}^{2} > {R}_{1} \biggr) \le \frac{\eps }{3}.
$$
Let ${B}_{1}:=\bigl\{ u \in \zcal :\, \, \sup_{s \in [0,T]}|u(s) {|}_{H}^{2} \le {R}_{1} \bigr\} $.

\bigskip  \noindent
By the  Chebyshev inequality and  (b), we infer that for any $n \in \nat $ and any $r>0$
$$
 \p \bigl( \norm{\Xn }{{L}^{2}(0,T;V)}{} > r  \bigr)
  \le \frac{\e \bigl[ \norm{\Xn }{{L}^{2}(0,T;V)}{2} \bigr]  }{{r}^{2}}
  \le \frac{{C}_{2}}{{r}^{2}}.
$$
Let ${R}_{2}$ be such that $\frac{{C}_{2}}{{R}_{2}^{2}} \le \frac{\eps }{3}$. Then
$$
  \sup_{n\in \nat }  \p \bigl( \norm{\Xn }{{L}^{2}(0,T;V)}{} > {R}_{2}  \bigr) \le \frac{\eps }{3}.
$$
Let ${B}_{2} := \bigl\{ u \in \zcal : \, \, \norm{u }{{L}^{2}(0,T;V)}{} \le {R}_{2} \bigr\} $.

\bigskip  \noindent
By Lemmas \ref{L:Aldous_equiv} and \ref{L:modulus_conv} there exists a subset
${A}_{\frac{\eps }{3}} \subset \ccal ([0,T], U^{\prime }) $ such that
${\tilde{\p }}_{n} \bigl( {A}_{\frac{\eps }{3}}\bigr) \ge 1 - \frac{\eps }{3}$ and
$$
   \lim_{\delta \to 0 }  \sup_{u \in {A}_{\frac{\eps }{3}}}
\sup_{\underset{|t-s| \le \delta }{s,t \in [0,T]}}  |u(t) - u(s){|}_{U^{\prime }} =0 .
$$
It is sufficient to define ${K}_{\eps } $ as the closure  of the set
${B}_{1} \cap {B}_{2} \cap {A}_{\frac{\eps }{3}}$ in $\zcal $.
By Lemma  \ref{L:comp}, ${K}_{\eps }$ is compact in $\zcal $. The proof is thus complete.
\end{proof}

\bigskip \noindent
\bf Diggresion. \rm If we omit assumption (a) in Corollary \ref{C:tigthness_criterion}, then
using Lemma \ref{L:Dubinsky_ext}, we obtain the following tightness criterion
in the space $(\tilde{\zcal }, \tilde{\tcal } )$ defined by (\ref{E:tilde_Z}).

\bigskip
\begin{cor} \bf (tightness criterion in $\tilde{\zcal }$) \it \label{C:tigthness_criterion_tilde_Z}
Let $(\Xn {)}_{n \in \nat }$ be a sequence of continuous $\mathbb{F}$-adapted $U^{\prime }$-valued processes
 such that
\begin{description}
\item[(i) ] there exists a positive constant $C$ such that
$$
           \sup_{n \in \nat} \e \biggl[  \int_{0}^{T} \norm{\Xn (s)}{V}{2} \, ds    \biggr]  \le C .
$$
\item[(ii)]  $(\Xn {)}_{n \in \nat }$ satisfies the Aldous condition \bf [A] \rm in $U^{\prime }$.
\end{description}
Let ${\tilde{\p }}_{n}$ be the law of $\Xn $ on $\tilde{\zcal }$.
Then for every $\eps >0 $ there exists a compact subset ${K}_{\eps }$ of $\tilde{\zcal }$ such that
$$
  \sup_{n\in \nat} {\tilde{\p }}_{n}({K}_{\eps })\ge 1-\eps .
$$
\end{cor}

\bigskip
%%%%%%%%%%%%%%%%%%%%%%%%%%%%%%%%%%%%%%%%%%%%%%%%%%%%%%%%%%%%%%%%%%%%%%%%%%%%%%%%%%%%%%%%%%%%%%%%%%%%%
\subsection{The Skorokhod Theorem}
%%%%%%%%%%%%%%%%%%%%%%%%%%%%%%%%%%%%%%%%%%%%%%%%%%%%%%%%%%%%%%%%%%%%%%%%%%%%%%%%%%%%%%%%%%%%%%%%%%%%%%

\bigskip \noindent
We will use the following Jakubowski's version of the Skorokhod Theorem in the form given by
Brze\'{z}niak and Ondrej\'{a}t  \cite{Brz+Ondrejat_2011}, see also \cite{Jakubowski_1998}.

\begin{theorem} \rm (Theorem A.1 in \cite{Brz+Ondrejat_2011}) \it
\label{T:Jakubowski_Skorokhod_Ondrejat}
Let $\xcal $ be a topological space such that there exists a sequence $\{ {f}_{m}\} $ of continuous functions ${f}_{m}:\xcal  \to \rzecz $ that separates points of $\xcal $. Let us denote by $\scal $ the $\sigma $-algebra generated by the maps $\{ {f}_{m}\} $.
Then
\begin{description}
\item[(j1) ] every compact subset of $\xcal $ is metrizable,
\item[(j2) ] if $({\mu }_{m})$ is a tight sequence of probability measures on $(\xcal ,\scal )$, then there exists a subsequence $({m}_{k})$, a probability space $(\Omega , \fcal ,\p )$ with $\xcal $-valued Borel measurable variables ${\xi }_{k}, \xi $ such that  ${\mu }_{{m}_{k}}$ is the law of ${\xi }_{k}$
and ${\xi }_{k}$ converges to $\xi $ almost surely on $\Omega $.
Moreover, the law of $\xi $ is a Radon measure.
\end{description}
\end{theorem}

\bigskip \noindent
Using  Theorem \ref{T:Jakubowski_Skorokhod_Ondrejat}, we obtain the following corollary which we will apply to construct a martingale solution to the Navier-Stokes equations.

\bigskip
\begin{cor}  \label{C:Skorokhod_Z}
Let $({\eta }_{n}{)}_{n \in \nat }$ be a sequence of $\zcal $-valued random variables such that their laws $\lcal ({\eta }_{n})$ on $(\zcal ,\tcal )$ form a tight sequence of probability measures.
Then there exists a subsequence $({n}_{k})$, a probability space $(\tOmega , \tfcal ,\tp )$ and $\zcal $-valued random variables $\tilde{\eta }, {\tilde{\eta }}_{k} $, $k \in \nat $ such that the variables ${\eta }_{k}$
and ${\tilde{\eta }}_{k}$ have the same laws on $\zcal $
and ${\tilde{\eta }}_{k}$ converges to $\tilde{\eta }$ almost surely on $\tOmega $.
\end{cor}

\bigskip  \noindent
\begin{proof}
It is sufficient to prove that on each space appearing in the definition (\ref{E:Z})
of the space $\zcal $ there exists a countable set of continuous real-valued functions separating points.

\bigskip \noindent
Since the spaces $\ccal ([0,T];U^{\prime })$ and ${L}^{2}(0,T; {H}_{loc })$ are separable metrizable and complete,
this condition is satisfied, see \cite{Badrikian_70}, expos\'{e} 8.

\bigskip  \noindent
For the space ${L}^{2}_{w}(0,T;V)$ it is sufficient to put
$$
    {f}_{m}(u):= \int_{0}^{T} \ilsk{u(t)}{{v}_{n}(t)}{V} \, dt \in \rzecz ,
 \qquad u \in {L}^{2}_{w}(0,T;V),\quad m \in \nat ,
$$
where $\{ {v}_{m}, m \in \nat  \} $ is a dense subset of ${L}^{2}(0,T;V)$.

\bigskip \noindent
Let us consider the space $\ccal ([0,T];{H}_{w})$ defined by (\ref{E:C([0,T];H_w)}).
Let $\{ {h}_{m}, \, m \in \nat  \}  $ be any dense subset of $H$ and let ${Q}_{T}$ be the  set of rational numbers belonging to the interval $[0,T]$.
Then the family $\{ {f}_{m,t}, \, m \in \nat , \, \, t \in {Q}_{T} \} $ defined by
$$
      {f}_{m,t}(u):= \ilsk{u(t)}{{h}_{m}}{H} \in \rzecz ,
 \qquad u \in  \ccal ([0,T];{H}_{w}), \quad m \in \nat ,
      \quad t \in {Q}_{T}
$$
consists of continuous functions separating points in  $\ccal ([0,T];{H}_{w})$.
Now, the statement follows from Theorem \ref{T:Jakubowski_Skorokhod_Ondrejat}, which completes the proof.
\end{proof}

\bigskip
%%%%%%%%%%%%%%%%%%%%%%%%%%%%%%%%%%%%%%%%%%%%%%%%%%%%%%%%%%%%%%%%%%%%%%%%%%%%%%%%%%%%%%%%%%%%%%%%%%%%%%%%%%
\section{Stochastic Navier-Stokes equations}  \label{S:Navier_Stokes}
%%%%%%%%%%%%%%%%%%%%%%%%%%%%%%%%%%%%%%%%%%%%%%%%%%%%%%%%%%%%%%%%%%%%%%%%%%%%%%%%%%%%%%%%%%%%%%%%%%%%%%%%%%%

\bigskip
\noindent
We  consider the following stochastic evolution equation
\begin{equation} \label{E:NS}
\begin{cases}
& du(t)+\acal u(t) \, dt+B\bigl( u(t),u(t) \bigr)\, dt = f(t) \, dt+G\bigl( u(t)\bigr) \, dW(t),
 \qquad t \in [0,T] , \\
& u(0) = {u}_{0} .
\end{cases}
\end{equation}

\bigskip
\noindent
\bf Assumptions.  \rm We assume that
\begin{description}
\item[(A.1)] $ W(t)$  is a cylindrical  Wiener process in a separable Hilbert space $Y$ defined on the stochastic basis $\bigl( \Omega , \fcal , \mathbb{F} , \p  \bigr) $ with a filtration $\mathbb{F}={\{ \ft \} }_{t \ge 0}$;
\item[(A.2)] $\, \, {u}_{0} \in H $, $ f \in {L}^{p} (0,T; V^{\prime })$, where $p$ satisfies condition \eqref{eqn-p_cond} below;
\item[(A.3)] The mapping $G: V \to \lhs (Y,H) $ is Lipschitz continuous and
\begin{equation} \label{E:G}
     2 \dual{\acal u }{u }{} -  \norm{G(u )}{\lhs (Y,H)}{2}
     \ge  \eta \norm{u}{}{2} -{\lambda }_{0} {|u |}_{H}^{2} - \rho , \qquad u \in V ,
  \tag{G}
\end{equation}
for some constants ${\lambda }_{0}$, $\rho $ and $\eta \in (0,2]$.

\bigskip \noindent
Moreover, $G $ extends to a  mapping $G : H \to \lhs (Y, {V^{\prime }}) $  such that
\begin{equation} \label{E:G*}
   \norm{G(u)}{\lhs (Y, {V^{\prime }})}{2} \le C (1 + {|u|}_{H}^{2}) , \qquad u \in H .
  \tag{G$^\ast $}
\end{equation}
for some $C>0$. Moreover, for every $\psi \in \vcal $
\begin{equation} \label{E:G**}
  \mbox{ the mapping } H \ni u \mapsto \dual{G(u)}{\psi }{} \in Y  \mbox{ is continuous},
   \tag{G$^{\ast \ast} $}
\end{equation}
if in the space $H$ we consider the Fr\'{e}chet topology  inherited from the space
${L}^{2}_{loc}(\mathcal{O} , {\mathbb{R}}^{d})$.
\end{description}

\bigskip \noindent
By ${L}^{2}_{loc}(\mathcal{O} , {\mathbb{R}}^{d})$ we denote the space of all Lebesgue measurable $\rd $-valued functions $v$ such that $\int_{K}|v(x){|}^{2} \, dx < \infty $ for every compact subset $K \subset \ocal $. In this space we consider the Fr\'{e}chet topology generated by the family of seminorms
$$
     \Bigl( \int_{{\ocal }_{R}} |v(x){|}^{2} \, dx {\Bigr) }^{\frac{1}{2}} , \qquad R \in \nat ,
$$
where $({\ocal }_{R}{)}_{R \in \nat }$ is any increasing sequence of open bounded subsets of $\ocal $.

\bigskip \noindent
More precisely, in condition (\ref{E:G**}) we identify $\dual{G(\cdot )}{\psi }{} $ with the mapping ${\psi}^{\ast \ast } G: H \to Y^\prime$
defined by
\begin{equation}
  \bigl( {\psi}^{\ast \ast } G(u)\bigr) y:= \bigl( G(u)y \bigr) \psi \in \rzecz ,
 \qquad u\in H, \quad y\in Y   .   \tag{G$^{\prime\ast \ast} $}
\end{equation}

\bigskip  \noindent
Inequality (G) in assumption (A.3) is the same as considered by Flandoli and G\c{a}tarek in \cite{Flandoli+Gatarek_1995} for bounded domains. The assumption $\eta = 2$ corresponds to the case when the noise term does not depend on $\nabla u$. We will prove that the set of measures induced on appropriate space by the solutions of the Galerkin equations  is tight provided that
 assumptions (G) and (\ref{E:G*}) are satisfied.
Assumptions  (\ref{E:G*}),  (\ref{E:G**})  will be important in passing to the limit as $n\to \infty $ in the Galerkin approximation.  Assumption (\ref{E:G**}) is essential if  the domain is unbounded.

\bigskip
\begin{definition}  \rm  \label{def-sol-martingale}%{D:solution}
We say that there exists \bf a martingale solution \rm of the equation (\ref{E:NS})
iff there exist
\begin{itemize}
\item a stochastic basis $\bigl( \hat{\Omega }, \hat{\fcal }, \hat{\fmath } ,
\hat{\p }  \bigr) $ with filtration $\hat{\fmath }={\{ {\hat{\fcal }}_{t} \} }_{t \ge 0}$,
\item a cylindrical Wiener process $\hat{W}$ on the space $Y$,
\item and a progressively measurable process
$u: [0,T] \times \hat{\Omega } \to H$ with $\hat{\p } $-a.e. paths
$$
  u(\cdot , \omega ) \in \ccal \bigl( [0,T], {H}_{w} \bigr)
   \cap {L}^{2}(0,T;V )
$$
\end{itemize}
such that for all $ t \in [0,T] $ and all $v \in \vcal $:
\begin{eqnarray}
 \ilsk{u(t)}{v}{H} +  \int_{0}^{t} \dual{\acal u(s)}{v}{} \, ds
+ \int_{0}^{t} \dual{B(u(s),u(s))}{v}{} \, ds & & \nonumber \\
  = \ilsk{{u}_{0}}{v}{H} +  \int_{0}^{t} \dual{f(s)}{v}{} \, ds
 + \Dual{\int_{0}^{t} G(u(s))\, d\hat{W}(s)}{v}{}  & &  \label{E:mart_sol_int_identity}
\end{eqnarray}
the identity holds $\hat{\p }$-a.s.
\end{definition}

\bigskip
%%%%%%%%%%%%%%%%%%%%%%%%%%%%%%%%%%%%%%%%%%%%%%%%%%%%%%%%%%%%%%%%%%%%%%%%%%%%%%%%%%%%%%%%%%%%%%%%%%%%%%%%%%%%%
\section{Existence of solutions}  \label{S:Existence}
%%%%%%%%%%%%%%%%%%%%%%%%%%%%%%%%%%%%%%%%%%%%%%%%%%%%%%%%%%%%%%%%%%%%%%%%%%%%%%%%%%%%%%%%%%%%%%%%%%%%%%%%%%%%%%%%

\bigskip
\begin{theorem} \label{T:existence}
Let assumptions (A.1)-(A.3) be satisfied. Then there exists a martingale solution
$\bigl( \hat{\Omega }, \hat{\fcal }, \hat{\fmath },\hat{\p } ,u \bigr) $ of problem (\ref{E:NS}) such that
\begin{equation}
  \hat{\e} \Bigl[ \sup_{t \in [0,T]} |u(t){|}_{H}^{2} + \int_{0}^{T}\norm{u(t)}{}{2} \, dt  \Bigr]
  < \infty .
\end{equation}

%There exists a martingale solution of  problem (\ref{E:NS}) provided assumptions (A.1)-(A.3) are %satisfied.
\end{theorem}

\subsection{Faedo-Galerkin approximation}

\bigskip  \noindent
Let $\{ {e}_{i} {\} }_{i =1}^{\infty  }$ be the orthonormal basis in $H$ composed of eigenvectors of $L$.
Let ${H}_{n}:= span \{ {e}_{1}, \! ...,\! {e}_{n} \} $ be the subspace with the norm inherited from $H$ and
let $\Pn : U^{\prime } \to {H}_{n} $  be defined by (\ref{E:tP_n}).
Consider the following mapping
$$
  \Bn (u):= \Pn B({\chi }_{n}(u),u) , \qquad u \in {H}_{n},
$$
where ${\chi }_{n}:H \to H $ is defined by ${\chi }_{n}(u) = {\theta }_{n}(|u {|}_{U^{\prime }})u$
with ${\theta }_{n } : \rzecz \to [0,1]$  of class ${\ccal }^{\infty }$ such that
$$
 {\theta }_{n}(r) =
 \begin{cases}
  1 \quad \mbox{if} \quad  r \le n  \\
   0 \quad \mbox{if} \quad  r \ge n+1 .
 \end{cases}
$$
Since ${H}_{n} \subset H$, ${B}_{n}$ is well defined. Moreover, ${B}_{n}:{H}_{n} \to {H}_{n}$ is globally Lipschitz continuous.

\bigskip  \noindent
Let us consider the classical Faedo-Galerkin approximation in the space $ {H}_{n}$
\begin{equation} \label{E:Galerkin}
\begin{cases}
  & d \un (t) =  - \bigl[ \Pn \acal \un (t)  + \Bn  \bigl(\un (t)\bigr)  - \Pn f(t) \bigr] \, dt
   + \Pn G\bigl( \un (t)\bigr) \, dW(t),   \quad t \in [0,T] , \\
  &  \un (0) = \Pn {u }_{0} .
\end{cases}
\end{equation}

\bigskip \noindent
The proof  of the next result is standard and thus omitted.

\begin{lemma} \label{L:Galerkin_existence}
For each $n \in \nat $,  there exists a solution of the Galerkin equation (\ref{E:Galerkin}). Moreover, $\un \in \ccal ([0,T];{H}_{n})$, $\p $-a.s. and $\e [\int_{0}^{T}|\un (s){|}_{H}^{q}\, ds]< \infty $ for any $q\in [2,\infty )$.
\end{lemma}

\bigskip  \noindent
Using the It\^{o} formula and the Burkholder-Davis-Gundy inequality, see \cite{DaPrato_Zabczyk_Erg}, we will prove the following lemma about \it a priori \rm estimates of the solutions $\un $ of (\ref{E:Galerkin}).
Let us put the following condition on $p$
\begin{equation} \label{eqn-p_cond}%{E:p_cond}
\begin{cases} p\in  \bigl[ 2, 2+ \frac{\eta }{2-\eta } \bigr) \quad \mbox{ if $\quad \eta \in(0,2)$,  } \\
p\in [2, \infty )  \quad \mbox{ if $\quad \eta =2 $}.
\end{cases}
\end{equation}

\bigskip
\begin{lemma} \label{L:Galerkin_estimates }
The processes $(\un {)}_{n \in \nat }$ satisfy the following estimates.
\begin{description}
\item[(i) ]
For every $p$ satisfying (\ref{eqn-p_cond}) there exist  positive constants ${C}_{1}(p)$ and ${C}_{2}(p)$ such that
\begin{equation} \label{E:H_estimate}
 \sup_{n \ge 1 } \e \bigl( \sup_{0 \le s \le T } |\un (s){|}_{H}^{p} \bigr) \le {C}_{1}(p) .
\end{equation}
and
\begin{equation} \label{E:HV_estimate}
 \sup_{n \ge 1 } \e \bigl[ \int_{0}^{T} |\un (s){|}_{H}^{p-2} \norm{ \un (s)}{}{2} \, ds \bigr] \le {C}_{2}(p)  .
\end{equation}
\item[(ii)] In particular, with ${C}_{2}:= {C}_{2}(2)$
\begin{equation} \label{E:V_estimate}
  \sup_{n \ge 1 }  \e \bigl[ \int_{0}^{T} \norm{ \un (s)}{}{2} \, ds \bigr] \le {C}_{2}.
\end{equation}
\end{description}
\end{lemma}

\bigskip
\begin{proof}
See Appendix A.
\end{proof}

\bigskip \noindent
%%%%%%%%%%%%%%%%%%%%%%%%%%%%%%%%%%%%%%%%%%%%%%%%%%%%%%%%%%%%%%%%%%%%%%%%%%%%%%%%%%%%%%%%%%%%%%%%%%%%%%%%%%%%%%%%%%%
\subsection{Tightness}
%%%%%%%%%%%%%%%%%%%%%%%%%%%%%%%%%%%%%%%%%%%%%%%%%%%%%%%%%%%%%%%%%%%%%%%%%%%%%%%%%%%%%%%%%%%%%%%%%%%%%%%%%%%%%%%%%%%

\bigskip  \noindent
For each $n \in \nat $, the solution $\un $ of the Galerkin equation defines a measure
$\lcal (\un )$ on $(\zcal , \tcal )$. Using Corollary \ref{C:tigthness_criterion}, inequality (\ref{E:V_estimate}) and inequality (\ref{E:H_estimate}) with $p=2$ we will prove the tightness of this set of measures.

\bigskip
\begin{lemma} \label{L:comp_Galerkin}
The set of measures $\bigl\{ \lcal (\un ) , n \in \nat  \bigr\} $ is tight on $(\zcal , \tcal )$.
\end{lemma}

\bigskip
\begin{proof}
We apply Corollary \ref{C:tigthness_criterion}.
According to estimates  (\ref{E:H_estimate}) and (\ref{E:V_estimate}), conditions (a), (b) are satisfied.
Thus, it is sufficient to  prove that the sequence $(\un {)}_{n \in \nat }$ satisfies the Aldous condition \textbf{[A]}.
Let ${(\taun )}_{n \in \nat} $ be a sequence of stopping times such that $0 \le \taun \le T$.
By (\ref{E:Galerkin}), we have
\begin{eqnarray*}
& & \un (t) \,
 =   \Pn {u}_{0}  - \int_{0}^{t} \Pn \acal  \un (s) \, ds
  - \int_{0}^{t} \Bn \bigl( \un (s) \bigr) \, ds
  + \int_{0}^{t} \Pn f(s) \, ds
  + \int_{0}^{t} \Pn G(\un (s)) \, dW(s)  \nonumber \\
& & =:    \Jn{1} + \Jn{2}(t) + \Jn{3}(t) + \Jn{4}(t) + \Jn{5}(t), \qquad t \in [0,T].
\end{eqnarray*}
Let $\theta  >0 $. First, we make some estimates for each term of the above equality.

\bigskip \noindent
\textbf{ Ad.} ${\Jn{2}}$. \rm  Since $\acal :V \to V^{\prime }$ and ${|\acal (u)|}_{V^{\prime }} \le \norm{u}{}{}$
and the embedding  $V^{\prime } \hookrightarrow U^{\prime }$ is continuous, then by the H\"{o}lder inequality and (\ref{E:V_estimate}), we have the following estimates
\begin{eqnarray}
& &\e \bigl[ \bigl| \Jn{2} (\taun + \theta ) - \Jn{2}(\taun ) {\bigr| }_{U^{\prime }}  \bigr]
 = \e \Bigl[ {\biggl| \int_{\taun }^{\taun + \theta } \Pn \acal \un (s) \, ds \biggr| }_{U^{\prime }} \Bigr]  \le c  \e \Bigl[  \int_{\taun }^{\taun + \theta }
{\bigl|  \acal  \un (s) \bigr| }_{V^{\prime }}  \, ds \biggr]
 \nonumber \\
& &
 \le c  \e \biggl[  \int_{\taun }^{\taun + \theta }
 \norm{  \un (s) }{}{}  \, ds \Bigr]
 \le c {\theta }^{\frac{1}{2}} \Bigl( \e \Bigl[  \int_{0 }^{T }
 \norm{  \un (s) }{}{2}  \, ds \Bigr] {\Bigr) }^{\frac{1}{2}}
\le c   {C}_{2} \cdot {\theta }^{\frac{1}{2}}=: {c}_{2} \cdot {\theta }^{\frac{1}{2}}.  \label{E:Jn2}
\end{eqnarray}

\bigskip \noindent
\textbf{Ad.} ${\Jn{3}}$. \rm Let $\gamma > \frac{d}{2} +1 $
Similarly, since $B: H \times H \to {V}_{\gamma }^{\prime }$ is bilinear and continuous, and the embedding
${V}_{\gamma }^{\prime } \hookrightarrow U^{\prime }$ is continuous, then by (\ref{E:H_estimate}) we have the following estimates
\begin{eqnarray}
& &\e \bigl[ \bigl| \Jn{3} (\taun + \theta ) - \Jn{3}(\taun) {\bigr| }_{U^{\prime }}  \bigr]
 = \e \Bigl[ { \Bigl| \int_{\taun }^{\taun + \theta }
  \Bn \bigl(\un (s) \bigr) \, ds \Bigr| }_{U^{\prime }} \Bigr]
 \le c\e \Bigl[  \int_{\taun }^{\taun + \theta }
{ \bigl|  B\bigl( \un (s)  \bigr)  \bigr| }_{{V}_{\gamma }^{\prime }} \, ds \Bigr] \nonumber \\
& & \le c\e \biggl[  \int_{\taun }^{\taun + \theta }
  \| B \| \cdot \bigl|  \un (s) { \bigr| }_{H}^{2}   \, ds \biggr]
 \le c\| B \|  \cdot  \e \bigl[ \sup_{s \in [0,T]} \bigl| \un (s) { \bigr| }_{H}^{2}\bigr] \cdot \theta
 \le c\| B \| \,  {C}_{1}(2)  \cdot \theta =: {c}_{3} \cdot \theta , \qquad \, \label{E:Jn3}
\end{eqnarray}
where $\| B \| $ stands for the norm of $B: H \times H \to {V}_{\gamma }^{\prime }$.

\bigskip \noindent
\textbf{Ad.} ${\Jn{4}}$. \rm By the continuity of the embedding $U \hookrightarrow V$ we have
\begin{eqnarray}
& &\e \bigl[ \bigl| \Jn{4} (\taun + \theta ) - \Jn{4}(\taun) {\bigr| }_{U^{\prime }}  \bigr]
 = \e \Bigl[ {\Bigl| \int_{\taun }^{\taun + \theta } \Pn f (s) \, ds \Bigr| }_{U^{\prime }} \Bigr]
\le c \, \e \biggl[ {\Bigl| \int_{\taun }^{\taun + \theta }
   f (s) \, ds \Bigr| }_{V^{\prime }} \biggr]  \nonumber \\
& & \le c   {\theta }^{\frac{1}{2}}   \Bigl( \e \Bigl[  \int_{0 }^{T }
{\bigl|  f(s) \bigr| }_{V^{\prime }}^{2}  \, ds \Bigr] {\Bigr) }^{\frac{1}{2}}
= c  {\theta }^{\frac{1}{2}}  \norm{f}{{L}^{2}(0,T;V^{\prime })}{} =: {c}_{4} \cdot {\theta }^{\frac{1}{2}}. \label{E:Jn4}
\end{eqnarray}

\bigskip \noindent
\textbf{Ad.} ${\Jn{5}}$. \rm Since  $ V^{\prime } \hookrightarrow U^{\prime } $, then
by (G*) and (\ref{E:H_estimate}), we obtain the following inequalities
\begin{eqnarray}
& &\e \bigl[ \bigl| \Jn{5} (\taun + \theta ) - \Jn{5}(\taun) {\bigr| }_{U^{\prime }}^{2}  \bigr]
 \nonumber \\
& & = \e \Bigl[ {\Bigl| \int_{\taun }^{\taun + \theta }
 \Pn G( \un (s) ) \, dW(s) \Bigr| }_{U^{\prime }}^{2} \Bigr]
=\e \Bigl[  \int_{\taun }^{\taun + \theta }
 \norm{\Pn G( \un (s) )}{\lhs (Y ,U^{\prime })}{2} \, ds  \Bigr]
    \nonumber \\
& & \le c \, \e \Bigl[  \int_{\taun }^{\taun + \theta }
 \norm{G( \un (s) )}{\lhs (Y ,V^{\prime })}{2} \, ds  \Bigr]
 \le c C\cdot \e \Bigl[  \int_{\taun }^{\taun + \theta }
  ( 1 + | \un (s) {|}_{H}^{2} )\, ds  \Bigr]  \nonumber \\
& & \le cC    \bigl( 1 +
  \e \bigl[ \sup_{s \in [0,T]} \bigl| \un (s) { \bigr| }_{H}^{2}\bigr] \bigr) \theta
\le c C (1+ {C}_{1}(2) ) \theta =: {c}_{5} \cdot \theta . \label{E:Jn5}
\end{eqnarray}

\bigskip \noindent
Let us fix $\eta > 0 $ and $\eps >0$. By the Chebyshev inequality and estimates
(\ref{E:Jn2})-(\ref{E:Jn4}), we obtain
$$
\p \bigl( \bigl\{
\bigl| \Jn{i} (\taun + \theta ) - \Jn{i}(\taun) {\bigr| }_{U^{\prime }} \ge \eta \bigr\} \bigr)
 \le \frac{1}{\eta } \e \bigl[ \bigl| \Jn{i} (\taun + \theta ) - \Jn{i}(\taun) {\bigr| }_{U^{\prime }}  \bigr]
 \le \frac{{c}_{i} \cdot \theta}{\eta }  ,  \qquad n \in \nat
$$
where $i=1,2,3,4$.  Let ${\delta }_{i}:= \frac{\eta }{{c}_{i}} \cdot \eps $.
Then
$$
 \sup_{n\in \nat }\sup_{1 \le \theta \le {\delta }_{i}}  \p \bigl\{
  \bigl| \Jn{i} (\taun + \theta ) - \Jn{i}(\taun) {\bigr| }_{U^{\prime }} \ge \eta \bigr\} \le \eps  ,
  \qquad i=1,2,3,4.
$$
By the Chebyshev inequality and (\ref{E:Jn5}), we have
$$
\p \bigl( \bigl\{
\bigl| \Jn{5} (\taun + \theta ) - \Jn{5}(\taun) {\bigr| }_{U^{\prime }} \ge \eta \bigr\} \bigr)
 \le \frac{1}{{\eta }^{2}} \e \bigl[
  \bigl| \Jn{5} (\taun + \theta ) - \Jn{5}(\taun) {\bigr| }_{U^{\prime }}^{2} \bigr]
 \le  \frac{{c}_{5} \cdot \theta}{{\eta }^{2}} ,  \qquad n \in \nat .
$$
Let ${\delta }_{5}:= \frac{{\eta }^{2}}{{c}_{5}} \cdot \eps $.
Then
$$
 \sup_{n\in \nat }\sup_{1 \le \theta \le {\delta }_{5}} \p \bigl\{
  \bigl| \Jn{5} (\taun + \theta ) - \Jn{5}(\taun) {\bigr| }_{U^{\prime }} \ge \eta \bigr\} \le \eps  .
$$
Since condition \bf [A] \rm holds for each term $\Jn{i}$, $i=1,2,3,4,5$, we infer that it holds also for $(\un )$.
This completes the proof of lemma.
 \end{proof}

%%%%%%%%%%%%%%%%%%%%%%%%%%%%%%%%%%%%%%%%%%%%%%%%%%%%%%%%%%%%%%%%%%%%%%%%%%%%%%%%%%%%%%%%%%%%%%%%%%%%%%%%%%%%%%%%
\bigskip  \noindent
\subsection{Proof of Theorem \ref{T:existence}}
%%%%%%%%%%%%%%%%%%%%%%%%%%%%%%%%%%%%%%%%%%%%%%%%%%%%%%%%%%%%%%%%%%%%%%%%%%%%%%%%%%%%%%%%%%%%%%%%%%%%%%%%%%%%%%%%%
\bigskip \noindent
The following proof differs from the approach of Mikulevicius and Rozovskii \cite{Mikulevicius+Rozovskii_2005}
and it is based on the method used by  Da Prato and  Zabczyk in \cite{DaPrato+Zabczyk_1992}, Section 8,
and on the Jakubowski's version of the Skorokhod Theorem for nonmetric spaces.

\bigskip \noindent
By Lemma \ref{L:comp_Galerkin} the set of measures $\bigl\{ \lcal (\un ),n\in \nat \bigr\} $ is tight on the space $(\zcal ,\tcal )$ defined by (\ref{E:Z}).
Hence by Corollary \ref{C:Skorokhod_Z} there exist  a subsequence $({n}_{k}{)}_{k}$, a probability space
$\bigl( \tOmega ,\tfcal ,\tp  \bigr) $ and, on this space,
$ \zcal $-valued random variables $\tu $, $\tunk $, $k \ge 1 $
such that
\begin{equation}   \label{E:Skorokhod_appl}
  \tunk \mbox{ \it has the same law as } \unk \mbox{ \it on } \zcal
  \mbox{ \it and } \tunk \to \tu \mbox{ \it in } \zcal,
    \quad  \tp \mbox{ - \it a.s.}
\end{equation}
Let us denote the subsequence $(\tunk{)}_{k}$ again by $(\tun {)}_{n}$.

\bigskip  \noindent
Since $\un \in \ccal ([0,T];\Pn H)$, $\p $-a.s. and $\tun $ and $\un $ have the same laws, and
$\ccal ([0,T];\Pn H)$ is a Borel subset of $\ccal ([0,T]; U^{\prime }) \cap {L}^{2}(0,T;{H}_{loc})$, we have
$$
  \lcal (\tun ) \bigl( \ccal ([0,T]; \Pn H )  \bigr) =1 , \qquad n \ge 1 .
$$
Since $\tun $ and $\un $ have the same laws, and $\ccal ([0,T];\Pn H)$ is a Borel subset of
$\ccal ([0,T]; U^{\prime }) \cap {L}^{2}(0,T;{H}_{loc})$ thus by (\ref{E:H_estimate}) and (\ref{E:V_estimate}) we have
\begin{equation} \label{E:H_estimate'}
 \sup_{n\in \nat }\e \bigl( \sup_{0\le s\le T } \bigl| \tun (s){\bigr| }_{H}^{p}\bigr) \le {C}_{1}(p),
\end{equation}
\begin{equation} \label{E:V_estimate'}
  \sup_{n\in \nat } \e \Bigl[ \int_{0}^{T} {\bigl\| \tun (s)\bigr\| }_{V}^{2}\, ds \Bigr] \le {C}_{2}
\end{equation}
for  all $p$ satisfying condition (\ref{eqn-p_cond}).

\bigskip \noindent
By inequality (\ref{E:V_estimate'}) we infer that the sequence $(\tun )$ contain subsequence, still denoted by $(\tun )$ convergent weakly  in the space ${L}^{2}([0,T]\times \tOmega ; V )$.
Since by (\ref{E:Skorokhod_appl}) $\tp $-a.s. $\tun \to \tu $ in $\zcal $, we conclude that
$\tu \in {L}^{2}([0,T]\times \tOmega ; V )$, i.e.
\begin{equation} \label{E:tu_V_estimate}
   \e \Bigl[ \int_{0}^{T} \norm{\tu (s)}{}{2}\, ds \Bigr] < \infty .
\end{equation}
Similarly, by inequality (\ref{E:H_estimate'}) with $p=2$ we can choose a subsequence of $(\tun )$ convergent weak star in the space ${L}^{2}(\tOmega ; {L}^{\infty }(0,T;H))$ and, using (\ref{E:Skorokhod_appl}), infer that
\begin{equation} \label{E:tu_H_estimate}
\e \bigl[ \sup_{0\le s\le T } \bigl| \tu (s){\bigr| }_{H}^{2}\bigr] < \infty .
\end{equation}

\bigskip  \noindent
For each $n \ge 1$, let us consider a process $\tMn $ with trajectories
in $\ccal ([0,T];H)$ defined by
\begin{equation}  \label{E:tMn}
\tMn (t)
   = \tun (t) \,  -   \Pn \tu (0)  + \int_{0}^{t} \Pn \acal \tun (s) \, ds
  + \int_{0}^{t} \Bn  \bigl( \tun (s)  \bigr) \, ds
  - \int_{0}^{t} \Pn f(s) \, ds ,  \quad
t \in [0,T].
\end{equation}
$\tMn $ is a square integrable martingale with respect to the filtration
${\tilde{\mathbb{F}}}_{n} =({\tilde{\fcal }}_{n,t})$, where
${\tilde{\fcal }}_{n,t} = \sigma \{ \tun (s), \, \, s \le t \} $, with quadratic variation
\begin{equation} \label{E:tMn_qvar}
  {\qvar{\tMn }}_{t} =  \int_{0}^{t} \Pn G(\tun (s)) {G(\tun (s))}^\ast \Pn  \, ds ,
  \qquad t \in [0,T].
\end{equation}
Indeed, since $\tun $ and $\un $ have the same laws, for all $s, t \in [0,T]$, $s \le t $ all functions $h$ bounded continuous on $\ccal ([0,s]; U^{\prime } )$, and all $\psi , \zeta  \in U $,  we have
\begin{equation} \label{E:EtMn}
 \e \bigl[ \dual{ \tMn (t)-\tMn (s)}{\psi }{} \, h \bigl( \tun {}_{|[0,s]} \bigr) \bigr]
 = 0
\end{equation}
and
\begin{eqnarray}
& & \e \Bigl[ \Bigl( \dual{\tMn (t)}{\psi }{} \dual{\tMn (t)}{\zeta }{}
 - \dual{\tMn (s)}{\psi }{} \dual{\tMn (s)}{\zeta }{} \nonumber \\
& & \quad - \int_{s}^{t} \Ilsk{{G(\tun (\sigma ))}^\ast \Pn \psi }{{G(\tun (\sigma ))}^\ast \Pn \zeta }{Y}
\, d\sigma  \Bigl)  \cdot h \bigl( \tun {}_{|[0,s]} \bigr) \Bigr] =0 .   \label{E:EqvtMn}
\end{eqnarray}
Here, $\dual{\cdot }{\cdot }{}$ stands for the dual pairing between $U^{\prime }$ and $U$.
Let us recall that $Y$ is a Hilbert space defined in assumption (A.1).
We will take the limits in (\ref{E:EtMn}) and (\ref{E:EqvtMn}).
Let $\tM $ be an $U^{\prime }$-valued process defined by
\begin{equation}  \label{E:tM}
\tM (t)
   = \tu (t) \,  -    \tu (0)  + \int_{0}^{t}  \acal \tu (s) \, ds
  + \int_{0}^{t} B  \bigl( \tu (s)  \bigr) \, ds
  - \int_{0}^{t}  f(s) \, ds , \quad
t \in [0,T].
 \end{equation}

\bigskip  \noindent
\begin{lemma} \label{L:pointwise_conv}
For all $s,t \in [0,T]$ such that $s \le t $ and all  $\psi  \in U$:
\begin{itemize}
\item[(a)] $\nlim \ilsk{\tun (t)}{\Pn \psi }{H} = \ilsk{\tu (t)}{\psi }{H}, \quad  \tp \mbox{ - a.s.}$,
\item[(b)] $\nlim \int_{s}^{t} \dual{ \acal \tun (\sigma) }{\Pn \psi }{} \, d\sigma
 =\int_{s}^{t} \dual{ \acal \tu (\sigma) }{ \psi }{}\, d\sigma , \quad  \tp \mbox{ - a.s.}$,
 \item[(c)] $
 \lim_{n \to \infty }
 \int_{s}^{t} \dual{ B \bigl( \tun (\sigma )  \bigr) }{\Pn \psi }{}  \, d\sigma
= \int_{s}^{t} \dual{ B\bigl( \tu (\sigma ) \bigr) }{\psi }{}  \, d\sigma ,
 \quad  \tp \mbox{ - a.s.}
$
\end{itemize}
\end{lemma}

\bigskip  \noindent
\begin{proof}
Let us fix $s,t \in [0,T]$, $s\le t$ and  $\psi \in U$. By (\ref{E:Skorokhod_appl}) we know that
\begin{equation}  \label{E:tun_conv_Z}
 \tun \to \tu \mbox{ in } \ccal ([0,T];U')
  \cap {L}^{2}_{w}(0,T;V) \cap {L}^{2}(0,T;{H}_{loc}) \cap \ccal ([0,T];{H}_{w}),
    \quad  \tp \mbox{ - a.s.}
\end{equation}
Thus $\tun \to \tu $ in $\ccal ([0,T], {H}_{w})$, $\tp $-a.s. and since by (\ref{E:P_n}) $\Pn \psi \to \psi $ in $H$, we infer that assertion (a) holds.

\bigskip  \noindent
Let us move to (b).
Since by (\ref{E:tun_conv_Z})  $\tun \to \tu $ in $ {L}^{2}_{w} (0,T; V )$, $\tp $-a.s.  and by assertion  (iii) in Lemma \ref{L:P_n|U} (c)
$\Pn \psi \to \psi $ in $V$, by (\ref{E:Acal_ilsk_Dir})
 we infer that $\tp $ - a.s.
$$
  \int_{s}^{t} \dual{ \acal \tun (\sigma) }{\Pn \psi }{} \, d\sigma =
  \int_{s}^{t} \dirilsk{\tun (\sigma)}{\Pn \psi }{} \, d\sigma
  \underset{\ninf }{\longrightarrow } \int_{s}^{t} \dirilsk{\tu (\sigma)}{\psi }{} \, d\sigma
  =\int_{s}^{t} \dual{ \acal \tu (\sigma) }{ \psi }{} \, d\sigma ,
$$
i.e. (b) holds.

\bigskip  \noindent
We will prove now assertion (c).
Since as above $\tun (\cdot ,\omega)\to \tu (\cdot ,\omega)$
in ${L}^{2}_{w}(0,T;V)$, in particular, $\tu (\cdot ,\omega) \in {L}^{2}(0,T;V)$ and the sequence
$(\tun (\cdot ,\omega ){)}_{n\ge 1 }$ is bounded in ${L}^{2}(0,T;V)$ for $\tp $-almost all $\omega \in \tOmega $. Thus $\tu (\cdot ,\omega) \in {L}^{2}(0,T;H)$ and the sequence
$(\tun (\cdot ,\omega ){)}_{n\ge 1 }$ is bounded in ${L}^{2}(0,T;H)$, as well.
Let us fix $\omega \in \tOmega $ such that
\begin{description}
\item[(i) ] $\tun (\cdot ,\omega)\to \tu (\cdot ,\omega)$ in ${L}^{2}(0,T,{H}_{loc})\cap \ccal ([0,T];U')$,
\item[(ii)] $\tu (\cdot ,\omega )\in {L}^{2}(0,T;H)$ and the sequence $(\tun (\cdot ,\omega ){)}_{n\ge 1 }$ is bounded in ${L}^{2}(0,T;H)$.
\end{description}
By (i)  the sequence
$(\tun (\cdot ,\omega ){)}_{n \ge 1 }$ is bounded in $\ccal ([0,T];U') $, i.e. for some $N >0 $
$$
    \sup_{n \ge 1} \norm{\tun (\cdot ,\omega ) }{\ccal ([0,T];U')}{} \le N .
$$
Thus ${\chi }_{n}(\tun (\cdot , \omega )) = \tun (\cdot ,\omega )$ for all $n > N $ and
$$
  B\bigl( {\chi }_{n}(\tun (\cdot , \omega )),\tun (\cdot ,\omega )\bigr)
 =  B\bigl( \tun (\cdot , \omega ),\tun (\cdot ,\omega )\bigr) \quad \mbox{for } n >N.
$$
Hence assertion (c) follows from  Corollary B.2.
This completes the proof of the lemma.
\end{proof}

\bigskip
\begin{lemma} \label{L:conv_martingale}
For all $ s,t \in [0,T] $ such that $s \le t$ and all $ \psi \in U$:
$$
\nlim \e \bigl[ \dual{\tMn (t)-\tMn (s) }{\psi }{}\, h \bigl( \tun {}_{|[0,s]} \bigr) \bigr]
= \e \bigl[ \dual{ \tM (t)-\tM (s) }{\psi }{} \, h \bigl( \tu {}_{|[0,s]} \bigr) \bigr] .
$$
\end{lemma}

\bigskip
\begin{proof}
Let us fix $s,t \in [0,T]$, $s \le t$ and  $\psi \in U$.
By (\ref{E:tP_n-P_n}) we have
\begin{eqnarray*}
& &\dual{\tMn (t)-\tMn (s)}{\psi }{}
  =  \ilsk{\tun (t)}{\Pn \psi }{H} - \ilsk{\tun (s)}{\Pn \psi }{H}
    + \int_{s}^{t} \dual{ \acal \tun (\sigma) }{\Pn \psi }{} \, d\sigma
     \\
 &  & +\int_{s}^{t} \dual{ B \bigl( {\chi }_{n}(\tun (\sigma  )),\tun (\sigma ) \bigr)  }{\Pn \psi }{}\, d\sigma
      +\int_{s}^{t} \dual{ f(\sigma) }{\Pn \psi }{} \, d\sigma .
\end{eqnarray*}
By Lemma \ref{L:pointwise_conv}, we infer that
\begin{equation} \label{E:mart_pointwise_conv}
   \lim_{n \to \infty }  \dual{\tMn (t)-\tMn (s)}{\psi }{} = \dual{\tM (t)-\tM (s)}{\psi }{},
   \quad  \tp \mbox{ - a.s.}
\end{equation}
Let us notice that $\tp $ - a.s. $\lim_{n \to \infty }h(\tun {}_{|[0,s]} ) =h( \tu {}_{|[0,s]})$
and $\sup_{n \in \nat } \norm{h( \tun {}_{|[0,s]})}{{L}^{\infty }}{} < \infty $.
Let us denote
\begin{equation*}
   {f}_{n}(\omega ) := \bigl( \dual{\tMn (t, \omega )}{\psi }{} - \dual{\tMn (s, \omega )}{\psi }{} \bigr)
    \, h \bigl( \tun {}_{|[0,s]} \bigr) , \qquad \omega \in \tOmega .
\end{equation*}
We will prove that the functions $\{ {f}_{n} {\} }_{n \in \nat }$ are uniformly integrable.
We claim  that
\begin{equation} \label{E:mart_uniform_int}
     \sup_{n \ge 1}  \te \bigl[ {| {f}_{n} |}^{2} \bigr] < \infty.
\end{equation}
Indeed, by the continuity of the embedding $U\hookrightarrow H$ and the Schwarz inequality, for each $n \in \nat $ we have
\begin{eqnarray}  \label{E:mart_uniform_int_1}
\te \bigl[ {| {f}_{n}|}^{2} \bigl] \le 2c\norm{h}{{L}^{\infty }}{2} \norm{\psi }{U}{2}
\te \bigl[ |\tMn (t){|}_{H}^{2}) + |\tMn (s){|}_{H}^{2} \bigr] .
\end{eqnarray}
Since $\tMn  $  is a continuous martingale with quadratic variation defined in (\ref{E:tMn_qvar}), by the Burkhol\-der-Davis-Gundy inequality we obtain
\begin{eqnarray} \label{E:mart_BDG_est}
\te  \bigl[ \sup_{t \in [0,T]} {| \tMn (t) | }^{2}\bigr]
\le c \te  \Bigl[ \Bigl( \int_{0}^{T}\norm{\Pn G(\tun (\sigma ))}{\lhs (Y,H)}{2}  \, d\sigma {\Bigr) }^{\frac{1}{2}} \Bigr] .
\end{eqnarray}
Since the restriction of $\Pn $ to $H$ is an orthogonal projection onto ${H}_{n}$,
by inequality (G) in assumption (A.3), we have
\begin{eqnarray} \label{E:mart_BDG_est_1}
\norm{ \Pn G(\tun (\sigma ))}{\lhs (Y,H)}{2} \le (2-\eta ) \norm{\tun (\sigma )}{}{2}
 + {\lambda }_{0} |\tun (\sigma ){|}_{H}^{2} + \rho , \qquad \sigma \in [0,T].
\end{eqnarray}
By (\ref{E:mart_BDG_est}), (\ref{E:mart_BDG_est_1}), (\ref{E:V_estimate'}) and (\ref{E:H_estimate'}), we infer that
\begin{eqnarray} \label{E:mart_uniform_int_2}
   \sup_{n \in \nat } \te  \bigl[ \sup_{t \in [0,T]} {| \tMn (t) | }^{2}\bigr] < \infty .
\end{eqnarray}
Then by (\ref{E:mart_uniform_int_1}) and (\ref{E:mart_uniform_int_2}) we see that (\ref{E:mart_uniform_int}) holds.
Since the sequence $\{ {f}_{n} {\} }_{n \in \nat }$ is uniformly integrable and by
(\ref{E:mart_pointwise_conv}) it is $\tp $-a.s. pointwise convergent, application of the Vitali Theorem
completes the proof of the Lemma.
\end{proof}

\bigskip
\begin{lemma} \label{L:qvar_conv_left}
For all $s,t \in [0,T]$ such that $s \le t$ and all $\psi ,\zeta \in U$:
\begin{eqnarray*}
 \lim_{n\to\infty }
 \e \Bigl[ \bigl\{  \dual{\tMn (t)}{\psi }{} \dual{\tMn (t)}{\zeta }{}
 - \dual{\tMn (s)}{\psi }{} \dual{\tMn (s)}{\zeta }{} \bigr\} \, h \bigl( \tun {}_{|[0,s]} \bigr)\Bigr] & & \\
= \e \Bigl[ \bigl\{  \dual{\tM (t)}{\psi }{} \dual{\tM (t)}{\zeta }{}
 - \dual{\tM (s)}{\psi }{} \dual{\tM (s)}{\zeta }{} \bigr\} \, h \bigl( \tu {}_{|[0,s]}\bigr) \Bigr] . & &
 \end{eqnarray*}
\end{lemma}

\bigskip
\begin{proof}
Let us fix $s,t \in [0,T]$ such that $s \le t$ and  $\psi ,\zeta \in U$ and let us denote
\begin{eqnarray*}
  & & {f}_{n}(\omega ) := \bigl\{  \dual{\tMn (t,\omega )}{\psi }{} \dual{\tMn (t, \omega )}{\zeta }{}
 - \dual{\tMn (s, \omega )}{\psi }{} \dual{\tMn (s, \omega )}{\zeta }{} \bigr\} \, h \bigl( \tun {}_{|[0,s]}(\omega )\bigr) ,\\
  & & f(\omega ) := \bigl\{  \dual{\tM (t,\omega )}{\psi }{} \dual{\tM (t, \omega )}{\zeta }{}
 - \dual{\tM (s, \omega )}{\psi }{} \dual{\tM (s, \omega )}{\zeta }{} \bigr\} \, h \bigl( \tu {}_{|[0,s]}(\omega )\bigr),
\qquad \omega \in \tOmega .
\end{eqnarray*}
By Lemma \ref{L:pointwise_conv}, we infer that
$\lim_{n \to \infty } {f}_{n} (\omega ) = f(\omega )$ for  $\tp $-almost all $\omega \in \tOmega $.

\bigskip  \noindent
We will prove that the functions $\{ {f}_{n} {\} }_{n \in \nat }$ are uniformly integrable.
To this end, it is sufficient to show that for some $r>1$,
\begin{equation} \label{E:uniform_int}
     \sup_{n \ge 1}  \e \bigl[ {| {f}_{n} |}^{r} \bigr] < \infty.
\end{equation}
In fact, we will show that condition (\ref{E:uniform_int}) holds for any $r \in \bigl(1 , 1+ \frac{\eta }{2(2-\eta )}\bigr) $ if $0 < \eta < 2$
and any  $r>1$ if $\eta =2$. Indeed,
for each $n \in \nat $ we have
\begin{eqnarray} \label{E:uniform_int_1}
& &  \e \bigl[ |{f}_{n}{|}^{r} \bigr]
\le C\norm{h}{{L}^{\infty }}{r}\norm{\psi }{U}{r} \norm{\eta }{U}{r} \e \bigl[ {| \tMn (t) | }^{2r}+ {| \tMn (s) | }^{2r}\bigr] .
\end{eqnarray}
Since $\tMn  $ is a continuous martingale with quadratic variation defined in (\ref{E:tMn_qvar}), by the Burkhol\-der-Davis-Gundy inequality we obtain
\begin{eqnarray} \label{E:BDG_est}
\e \bigl[ \sup_{t \in [0,T]}{| \tMn (t) | }^{2r}\bigr]
\le c \e \Bigl[ \Bigl( \int_{0}^{T} \norm{\Pn G(\tun (\sigma ))}{\lhs (Y,U')}{2}  \, d\sigma {\Bigr) }^{r} \Bigr]
\end{eqnarray}
By the continuity of the injection  $U\embed V_s$, Lemma \ref{L:P_n|U} and assumption  \eqref{E:G*}
$$
   \norm{\Pn G(\tun (\sigma ))}{\lhs (Y,U')}{2}
\le c{|\Pn |}_{\lcal (U,V)} \norm{ G(\tun (\sigma ))}{\lhs (Y,V')}{2}
\le C ({|\tun (\sigma ) |}_{H}^{2}+1 ), \qquad \sigma \in [0,T].
$$
Since $2r$ satisfies condition (\ref{eqn-p_cond}), by (\ref{E:H_estimate'}) we obtain the following inequalities
\begin{eqnarray}
 & & \e \Bigl[ \Bigl( \int_{0}^{T} \norm{\Pn G(\tun (s))}{\lhs (Y,U')}{2}  \, ds{\Bigr) }^{r} \Bigr]
 \le \tilde{C} \e \Bigl[ \Bigl( \int_{0}^{T}({|\tun (s) |}_{H}^{2}+1 )\, ds {\Bigr) }^{r} \Bigr] \nonumber \\
& & \le c \Bigl( \e \bigl[ \sup_{s\in [0,T]} {|\tun (s) |}_{H}^{2r} \bigr] +1 \Bigr)
 \le c \bigl( {C}_{1}(2r) + 1 \bigr) . \label{E:BDG_est_1}
\end{eqnarray}
By (\ref{E:BDG_est}), (\ref{E:BDG_est_1}) and (\ref{E:uniform_int_1}) we infer that
  condition (\ref{E:uniform_int}) holds. By the Vitali Theorem
$$
   \lim_{n\to \infty } \te \bigl[ {f}_{n} \bigr] = \te [f].
$$
The proof of the lemma is thus complete.
\end{proof}

\bigskip
\begin{lemma}  \label{L:conv_quadr_var}
\bf (Convergence in quadratic variation). \rm
For any $s,t \in [0,T]$ and $\psi ,\zeta \in U$, we have
\begin{eqnarray*}
 \lim_{n \to \infty } \e \biggl[  \biggl(
  \int_{s}^{t} \ilsk{{G(\tun (\sigma ))}^\ast \Pn \psi }{{G(\tun (\sigma ))}^\ast \Pn \zeta }{Y} \, d\sigma  \biggl)
   \cdot h \bigl( \tun {}_{|[0,s]} \bigr) \biggr] & & \\
 = \e \biggl[  \biggl(
   \int_{s}^{t} \ilsk{{G(\tu (\sigma ))}^\ast  \psi }{{G(\tu (\sigma ))}^\ast  \zeta }{Y} \, d\sigma  \biggl)
  \cdot h \bigl( \tu {}_{|[0,s]} \bigr) \biggr] . & &
\quad \end{eqnarray*}
\end{lemma}
\noindent
Let us recall that $Y$ is the Hilbert space defined in assumption (A.1) in Section \ref{S:Navier_Stokes}.

\bigskip \noindent
\begin{proof}
Let us fix $\psi , \zeta \in U $ and let us denote
$$
  \fn (\omega ):=  \biggl(
 \int_{s}^{t} \Ilsk{{G(\tun (\sigma , \omega ))}^\ast \Pn \psi }{{G(\tun (\sigma ,\omega ))}^\ast \Pn \zeta }{Y}
 \, d\sigma  \biggl)  \cdot h \bigl( \tun {}_{|[0,s]} \bigr) ,
\qquad \omega \in \tOmega .
$$
We will prove that the functions are uniformly integrable and convergent $\tp $-a.s.

\bigskip  \noindent
\bf Uniform integrability. \rm
It is sufficient to show that for some $r>1$
\begin{equation} \label{E:uniform_bound_f_n}
     \sup_{n \ge 1}  \e \bigl[ {|\fn |}^{r} \bigr] < \infty .
\end{equation}
We will prove that the above condition holds for every
$r \in \bigl(1 , 1+ \frac{\eta }{2(2-\eta )}\bigr) $ if $0 < \eta < 2$
and with every $r>1$ if $\eta =2$.

\bigskip  \noindent
Since $ \lhs (Y, {V}_{}^{\prime }) \hookrightarrow \lcal (Y, {V}_{}^{\prime }) $ and $U\hookrightarrow V$ continuously, then by (G$^\ast $) and Lemma \ref{L:P_n|U} we have
the following inequalities
$$
   \bigl\| { G(\tun (\sigma ,\omega )) }^\ast  \Pn \zeta  {\bigl\| }_{Y}^{}
   \le \norm{G(\tun (\sigma ,\omega )) }{\lcal (Y, {V}_{}^{\prime }) }{}
   \cdot \norm{\Pn \zeta }{{V}_{}}{} \le
 c \sqrt{  \bigl( | \tun (\sigma , \omega ) {|}_{H}^{2} + 1 \bigr)} \,  \norm{\zeta }{U}{}
$$
for some $c>0$.
Thus we have the following inequalities
\begin{eqnarray*}
& & {|\fn |}^{r}   =
{\Bigl| \biggl(
 \int_{s}^{t} \Ilsk{{G(\tun (\sigma , \omega ))}^\ast \Pn \psi }{{G(\tun (\sigma ,\omega ))}^\ast \Pn \zeta }{Y}
  \, d\sigma  \biggl)  \cdot h \bigl( \tun {}_{|[0,s]} \bigr)  \Bigr| }^{r} \\
& &\le \norm{h}{{L}^{\infty }}{r}
  { \biggl(
  \int_{s}^{t} \bigl\| { G(\tun (\sigma , \omega ))}^\ast \Pn \psi {\bigl\| }_{Y}
   \cdot \bigl\| { G(\tun (\sigma ,\omega ))}^\ast  \Pn \zeta  {\bigl\| }_{Y} \, d\sigma   \biggr) }^{r}  \\
& & \le {c}^{2r}\norm{h}{{L}^{\infty }}{r}  \cdot {\| \psi \| }_{U}^{r} \cdot
   {\| \zeta \| }_{U}^{r}  \cdot
  \biggl( \int_{s}^{t}  \bigl( | \tun (\sigma , \omega ) {|}_{H}^{2} + 1 \bigr)  \, d\sigma
  {\biggr) }^{r}   .
\end{eqnarray*}
Using the H\"{o}lder inequality, we obtain
$$
 \biggl( \int_{s}^{t}  \bigl( | \tun (\sigma , \omega ) {|}_{H}^{2} + 1 \bigr)  \, d\sigma {\biggr) }^{r}
  \le  (t-s{)}^{r-1}
        \cdot \int_{s}^{t} \bigl( | \tun (\sigma , \omega ) {|}_{H}^{2} + 1 {\bigr) }^{r} \, d\sigma
\le C \cdot
   \sup_{\sigma \in [0,T]}  \bigl( | \tun (\sigma , \omega ) {|}_{H}^{2r} + 1 \bigr)
$$
for some $C>0$. Thus
$$
 {|\fn |}^{r}
 \le \tilde{C} \cdot \sup_{\sigma \in [0,T]}  \bigl( | \tun (\sigma , \omega ) {|}_{H}^{2r} + 1 \bigr)
$$
for some $\tilde{C}>0$.  Hence  by (\ref{E:H_estimate'})
$$
\sup_{n\in \nat }\e \bigl[ {|\fn |}^{r} \bigr]
 \le \tilde{C} \cdot \e \bigl[ \sup_{\sigma \in [0,T]}|\tun (\sigma ,\omega ){|}_{H}^{2r} + 1 \bigr]
\le \tilde{C} \bigl( {C}_{1}(2r) +1  \bigr) < \infty .
$$
Thus condition (\ref{E:uniform_bound_f_n}) holds.

\bigskip  \noindent
\bf Pointwise convergence on $\tOmega$. \rm
Let us fix $\omega \in \tOmega $ such that
\begin{description}
\item[(i) ] $\tun (\cdot ,\omega) \to \tu (\cdot ,\omega)$ in ${L}^{2}(0,T, {H}_{loc})$,
\item[(ii)]  $\tu (\cdot ,\omega )\in {L}^{2}(0,T;H)$ and the sequence $(\tun (\cdot ,\omega ){)}_{n\ge 1}$ is bounded in ${L}^{2}(0,T;H)$.
\end{description}
We will prove that
$$
\lim_{n \to \infty } \int_{s}^{t} \Ilsk{ {G(\tun (\sigma , \omega ))}^\ast \Pn \psi }{
 {G(\tun (\sigma , \omega  ))}^\ast \Pn \zeta }{Y} \, d\sigma
 =
 \int_{s}^{t} \Ilsk{ {G(\tu (\sigma , \omega ))}^\ast  \psi }{{G(\tu (\sigma , \omega  ))}^\ast  \zeta }{Y} \, d\sigma  .
$$
Let us notice that it is sufficient to prove that
\begin{equation} \label{E:L^2(s,t;Y)_conv}
  {G(\tun (\cdot , \omega  ))}^\ast \Pn \psi  \to {G(\tu (\cdot , \omega  ))}^\ast  \psi
 \quad \mbox{in} \quad {L}^{2}(s,t;Y).
\end{equation}
We have
\begin{eqnarray}
& &\int_{s}^{t}  \bigl\| {G(\tun (\sigma , \omega  ))}^\ast \Pn \psi
 - {G(\tu (\sigma , \omega ))}^\ast  \psi  {\bigr\| }_{Y}^{2}  \, d\sigma  \nonumber \\
& &\le \int_{s}^{t}  \biggl( \bigl\| {G(\tun (\sigma , \omega  ))}^\ast ( \Pn \psi  -\psi ) {\bigr\| }_{Y}
 + \bigl\|  {G(\tun (\sigma , \omega ))}^\ast \psi
- {G(\tu (\sigma , \omega  ))}^\ast  \psi  {\bigr\| }_{Y}^{} {\biggr) }^{2}  \, d\sigma  \nonumber \\
& &\le 2 \int_{s}^{t}   \bigl| {G(\tun (\sigma , \omega  ))}^\ast {\bigl|}_{\lcal (V^{\prime },Y)}^{2}
 \cdot \norm{\Pn \psi  -\psi }{V}{2}  \, d \sigma
 +  2 \int_{s}^{t}
\bigl\|  {G(\tun (\sigma , \omega  ))}^\ast \psi
- {G(\tu (\sigma , \omega ))}^\ast  \psi  {\bigr\| }_{Y}^{2}   \, d\sigma  \nonumber \\
& & =: 2 \{ {I}_{1}(n) + {I}_{2}(n)\}    \label{E:2_(I1+I2)}.
\end{eqnarray}
Let us consider the term ${I}_{1}(n)$. Since $\psi \in U $, by assertion (iii) in Lemma \ref{L:P_n|U} (c), we have
$$ \lim_{n\to \infty } \norm{ \Pn \psi - \psi }{V}{}  =0  .$$
By  \eqref{E:G*}, the continuity of the embedding $\lhs (Y, {V}_{}^{\prime }) \hookrightarrow \lcal (Y,V^{\prime })$ and (ii),
we infer that
$$
 \int_{s}^{t}
  \bigl| {G(\tun (\sigma , \omega  ))}^\ast {\bigl|}_{\lcal ({V}_{}^{\prime },Y)}^{2} \, d \sigma
  \le C  \int_{s}^{t}   \bigl(
   {|\tun (\sigma , \omega ) |}_{H}^{2} + 1 \bigr) \, d \sigma
\le \tilde{C} \bigl( \sup_{n \ge 1} \norm{\tun (\cdot ,\omega )}{{L}^{2}(0,T;H)}{2} +1 \bigr) \le K
$$
for some constant $K>0$. Thus
$$
\lim_{n \to \infty } {I}_{1}(n)=
\lim_{n \to \infty }  \int_{s}^{t}   \bigl| {G(\tun (\sigma , \omega  ))}^\ast {\bigl|}_{\lcal (H,Y)}^{2}
 \cdot \norm{\Pn \psi  -\psi }{V}{2} \, d \sigma  = 0 .
$$
Let us move to the  term ${I}_{2}(n)$ in (\ref{E:2_(I1+I2)}). We will prove that for every $\psi \in V $ the term ${I}_{2}(n)$ tends to zero  as $n \to \infty $.
 Assume first that $\psi \in \vcal $. Then there exists $R>0$ such that
$\supp \psi $ is a compact subset of ${\ocal }_{R}$.
Since $\tun (\cdot ,\omega ) \to \tu (\cdot ,\omega )$ in ${L}^{2}(0,T;{H}_{loc})$, then in particular
$$
    \lim_{n\to \infty } {q}_{T,R} \bigl( \tun (\cdot ,\omega ) - \tu (\cdot ,\omega ) \bigr) =0,
$$
where ${q}_{T,R}$ is the seminorm defined by (\ref{E:seminorms}). In the other words,
$\tun (\cdot ,\omega ) \to \tu (\cdot ,\omega )$ in ${L}^{2}(0,T;$ ${H}_{{\ocal }_{R}})$.
Therefore there exists a subsequence $(\tunk (\cdot ,\omega ){)}_{k} $ such that
$$
  \tunk (\sigma ,\omega ) \to \tu (\sigma ,\omega )
   \quad \mbox{ in ${H}_{{\ocal }_{R}}$ for almost all $\sigma \in [0,T]$ as $ k\to \infty  $. }
$$
Hence by assumption (\ref{E:G**})
$$
  G \bigl( \tunk (\sigma ,\omega ){\bigr) }^\ast \psi \to G \bigl( \tu (\sigma ,\omega ) {\bigr) }^\ast \psi
   \quad \mbox{ in $Y$ for almost all $\sigma \in [0,T]$ as $ k\to \infty  $. }
$$
In conclusion, by the Vitali Theorem
$$
     \lim_{k\to \infty } \int_{s}^{t} \norm{ G \bigl( \tunk (\sigma ,\omega ){\bigr) }^\ast \psi -
      G \bigl( \tu (\sigma ,\omega ) {\bigr) }^\ast \psi}{Y}{2} \, d\sigma =0
      \qquad \mbox{for $\psi \in \vcal $.}
$$
Repeating the above reasoning for all subsequences, we infer that from every subsequence of
the sequence $\bigl( G \bigl( \tun (\sigma ,\omega ){\bigr) }^\ast \psi {\bigr) }_{n}$ we can choose the subsequence convergent in ${L}^{2}(s,t;Y)$ to the same limit. Thus the whole sequence
$\bigl( G \bigl( \tun (\sigma ,\omega ){\bigr) }^\ast \psi {\bigr) }_{n}$ is convergent  to
$G \bigl( \tu (\sigma ,\omega ){\bigr) }^\ast \psi  $ in ${L}^{2}(s,t;Y)$. At the same time
$$
    \lim_{n\to\infty } {I}_{2}(n) =0 \qquad \mbox{for every $\psi \in \vcal $.}
$$

\bigskip  \noindent
If $\psi \in V$ then for every $\eps >0$ we can find ${\psi }_{\eps } \in \vcal $ such that
$\norm{\psi - {\psi }_{\eps } }{V}{} < \eps $.
By the continuity of the embedding $\lhs (Y,V^{\prime }) \hookrightarrow \lcal (Y,V^{\prime })$, condition (G$^\ast $) and (ii), we obtain
\begin{eqnarray*}
& &\int_{s}^{t}   \bigl\|  {G(\tun (\sigma , \omega ))}^\ast \psi
- {G(\tu (\sigma ,\omega ))}^\ast  \psi  {\bigr\| }_{Y}^{2}   \, d\sigma  \\
& &\le 2  \int_{s}^{t} \!  \bigl\| [ {G(\tun (\sigma , \omega ))}^\ast
- {G(\tu (\sigma ,\omega ))}^\ast ] ( \psi - {\psi }_{\eps } ) {\bigr\| }_{Y}^{2}    d\sigma  \\
& & \qquad + 2  \int_{s}^{t} \!  \bigl\| [ {G(\tun (\sigma , \omega ))}^\ast
- {G(\tu (\sigma ,\omega ))}^\ast ]  {\psi }_{\eps }  {\bigr\| }_{Y}^{2}  d\sigma \\
& &\le 4 \int_{s}^{t}    \bigl[ \norm{G(\tun (\sigma , \omega ))}{\lcal (Y, V^{\prime })}{2}
+ \norm{G(\tu (\sigma , \omega ))}{\lcal (Y, V^{\prime })}{2} \bigr] \norm{ \psi - {\psi }_{\eps } }{V}{2}    d\sigma \\
& &\qquad + 2\int_{s}^{t}   \bigl\| [ {G(\tun (\sigma , \omega ))}^\ast
- {G(\tu (\sigma ,\omega ))}^\ast ]  {\psi }_{\eps }  {\bigr\| }_{Y}^{2}   \, d\sigma \\
& & \le c \bigl( \norm{\tun (\cdot , \omega )}{{L}^{2}(0,T;H)}{2}
  + \norm{\tu (\cdot , \omega )}{{L}^{2}(0,T;H)}{2} + 2T \bigr)
   \cdot {\eps }^{2} \\
& & \qquad   + 2\int_{s}^{t}   \bigl\| [ {G(\tun (\sigma , \omega ))}^\ast
- {G(\tu (\sigma ,\omega ))}^\ast ]  {\psi }_{\eps }  {\bigr\| }_{Y}^{2}   \, d\sigma  \\
& &\le C  {\eps }^{2}  + 2\int_{s}^{t}   \bigl\| [ {G(\tun (\sigma , \omega ))}^\ast
- {G(\tu (\sigma ,\omega ))}^\ast ]  {\psi }_{\eps }  {\bigr\| }_{Y}^{2}   \, d\sigma ,
\end{eqnarray*}
for some positive constants $c$ and $C$.
Passing to the upper limit as $n \to \infty $, we infer that
$$
 \limsup_{n \to \infty } \int_{s}^{t}   \bigl\|  {G(\tun (\sigma , \omega ))}^\ast \psi
- {G(\tu (\sigma ,\omega ))}^\ast  \psi  {\bigr\| }_{Y}^{2}   \, d\sigma
\le C  {\eps }^{2} .
$$
In conclusion, we proved that
$$
 \lim_{n \to \infty } \int_{s}^{t}   \bigl\|  {G(\tun (\sigma , \omega ))}^\ast \psi
- {G(\tu (\sigma ,\omega ))}^\ast  \psi  {\bigr\| }_{Y}^{2}   \, d\sigma   =0
$$
which completes the proof of (\ref{E:L^2(s,t;Y)_conv}) and of lemma \ref{L:conv_quadr_var}.
\end{proof}

\noindent
By Lemma \ref{L:conv_martingale} we can pass to the limit in (\ref{E:EtMn}).
By Lemmas \ref{L:qvar_conv_left} and \ref{L:conv_quadr_var} we can pass to the limit in  (\ref{E:EqvtMn}) as well.
After  passing to the limits in (\ref{E:EtMn}) and (\ref{E:EqvtMn}) we infer that
for all $\psi,\zeta \in U$:
\begin{equation} \label{E:EtM}
 \e \bigl[ \dual{\tM (t) - \tM (s)}{\psi}{} \, h\bigl( {\tu }_{|[0,s]} \bigr) \bigr] = 0
\end{equation}
and
\begin{equation} \label{E:EqvM}
  \e \Bigl[ \Bigl( \dual{\tM (t)}{\psi}{} \dual{\tM (t)}{\zeta}{}
   - \dual{\tM (s)}{\psi}{} \dual{\tM (s)}{\zeta}{}
  - \int_{s}^{t} \ilsk{{G(\tu (\sigma ))}^\ast  \psi }{{G(\tu (\sigma ))}^\ast  \zeta }{Y}
  \, d\sigma  \Bigl)  \cdot h \bigl( {\tu }_{|[0,s]} \bigr) \Bigr] =0 .
\end{equation}
where $\tM $ is an $U^{\prime }$-valued process defined by (\ref{E:tM}).

\bigskip  \noindent
\bf Continuation of the proof of Theorem \ref{T:existence}. \rm
Now, we apply the idea analogous to the reasoning used by Da Prato and Zabczyk, see \cite{DaPrato+Zabczyk_1992}, Section 8.3. Consider operator $L : U \supset D(L) \to H $ defined by (\ref{E:op_L}),
the inverse ${L}^{-1}:H \to U$ and its dual $({L}^{-1})^{\prime }: U^{\prime } \to H^{\prime }$.
By (\ref{E:EtM}) and  (\ref{E:EqvM}) with
$\psi := {L}^{-1} \varphi $ and $\zeta := {L}^{-1}\eta $, where $\varphi , \eta \in H $ and equality
 (\ref{E:L=AAsLs}), we infer that
\begin{eqnarray*}
& &  ({L}^{-1 })^{\prime } \tM (t), \, \, t \in [0,T]
 \mbox{ is a continuous square integrable martingale in } H^{\prime }\cong H \mbox{ with respect}\\
& &\mbox{ to the filtration }
  \tilde{\mathbb{F}} = \bigl( {\tfcal }_{t} \bigr) ,
  \mbox{ where } {\fcal }_{t}= \sigma \{ \tu (s), \, \, s \le t \} \mbox{ with the quadratic variation}\\
& & \qquad \qquad \qvar{({L}^{-1})^{\prime } \tM  }_{t}
 =\int_{0}^{t} ({L}^{-1})^{\prime } G(\tu (s)) {(G(\tu (s)) ({L}^{-1})^{\prime }) }^\ast   \, ds .
\end{eqnarray*}
In particular, the continuity of the process $({L}^{-1 })^{\prime } \tM$ follows from the fact that
$\tu \in \ccal  ([0,T]; U^{\prime } ) $.
\bigskip  \noindent
By the Martingale Representation Theorem, see \cite{DaPrato+Zabczyk_1992}, there exist
\begin{itemize}
\item a stochastic basis
$\bigl( \ttOmega , \ttfcal , {\{ \ttfcal {}_{t} \} }_{t \ge 0} , \ttp  \bigr) $,
\item a cylindrical Wiener process $\ttW (t)$ defined on this basis,
\item and a progressively measurable process $\ttu (t)$
such that
\end{itemize}
\begin{eqnarray*}
& &({L}^{-1})^{\prime } \ttu (t) \,
  - ({L}^{-1})^{\prime } \ttu (0)  + ({L}^{-1})^{\prime } \int_{0}^{t} \acal \ttu (s) \, ds
  + ({L}^{-1})^{\prime } \int_{0}^{t}  B \bigl( \ttu (s) , \ttu (s) \bigr) \, ds \\
& &
   -({L}^{-1})^{\prime } \int_{0}^{t}  f(s) \, ds =\int_{0}^{t} ({L}^{-1})^{\prime } G \bigl( \ttu (s) \bigr) \, d \ttW (s).
\end{eqnarray*}
However,
$$
  \int_{0}^{t} ({L}^{-1})^{\prime } G \bigl( \ttu (s) \bigr) \, d \ttW (s)
 = ({L}^{-1})^{\prime } \int_{0}^{t} G \bigl( \ttu (s) \bigr) \, d \ttW (s).
$$
Hence
for all $t\in [0,T ]$ and all $v \in U$
\begin{eqnarray*}
& &\ilsk{\ttu (t)}{v}{H} \,
 - \ilsk{\ttu (0)}{v}{H}  +  \int_{0}^{t} \dual{ \acal \ttu (s) }{v}{} \, ds
  +   \int_{0}^{t} \dual{ B \bigl( \ttu (s)\bigr)}{v}{} \, ds   \\
& &\qquad \qquad  =  \int_{0}^{t} \dual{ f(s)}{v}{} \, ds
+ \Dual{ \int_{0}^{t} G \bigl( \ttu (s) \bigr) \, d \ttW (s)}{v}{}.
\end{eqnarray*}
Thus the conditions from Definition \ref{def-sol-martingale} hold with
$\bigl( \hat{\Omega }, \hat{\fcal }, {\{ {\hat{\fcal }}_{t} \} }_{t \ge 0} ,
\hat{\p }  \bigr) = \bigl( \ttOmega , \ttfcal , {\{ \ttfcal {}_{t} \} }_{t \ge 0} , \ttp  \bigr) $,
$\hat{W}= \ttW $ and $u = \ttu $. The proof of Theorem \ref{T:existence} is thus complete.
\qed

%%%%%%%%%%%%%%%%%%%%%%%%%%%%%%%%%%%%%%%%%%%%%%%%%%%%%%%%%%%%%%%%%%%%%%%%%%%%%%%%%%%%%%%%
%%%%%%%%%%%%%%%%%%%%%%%%%%%%%%%%%%%%%%%%%%%%%%%%%%%%%%%%%%%%%%%%%%%%%%%%%%%%%%%%%%%%%%%%%%%%
\section{An Example}  \label{S:Example}
%%%%%%%%%%%%%%%%%%%%%%%%%%%%%%%%%%%%%%%%%%%%%%%%%%%%%%%%%%%%%%%%%%%%%%%%%%%%%%%%%%%%%%%%%%%%%%
%%%%%%%%%%%%%%%%%%%%%%%%%%%%%%%%%%%%%%%%%%%%%%%%%%%%%%%%%%%%%%%%%%%%%%%%%%%%%%%%%%%%%%%%%%%%%%%%%%

\noindent
Let
\begin{equation}
 G(u)(t,x) d W(t):= \sum_{i=1}^{\infty } \bigl[
 \bigl( {b}^{(i)}(x) \cdot \nabla  \bigr) u (t,x) + {c}^{(i)} (x) u(t,x) \bigr]
d {\beta }^{(i)}(t) ,
\end{equation}
where
\begin{eqnarray*}
& & {\beta }^{(i)} , i\in \nat \mbox{ - independent standard Brownian motions,} \\
& & \bi  : \overline{\ocal} \to \rd  \mbox{ - of class } {\ccal }^{\infty } , \quad i \in \nat \\
& & \ci  : \overline{\ocal} \to \rzecz   \mbox{ - of class } {\ccal }^{\infty } , \quad i \in \nat
\end{eqnarray*}
are given. Assume that
\begin{equation}  \label{E:C_1}
 {C}_{1}:=\sum_{i=1}^{\infty }\bigl( \| \bi {\| }_{{L}^{\infty }}^{2}+\| \diver \bi {\| }_{{L}^{\infty }}^{2} + \| \ci {\| }_{{L}^{\infty }}^{2}\bigr) < \infty
\end{equation}
and
\begin{equation} \label{E:est_coercive}
 \sum_{j,k=1}^{d} \bigl( 2 {\delta }_{jk}
- \sum_{i=1}^{\infty } {b}^{(i)}_{j}(x){b}^{(i)}_{k}(x) \bigr) {\zeta }_{j} {\zeta }_{k}
\ge a {|\zeta |}^{2} , \qquad \zeta \in \rd
\end{equation}
for some $a \in (0,2]$. Assumption (\ref{E:est_coercive}) is equivalent to the following one
\begin{equation}  \label{E:est_bi_bj_bk}
 \sum_{i=1}^{\infty } \sum_{j,k=1}^{d}
{b}^{(i)}_{j}(x){b}^{(i)}_{k}(x) \bigr) {\zeta }_{j} {\zeta }_{k}
\le 2  \sum_{j,k=1}^{d} {\delta }_{jk} {\zeta }_{j} {\zeta }_{k} - a {|\zeta |}^{2}
 = (2-a ){|\zeta |}^{2} .
\end{equation}

\bigskip  \noindent
Let $Y:= {l}^{2}(\nat )$, where ${l}^{2}(\nat )$ denotes the space of all sequences $({h}_{i}{)}_{i \in \nat } \subset \rzecz $ such that $\sum_{i=1}^{\infty }{h}_{i}^{2} < \infty $. It is a Hilbert space with the scalar product given by
$
     \ilsk{h}{k}{{l}^{2}} := \sum_{i=1}^{\infty } {h}_{i} {k}_{i},
$
where $h=({h}_{i})$ and $k=({k}_{i})$ belong to ${l}^{2}(\nat )$.
Let us put
\begin{equation} \label{E:G_def}
 G(u) h = \sum_{i=1}^{\infty } \bigl[
 \bigl( {b}^{(i)} \cdot \nabla  \bigr) u  + {c}^{(i)} u \bigr] {h}_{i} ,
\qquad u \in V, \quad h=({h}_{i}) \in {l}^{2}(\nat ).
\end{equation}
We will show that the mapping $G$ fulfils assumption \bf (A.3)\rm . Since  $G$ is linear, it is Lipschitz continuous provided that it is bounded.
Thus we will show that:
\begin{itemize}
\item The following inequality holds
\begin{equation}
     2 \dual{\acal u }{u }{} -  \norm{G(u )}{\lhs (Y,H)}{2}
     \ge  \eta \norm{u}{}{2} -{\lambda }_{0} {|u |}_{H}^{2}  , \qquad u \in V
  \tag{$\mathbf{\tilde{G}}$}
\end{equation}
for some constants ${\lambda }_{0}$ and $\eta \in (0,2]$.
\item Moreover, $G $ extends to a linear mapping $G : H \to \lhs (Y, V^{\prime }) $  and
\begin{equation}
   \norm{G(u)}{\lhs (Y, V^{\prime })}{} \le C {|u|}_{H}^{} , \qquad u \in H .
  \tag{$\mathbf{\tilde{G}}^\ast $}
\end{equation}
for some $C>0$.
\item Furthermore, for each $R>0$ the mapping $G : {H}_{{\ocal }_{R}} \to \lhs (Y, V^{\prime }({\ocal }_{R})) $ is well defined and satisfies the following estimate
\begin{equation}
  \norm{G(u)}{\lhs (Y, V^{\prime }({\ocal }_{R}))}{} \le {C}_{R}  {|u|}_{{H}_{{\ocal }_{R}}}^{} , \qquad u \in H .
  \tag{$\mathbf{\tilde{G}}{}_{R}^\ast $}
\end{equation}
for some ${C}_{R}>0$.

\bigskip \noindent
Here $V^{\prime } ({\ocal }_{R})$ is the dual space to $V({\ocal }_{R})$, where
\begin{equation}  \label{E:V(O_R)}
  V({\ocal }_{R}) := \mbox{ the closure of $\vcal ({\ocal }_{R})$ in ${H}^{1}({\ocal }_{R}, \rd  )$} ,
\end{equation}
and $\vcal ({\ocal }_{R})$ denotes the space of all divergence-free vector fields of class
${\ccal }^{\infty }$ with compact supports contained in ${\ocal }_{R}$.

\bigskip  \noindent
Let us recall that ${H}_{{\ocal }_{R}}$ is the space of restrictions to the subset ${\ocal }_{R}$ of
elements of the space $H$, i.e.
$$
    {H}_{{\ocal }_{R}} := \{ {u}_{|{\ocal }_{R}} ; \, \, \,  u \in H  \}
$$
with the scalar product defined by
$\ilsk{u}{v}{{H}_{{\ocal }_{R}}} := \int_{{\ocal }_{R}} u v \, dx $,  $u,v \in {H}_{{\ocal }_{R}}$.
\end{itemize}

\bigskip  \noindent
\bf Remark. \rm \it From condition ($\mathbf{\tilde{G}}{}_{R}^\ast $) it follows that the mapping $G$ satisfies condition (\ref{E:G**}) in assumption \rm (A.3).

\bigskip \noindent
Indeed, by estimate ($\mathbf{\tilde{G}}{}_{R}^\ast $) and the continuity of the embedding $\lhs (Y, V^{\prime }({\ocal }_{R})) \hookrightarrow \lcal (Y, V^{\prime }({\ocal }_{R}))$, we obtain
$$
    |G(u)y{|}_{V^{\prime }({\ocal }_{R})} \le C(R) {|u|}_{{H}_{{\ocal }_{R}}}^{} \norm{y}{Y}{} , \qquad
 u \in H , \quad y \in Y
$$
for some constant $C(R)>0$.
Thus for any $\psi  \in V({\ocal }_{R})$
$$
    |(G(u)y)\psi | \le C(R) {|u|}_{{H}_{{\ocal }_{R}}}^{} \norm{y}{Y}{} \norm{\psi }{V({\ocal }_{R})}{} , \qquad
 u \in H , \quad y \in Y.
$$
Since by definition $({\psi }^{\ast \ast }G(u))y:= (G(u)y)\psi $, thus from the above inequality we infer that
\begin{equation} \label{E:psi**_est}
  |{\psi }^{\ast \ast }G(u){|}_{Y^\prime} \le C(R) \norm{\psi }{V}{} {|u|}_{{H}_{{\ocal }_{R}}}^{}.
\end{equation}
Therefore if we fix $\psi \in \vcal $ then there exists ${R}_{0}>0$
such that $\supp \psi $ is a compact subset of ${\ocal }_{{R}_{0}}$.
Since $G$ is linear,  estimate (\ref{E:psi**_est}) with $R:={R}_{0}$ yields that the mapping
$$
     {L}^{2}_{loc}(\ocal , \rd ) \supset H \ni u \mapsto {\psi }^{\ast \ast }G(u) \in Y^\prime
$$
is continuous in the Fr\'{e}chet topology inherited on the space $H$ from the space
${L}^{2}_{loc}(\ocal , \rd )$. Thus the mapping ${\psi }^{\ast \ast }G$ satisfies condition (\ref{E:G**}). \qed

\bigskip  \noindent
\bf Proof of ($\mathbf{\tilde{G}}$). \rm Let us consider a standard orthonormal basis ${h}^{(i)}$, $i\in \nat $, in ${l}^{2}(\nat )$. Let $u \in V $.
Then for each $i \in \nat $ we have
\begin{eqnarray*}
& & {\bigl| G(u) \hi \bigr| }_{H}^{2}
 = \Ilsk{ \sum_{j=1}^{d} \bij \frac{\partial u}{\partial {x}_{j}} + \ci u }{
    \sum_{k=1}^{d} \bik \frac{\partial u}{\partial {x}_{k}} + \ci u}{H} \\
& &= \Ilsk{ \sum_{j=1}^{d} \bij \frac{\partial u}{\partial {x}_{j}}  }{
    \sum_{k=1}^{d} \bik \frac{\partial u}{\partial {x}_{k}} }{H}
+ 2 \Ilsk{\sum_{j=1}^{d} \bij \frac{\partial u}{\partial {x}_{j}}  }{\ci u }{H}
+ {\bigl| \ci u  \bigr| }_{H}^{2} . \\
\end{eqnarray*}
Thus
\begin{eqnarray}
& &\norm{G(u)  }{\lhs (Y,H) }{2}  = \sum_{i=1}^{\infty } {\bigl| G(u) \hi \bigr| }_{H}^{2} \nonumber \\
& &= \sum_{i=1}^{\infty }  \Ilsk{ \sum_{j=1}^{d} \bij \frac{\partial u}{\partial {x}_{j}}  }{
    \sum_{k=1}^{d} \bik \frac{\partial u}{\partial {x}_{k}} }{H}
+ 2 \sum_{i=1}^{\infty }  \Ilsk{\sum_{j=1}^{d} \bij \frac{\partial u}{\partial {x}_{j}}  }{\ci u }{H}
+ \sum_{i=1}^{\infty } {\bigl| \ci u  \bigr| }_{H}^{2} . \label{E:norm_G(u)_(Y,H)}
\end{eqnarray}
We estimate each term on the right-hand side of the above equality.
By assumption (\ref{E:est_bi_bj_bk})
$$
\sum_{i=1}^{\infty }  \Ilsk{ \sum_{j=1}^{d} \bij \frac{\partial u}{\partial {x}_{j}}  }{
    \sum_{k=1}^{d} \bik \frac{\partial u}{\partial {x}_{k}} }{H}
 = \int_{\ocal} \sum_{i=1}^{\infty } \sum_{j,k=1}^{d}
 \bij (x)\bik (x) \frac{\partial u}{\partial {x}_{j}} \frac{\partial u}{\partial {x}_{k}} \, dx
 \le (2-a ) {|\nabla u |}_{H}^{2} .
$$
Let us move to the second term in (\ref{E:norm_G(u)_(Y,H)}).
For each $i \in \nat $, we have
$$
\Ilsk{ \sum_{j=1}^{d} \bij \frac{\partial u}{\partial {x}_{j}}  }{\ci u }{H}
\le  {\| \bi  \| }_{{L}^{\infty }}  {\| \ci \| }_{{L}^{\infty }} \norm{u}{}{} \cdot |u{|}_{H}
\le  \frac{1}{2} \bigl( {\| \bi  \| }_{{L}^{\infty }}^{2} +  {\| \ci \| }_{{L}^{\infty }}^{2} \bigr)
\norm{u}{}{} \cdot |u{|}_{H}.
$$
Thus for any $\eps >0$
$$
2 \sum_{i=1}^{\infty } \Ilsk{ \sum_{j=1}^{d} \bij \frac{\partial u}{\partial {x}_{j}}  }{\ci u }{H}
 \le  \sum_{i=1}^{\infty } \bigl(
 {\| \bi  \| }_{{L}^{\infty }}^{2}
 +{\| \ci \| }_{{L}^{\infty }}^{2} \bigr) \norm{u}{}{} \cdot |u{|}_{H}
 = {C}_{1} \norm{u}{}{} \cdot  |u{|}_{H}
\le \eps \norm{u}{}{2} + \frac{{C}_{1}^{2}}{4\eps } |u{|}_{H}^{2},
$$
where ${C}_{1}$ is defined by (\ref{E:C_1}).
The third term in (\ref{E:norm_G(u)_(Y,H)}) we estimate as follows
$$
\sum_{i=1}^{\infty } {\bigl| \ci u  \bigr| }_{H}^{2}
\le \sum_{i=1}^{\infty } {\| \ci  \| }_{{L}^{\infty }}^{2} |u{|}_{H}^{2}
\le {C}_{1} |u{|}_{H}^{2}  .
$$
Hence
$$
\norm{G(u)  }{\lhs (Y,H) }{2}  \le (2 + \eps -a ) {\| u \| }_{}^{2}
+ \biggl( \frac{{C}_{1}^{2}}{4\eps } + {C}_{1} \biggr) |u{|}_{H}^{2}
$$
and
\begin{eqnarray*}
& & 2 \bigl< \acal u |u \bigr> - \norm{G(u)  }{\lhs (Y,H) }{2}  \ge 2 \norm{u}{}{2}
- (2 + \eps -a ) {\| u \| }_{}^{2}
- \biggl( \frac{{C}_{1}^{2}}{4\eps } + {C}_{1} \biggr) |u{|}_{H}^{2}  \\
& & =  (a - \eps  ) {\| u \| }_{}^{2}
- \biggl( \frac{{C}_{1}^{2}}{4\eps } + {C}_{2} \biggr) |u{|}_{H}^{2} .
\end{eqnarray*}
It is sufficient to take $\eps >0 $ such that $a - \eps \in(0,2]$.
Then condition ($\mathbf{\tilde{G}}$) holds with $\eta := a -\eps $ and ${\lambda }_{0}:= \frac{{C}_{1}^{2}}{4\eps } + {C}_{1}$. \qed

\bigskip  \noindent
\bf Proof of  ($\mathbf{\tilde{G}}^\ast $). \rm
Let $b=({b}_{1}, ..., {b}_{d}) : \overline{\ocal} \to \rd $ and let $u,v \in \vcal $. Then
\begin{equation}  \label{E:differentiation}
\sum_{j=1}^{d} \frac{\partial }{\partial {x}_{j}} ({b}_{j} u)
 = \sum_{j=1}^{d} \biggl(  \frac{\partial {b}_{j} }{\partial {x}_{j}} u
    + {b}_{j} \frac{\partial u }{\partial {x}_{i}} \biggr)
 = (\diver b ) u + \sum_{j=1}^{d} {b}_{j} \frac{\partial u }{\partial {x}_{j}} .
\end{equation}
Thus using the integration by parts formula, we obtain
\begin{eqnarray*}
 \int_{\ocal} \bigl( \sum_{j=1}^{d} {b}_{j} \frac{\partial u }{\partial {x}_{j}}  \bigr) v \, dx
 = \sum_{j=1}^{d} \int_{\ocal} \frac{\partial }{\partial {x}_{j}} ({b}_{j} u) v \, dx
    - \int_{\ocal} (\diver b) u v \, dx & &  \\
 = - \sum_{j=1}^{d} \int_{\ocal} ({b}_{j} u) \frac{\partial v}{\partial {x}_{j}}  \, dx
    - \int_{\ocal} (\diver b) u v \, dx .  & &
\end{eqnarray*}
Hence using the H\"{o}lder inequality, we obtain
\begin{eqnarray*}
 & & \bigl| \int_{\ocal} \bigl( \sum_{j=1}^{d} {b}_{j} \frac{\partial u }{\partial {x}_{j}}  \bigr) v \, dx  \bigr|
  \le \| b {\| }_{{L}^{\infty }} {| u | }_{H} {\| v \| }_{V}
 + \| \diver b {\| }_{{L}^{\infty }} {| u | }_{H} {\| v \| }_{V} .
\end{eqnarray*}
Therefore the bilinear form
$$
  \hat{b} (u,v):= \int_{\ocal} \sum_{j=1}^{d} {b}_{j} u  \frac{\partial v}{\partial {x}_{j}}  \, dx ,
  \qquad u,v \in \vcal
$$
is continuous on $\vcal \times \vcal \subset H \times V$. Thus $\hat{b}$ can be uniquely extended to
the bilinear form (denoted by the same letter)
$
       \hat{b} :H \times V \to \rzecz
$
and
\begin{equation}
  |\hat{b}(u,v)| \le \bigl( \| b {\| }_{{L}^{\infty }} + \| \diver b {\| }_{{L}^{\infty }} \bigr)
  {| u | }_{H} {\| v \| }_{V} , \qquad u \in H, \quad v \in V .
\end{equation}
Hence if we define a linear map $\hat{B}$ by $\hat{B}u:= \hat{b}(u,\cdot )$, then $\hat{B}u\in V^{\prime }$
for all $u \in H$ and the following inequality holds
$$
  |\hat{B}u{|}_{V^{\prime }}
  \le \bigl( \| b {\| }_{{L}^{\infty }} + \| \diver b {\| }_{{L}^{\infty }} \bigr) {| u | }_{H} ,
  \qquad u \in H
$$
%%%%%%%%%%%%%%%%%
Using more classical notation, we can rewrite the above inequality in the following form
\begin{equation}  \label{E:hat_B_est}
 |(b \cdot \nabla ) u {|}_{V^{\prime }}
\le \bigl( \|  b {\| }_{{L}^{\infty }} + \| \diver b {\| }_{{L}^{\infty }} \bigr) \cdot {| u | }_{H} ,
 \qquad u \in H .
\end{equation}
Moreover,
\begin{equation}  \label{E:c_i_est}
   {| \ci u |}_{V^{\prime }}  \le  {\| \ci  \| }_{{L}^{\infty }} {|u|}_{H} , \qquad u \in H.
\end{equation}
Since by (\ref{E:G_def}) $G(u) \hi = \bigl( \bi \cdot \nabla  \bigr) u + \ci u  $, then
by the Schwarz inequality we get
$$
 | G(u) \hi {|}_{V^{\prime }}^{2}
   = | \bigl( \bi \cdot \nabla  \bigr) u + \ci u {|}_{V^{\prime }}^{2}
    \le 2 \bigl( | \bigl( \bi \cdot \nabla  \bigr) u  {|}_{V^{\prime }}^{2}  + | \ci u {|}_{V^{\prime }}^{2} \bigr) ,
\qquad u \in H .
$$
Thus by estimates (\ref{E:hat_B_est}) and (\ref{E:c_i_est}) we obtain
\begin{eqnarray*}
 & & \| G(u) h {\| }_{\lhs (Y,V^{\prime })}^{2}
  =  \sum_{i=1}^{\infty } | G(u) \hi {|}_{V^{\prime }}^{2}
    \le 2  \sum_{i=1}^{\infty } \Bigl( | \bigl( \bi \cdot \nabla  \bigr) u  {|}_{V^{\prime }}^{2}
     + | \ci u {|}_{V^{\prime }}^{2} \Bigr)  \\
 & & \le 2 \sum_{i=1}^{\infty }
    \bigl( 2 \|  \bi {\| }_{{L}^{\infty }}^{2} + 2 \| \diver \bi {\| }_{{L}^{\infty }}^{2} + {\| \ci  \| }_{{L}^{\infty }}^{2} \bigr)
    {|u|}_{H}^{2} , \qquad u \in H.
\end{eqnarray*}
Hence $G(u) \in \lhs (Y,V^{\prime })$ and
$$
\| G(u)  {\| }_{\lhs (Y,V^{\prime })}  \le  C \cdot  {|u|}_{H},  \qquad u \in H ,
$$
where $C=2{C}_{1}$. Thus condition ($\mathbf{\tilde{G}}^\ast $) holds. \qed

\bigskip
\noindent
\bf Proof of ($\mathbf{\tilde{G}}{}_{R}^\ast $). \rm
Let us fix $R>0$. We will proceed similarly as in the proof of ($\mathbf{\tilde{G}}^\ast $).
Let $v \in \vcal ({\ocal }_{R})$. Since the values  of $v$ on the boundary
$\partial {\ocal }_{R}$ are equal to zero,
thus using (\ref{E:differentiation}) and the integration by parts formula, we infer that
\begin{eqnarray*}
 \int_{{\ocal}_{R}} \bigl( \sum_{j=1}^{d} {b}_{j} \frac{\partial u }{\partial {x}_{j}}  \bigr) v \, dx
 = \sum_{j=1}^{d} \int_{{\ocal }_{R}} \frac{\partial }{\partial {x}_{j}} ({b}_{j} u) v \, dx
    - \int_{{\ocal }_{R}} (\diver b) u v \, dx  & & \\
 = - \sum_{j=1}^{d} \int_{{\ocal }_{R}} ({b}_{j} u) \frac{\partial v}{\partial {x}_{j}}  \, dx
    - \int_{{\ocal }_{R}} (\diver b) u v \, dx .  & &
\end{eqnarray*}
Using the H\"{o}lder inequality, we obtain
\begin{eqnarray}
& & \bigl| \int_{{\ocal }_{R}} \bigl( \sum_{j=1}^{d} {b}_{j} \frac{\partial u }{\partial {x}_{j}}  \bigr) v \, dx  \bigr|
  \le \| b {\| }_{{L}^{\infty }} {| u | }_{{H}_{{\ocal }_{R}}} {\| v \| }_{V ({\ocal }_{R})}
 + \| \diver b {\| }_{{L}^{\infty }} {| u | }_{{H}_{{\ocal }_{R}}} {\| v \| }_{V({\ocal }_{R})}
 \label{E:Holder_est_O_R} .
\end{eqnarray}
Therefore if we define a linear functional ${\hat{B}}_{R}$ by
$$
 {\hat{B}}_{R} v
 :=\int_{{\ocal}_{R}} \bigl( \sum_{j=1}^{d} {b}_{j} \frac{\partial u }{\partial {x}_{j}} \bigr) v \, dx,
 \qquad v \in \vcal ({\ocal }_{R}),
$$
we infer that it is bounded in the norm of the space $V({\ocal }_{R})$. Thus it can be uniquely extended to a linear bounded functional (denoted also by ${\hat{B}}_{R}$) on $V({\ocal }_{R})$. Moreover,
by estimate (\ref{E:Holder_est_O_R}) we have the following inequality
$$
     |{\hat{B}}_{R} {|}_{V^{\prime }({\ocal }_{R})}
    \le \bigl(  \| b {\| }_{{L}^{\infty }} + \| \diver b {\| }_{{L}^{\infty }}   \bigr)
   {| u | }_{{H}_{{\ocal }_{R}}}
$$
or equivalently
\begin{equation}  \label{E:hat_B_R_est}
\| (b \cdot \nabla ) u {\| }_{V^{\prime }({\ocal }_{R})}
\le \bigl( \|  b {\| }_{{L}^{\infty }} + \| \diver b {\| }_{{L}^{\infty }} \bigr)
 \cdot {| u | }_{{H}_{{\ocal }_{R}}} .
\end{equation}
Furthermore,
\begin{equation}  \label{E:c_i_R_est}
   {\| \ci u \| }_{V^{\prime }({\ocal }_{R})}  \le  {\| \ci  \| }_{{L}^{\infty }} {|u|}_{{H}_{{\ocal }_{R}}}.
\end{equation}
Since by (\ref{E:G_def}) $G(u) \hi = \bigl( \bi \cdot \nabla  \bigr) u + \ci u  $, thus by the Schwarz inequality we get
$$
 | G(u) \hi {|}_{V^{\prime }({\ocal }_{R})}^{2}
   = | \bigl( \bi \cdot \nabla  \bigr) u + \ci u {|}_{V^{\prime }({\ocal }_{R})}^{2}
    \le 2 \bigl( | \bigl( \bi \cdot \nabla  \bigr) u  {|}_{V^{\prime }({\ocal }_{R})}^{2}
   + | \ci u {|}_{V^{\prime }({\ocal }_{R})}^{2} \bigr) .
$$
Hence by estimates (\ref{E:hat_B_R_est}) and (\ref{E:c_i_R_est}) we obtain
\begin{eqnarray*}
& &  \| G(u) h {\| }_{\lhs (Y,V^{\prime }({\ocal }_{R}))}^{2}
  =  \sum_{i=1}^{\infty } | G(u) \hi {|}_{V^{\prime }({\ocal }_{R})}^{2}
    \le 2  \sum_{i=1}^{\infty } \biggl( | \bigl( \bi \cdot \nabla  \bigr) u  {|}_{V^{\prime }({\ocal }_{R})}^{2}
     + | \ci u {|}_{V^{\prime }({\ocal }_{R})}^{2} \biggr)  \\
 & & \le 2 \sum_{i=1}^{\infty }
    \bigl( 2 \|  \bi {\| }_{{L}^{\infty }}^{2} + 2 \| \diver \bi {\| }_{{L}^{\infty }}^{2} + {\| \ci  \| }_{{L}^{\infty }}^{2} \bigr)
    {|u|}_{{H}_{{\ocal }_{R}}}^{2} .
\end{eqnarray*}
Therefore $G(u) \in \lhs (Y,V^{\prime }({\ocal }_{R}))$ and
$$
\| G(u)  {\| }_{\lhs (Y,V^{\prime }({\ocal }_{R}))}  \le {C}_{R} \cdot  {|u|}_{{H}_{{\ocal }_{R}}},
$$
where ${C}_{R}=2{C}_{1}$.
Thus condition ($\mathbf{\tilde{G}}{}_{R}^\ast $) holds. \qed

%%%%%%%%%%%%%%%%%%%%%%%%%%%%%%%%%%%%%%%%%%%%%%%%%%%%%%%%%%%%%%%%%%%%%%%%%%%%%%%%%%%%%%%%%%%%%%%%%%
%%%%%%%%%%%%%%%%%%%%%%%%%%%%%%%%%%%%%%%%%%%%%%%%%%%%%%%%%%%%%%%%%%%%%%%%%%%%%%%%%%%%%%%%%%%%%%%%%%

%%%%%%%%%%%%%%%%%%%%%%%%%%%%%%%%%%%%%%%%%%%%%%%%%%%%%%%%%%%%%%%%%%%%%%%%%%%%%%%%%%%%%%%%%%%%%%%%%%%%%
\section{\bf 2D stochastic Navier-Stokes equations}  \label{S:2D_NS}
%%%%%%%%%%%%%%%%%%%%%%%%%%%%%%%%%%%%%%%%%%%%%%%%%%%%%%%%%%%%%%%%%%%%%%%%%%%%%%%%%%%%%%%%%%%%%%%%%%%%%

\bigskip \noindent
In the two-dimensional case the martingale solution of the stochastic Navier-Stokes equation, given by Theorem \ref{T:existence}, has stronger
regularity properties. We will prove that $\hat{\p }$-a.s. the trajectories are equal almost everywhere to an $H$-valued continuous functions defined on $[0,T]$. Similarly to Capi\'{n}ski \cite{Capinski_1993} we will prove that the solutions are pathwise unique.
Moreover, using the results due to Ondrej\'{a}t \cite{Ondrejat_2004}
 we will show the existence of strong solutions and  uniqueness in law.

\bigskip  \noindent
It is well known that if $d=2$ then the following inequality holds,
see \cite[Lemma III.3.3]{Temam79},
\begin{equation} \label{E:2D_norm_L^4_est}
   \norm{u}{{L}^{4}}{} \le {2}^{\frac{1}{4}} {|u|}_{H}^{\frac{1}{2}} \norm{u}{}{\frac{1}{2}},
 \qquad u \in {H}^{1}_{0}.
\end{equation}
In the following Lemma we recall basic properties of the form $b$, defined by
(\ref{E:form_b}), valid in the two-dimensional case.

\bigskip
\begin{lemma} \label{L:III.3.4_Temam'79}\rm (Lemma III.3.4 in \cite{Temam79}) \it
If $d=2$ then
\begin{equation}
  |b(u,v,w)| \le {2}^{\frac{1}{2}}\, |u{|}^{\frac{1}{2}} \norm{u}{}{\frac{1}{2}} \cdot \norm{v}{}{}
  \cdot |w{|}^{\frac{1}{2}} \norm{w}{}{\frac{1}{2}} , \qquad u,v,w \in V.
 \label{E:III.3.53}
\end{equation}
If $u \in {L}^{2}(0,T;V) \cap {L}^{\infty }(0,T;H)$ then $B(u) \in {L}^{2}(0,T;V^{\prime })$ and
\begin{equation}
  \norm{B(u)}{{L}^{2}(0,T;V^{\prime })}{} \le {2}^{\frac{1}{2}} \, |u{|}_{{L}^{\infty }(0,T;H)}
   \norm{u}{{L}^{2}(0,T;V)}{} .
 \label{E:III.3.54}
\end{equation}
\end{lemma}

\bigskip

\subsection{Regularity properties}

\bigskip
\begin{lemma} \label{L:2D_NS_regularity}
Let $d=2$ and assume that conditions (A.1)-(A.3) are satisfied. Let
$(\hat{\Omega }, \hat{\fcal }, \hat{\fmath }, \hat{\p }, u)$ be a martingale solution
of problem  (\ref{E:NS})
such that
\begin{equation}  \label{E:Th_2D_u_est}
\hat{\e } \Bigl[ \sup_{t \in [0,T]} |u(t){|}_{H}^{2} + \int_{0}^{T} \norm{u(t)}{}{2} \, dt  \Bigr] < \infty .
\end{equation}
Then for $\hat{\p }$-almost all $\omega \in \hat{\Omega }$ the trajectory $u(\cdot , \omega )$ is equal almost everywhere to a continuous $H$-valued function defined on $[0,T]$.
\end{lemma}

\bigskip  \noindent
\begin{proof}
If $u$ be  a martingale solution of problem (\ref{E:NS})
then,  in particular,
$u \in  \ccal \bigl( [0,T], {H}_{w} \bigr) \cap {L}^{2}(0,T;V )$, $\hat{\p }$-a.s. and
\begin{equation} \label{E:2D_NS}
  u(t) = {u}_{0} - \int_{0}^{t}\bigl[ \acal u(s) + B (u(s)) \bigr] \, ds + \int_{0}^{t}f(s)\, ds
    + \int_{0}^{t}G(u(s)) \, d \hat{W}(s), \quad t \in [0,T].
\end{equation}
Let us consider the following "shifted" Stokes equations
\begin{equation}  \label{E:2D_Stokes}
  z(t) = - \int_{0}^{t} Az(s) \, ds + \int_{0}^{t}G(u(s)) \, d \hat{W}(s),
\end{equation}
where  $A$ is the operator defined by
(\ref{E:op_A}).
Since $A$ satisfies condition (\ref{E:op_A_coerciv}) and
$u$ satisfies inequality (\ref{E:Th_2D_u_est}), by Theorem 1.3 in \cite{Pardoux79} we infer that
equation (\ref{E:2D_Stokes}) has a unique progressively measurable solution $z$ such that $\hat{\p }$-a.s. $z \in \ccal ([0,T];H)$ and
\begin{equation}  \label{E:2D_z_est}
\hat{\e } \Bigl[ \sup_{t \in [0,T]} |z(t){|}_{H}^{2} + \int_{0}^{T} \norm{z(t)}{}{2} \, dt  \Bigr] < \infty .
\end{equation}
Let
$$
    v(t):= u(t)-z(t), \qquad t \in [0,T] .
$$
For $\hat{\p }$-almost all $\omega \in \hat{\Omega }$ the function $v= v(\cdot ,\omega )$ is a weak solution of  the following deterministic equation
\begin{equation} \label{E:2D_NS_shift}
 \frac{d v(t)}{dt} =   - A v(t) +v(t)+z(t) -  B (v(t)+z(t))  + f(t) .
\end{equation}
Let $\omega \in \hat{\Omega }$ be such that $u(\cdot , \omega ) \in {L}^{2}(0,T;V) \cap \ccal ([0,T];{H}_{w})$ and
$ z(\cdot ,\omega ) \in {L}^{2}(0,T,V) \cap \ccal ([0,T];H)$
Let $\tilde{v} \in {L}^{2}(0,T;V) \cap \ccal ([0,T];H)$ be the  unique  solution of equation (\ref{E:2D_NS_shift})  with the initial condition $\tilde{v}(0)={u}_{0}$ whose existence is ensured by Theorem D.2. %\ref{T:App_D_existence_unique}.
By the uniqueness, we obtain for almost all $t \in [0,T]$
$$
    \tilde{v}(t) = u(t) - z(t)
$$
Put
$$
    \hat{u} (t) := \tilde{v}(t) + z(t) , \qquad t \in [0,T].
$$
Then $\hat{u} \in \ccal ([0,T];H)$ and   $u(t)= \hat{u}(t)$ for almost all $t \in [0,T]$.
This completes the proof of the lemma.
\end{proof}

\bigskip
\subsection{Uniqueness and strong solutions}

\bigskip  \noindent
Let us recall that by assumption (A.3) the mapping $G:V \to \lhs (Y,H)$ is Lipschitz continuous, i.e. for some  $L>0 $ the following inequality holds
\begin{eqnarray}
 \norm{ G({u}_{1}(s))-G({u}_{2}(s)) }{\lhs (Y,H)}{}
 \le L \norm{{u}_{1}(s)-{u}_{2}(s)}{V}{} , \quad s \in [0,T].  \label{E:2D_uniq_Lipsch}
\end{eqnarray}
In the following lemma we will prove that in the case when $d=2$ the solutions of
problem  (\ref{E:NS}) are pathwise unique. The proof is similar to that of Theorem 3.2 in \cite{Capinski_1993} and uses the Schmalfuss idea of application of the It\^{o} formula for appropriate function, see \cite{Schmalfuss_1997}.

\bigskip  \noindent
\begin{lemma}  \label{L:2D_pathwise_uniqueness}
Let $d=2$ and assume that conditions (A.1)-(A.3) are satisfied.
Moreover, assume that $L <\sqrt{2}$.
If ${u}_{1}$ and  ${u}_{2}$ are two  solutions of problem  (\ref{E:NS})
defined on the same filtered probability space
$(\hat{\Omega }, \hat{\fcal }, \hat{\fmath }, \hat{\p })$ then $\hat{\p }$-a.s. for all $t \in [0,T]$,
${u}_{1}(t)={u}_{2}(t)$.
\end{lemma}

\begin{proof}
Let
$$
     U := {u}_{1} - {u}_{2} .
$$
Then $U$ satisfies the following equation
$$
   dU(t) = - \bigl\{ \acal U(t) + \bigl[ B({u}_{1}(t))-B({u}_{2}(t)) \bigr] \bigr\}  \, dt
    +  \bigl[ G({u}_{1}(t))-G({u}_{2}(t))  \bigr] \, dW(t) , \qquad t \in [0,T]
$$
Let us define the stopping time
\begin{equation}  \label{E:2D_uniq_stopping}
   {\tau }_{N} := T \wedge \inf \{ t \in [0,T] : |U(t){|}_{H}^{2} >N  \} ,
\qquad N \in \nat  .
\end{equation}
Since $\hat{\e }\bigl[ \sup_{t \in [0,T]} |U(t){|}_{H}^{2} \bigr] < \infty $, $\hat{\p }$-a.s.  $\lim_{N \to \infty }{\tau }_{N} = T $.
Let $r(t):= a \int_{0}^{t} \norm{{u}_{2}(s)}{}{2} \, ds $, $t \in [0,T]$, where $a$ is a positive constant. We apply the It\^{o} formula to the function
$$
 F(t,x)={e}^{-r(t)}|x{|}_{H}^{2},  \qquad (t,x)\in [0,T]\times H.
$$
Since $\frac{\partial F}{\partial t} = - r^\prime (t) {e}^{-r(t)}|x{|}_{H}^{2}$ and
$\frac{\partial F}{\partial x} =  2{e}^{-r(t)}x$, we obtain for all $t \in [0,T]$
\begin{eqnarray}
 & &  {e}^{-r(t\wedge {\tau }_{N})}|U(t\wedge {\tau }_{N}){|}_{H}^{2} \nonumber \\
 & & =
  \int_{0}^{t\wedge {\tau }_{N}} {e}^{-r(s)} \bigl\{ -r^\prime (s) |U(s){|}_{H}^{2}
 - 2 \dual{\acal U(s) + B({u}_{1}(s))-B({u}_{2}(s)) }{U(s)}{} \bigr\} \, ds \nonumber \\
  & & \quad  + \frac{1}{2}\int_{0}^{t\wedge {\tau }_{N}} \tr \Bigl[ \bigl( G({u}_{1}(s))-G({u}_{2}(s))  \bigr)
 \frac{{\partial }^{2}F }{\partial {x}^{2}} \bigl( G({u}_{1}(s))-G({u}_{2}(s)) {\bigr) }^{\ast } \Bigr] \, ds \nonumber \\
 & & \quad + 2 \int_{0}^{t\wedge {\tau }_{N}} {e}^{-r(s)} \dual{G({u}_{1}(s))-G({u}_{2}(s))}{U(s)\, dW(s)}{}
 \nonumber \\
& &  \le
 \int_{0}^{t\wedge {\tau }_{N}} {e}^{-r(s)} \bigl\{ -r^\prime (s) |U(s){|}_{H}^{2} -2 \norm{U(s)}{}{2}
 - 2 \dual{ B({u}_{1}(s))-B({u}_{2}(s)) }{U(s)}{} \bigr\} \, ds \nonumber \\
  & & \quad  + \int_{0}^{t\wedge {\tau }_{N}} {e}^{-r(s)} \norm{ G({u}_{1}(s))-G({u}_{2}(s)) }{\lhs (Y,H)}{2} \, ds   \nonumber \\
& & \quad  + 2 \int_{0}^{t\wedge {\tau }_{N}} {e}^{-r(s)} \dual{G({u}_{1}(s))-G({u}_{2}(s))}{U(s)\, dW(s)}{} .
  \label{E:2D_uniq_Ito}
\end{eqnarray}
We have
$$
   B({u}_{1}(s)) - B({u}_{2}(s)) = B({u}_{1}(s), U(s)) + B(U(s),{u}_{2}(s)), \qquad s \in [0,T].
$$
Thus by   (\ref{E:wirowosc_b})
$$
  \dual{B({u}_{1}(s))-B({u}_{2}(s))}{U(s)}{} = \dual{B(U(s),{u}_{2}(s))}{U(s)}{}
$$
and hence
\begin{eqnarray*}
 \bigl| \dual{B({u}_{1}(s))-B({u}_{2}(s))}{U(s)}{} \bigr|
  \le \sqrt{2} |U(s){|}_{H} \norm{U(s)}{}{} \norm{{u}_{2}(s)}{}{}, \quad s \in [0,T].
\end{eqnarray*}
Therefore for every $\eps > 0 $ there exists ${C}_{\eps }>0 $ such that
\begin{equation} \label{E:2D_uniq_B_est}
 2 \bigl| \dual{B({u}_{1}(s))-B({u}_{2}(s))}{U(s)}{} \bigr|
  \le \eps \norm{U(s)}{}{2} + {C}_{\eps } \norm{{u}_{2}(s)}{}{2} |U(s){|}_{H}^{2} , \qquad s \in [0,T] .
\end{equation}
Putting  $a:= {C}_{\eps }$,  we obtain
\[
  -r^\prime (s) |U(s){|}_{H}^{2} + {C}_{\eps } \norm{{u}_{2}(s)}{}{2} |U(s){|}_{H}^{2}
  = -a \norm{{u}_{2}(s)}{}{2} |U(s){|}_{H}^{2} + {C}_{\eps } \norm{{u}_{2}(s)}{}{2} |U(s){|}_{H}^{2}=0,
  \quad s \in [0,T]
\]
and hence by \eqref{E:2D_uniq_Ito}  and  \eqref{E:2D_uniq_B_est}
\begin{eqnarray}
 & & {e}^{-r(t\wedge {\tau }_{N})}{|U(t\wedge {\tau }_{N})|}_{H}^{2}
\le
(-2+\eps ) \int_{0}^{t\wedge {\tau }_{N}} {e}^{-r(s)}  \norm{U(s)}{}{2} \, ds \nonumber \\
  & &  + \int_{0}^{t\wedge {\tau }_{N}} {e}^{-r(s)} \norm{ G({u}_{1}(s))-G({u}_{2}(s)) }{\lhs (Y,H)}{2} \, ds  \nonumber \\
& &  + 2 \int_{0}^{t\wedge {\tau }_{N}} {e}^{-r(s)} \dual{G({u}_{1}(s))-G({u}_{2}(s))}{U(s)\, dW(s)}{},
\qquad t \in [0,T] .  \label{E:2D_NS_Ito_est}
\end{eqnarray}
By \eqref{E:2D_NS_Ito_est} and \eqref{E:2D_uniq_Lipsch} we obtain
\begin{eqnarray} \label{E:2D_uniq_est_0}
 & &  {e}^{-r(t\wedge {\tau }_{N})}{|U(t\wedge {\tau }_{N})|}_{H}^{2}
 +  (2- \eps - {L}^{2} ) \int_{0}^{t\wedge {\tau }_{N}} {e}^{-r(s)} \norm{U(s)}{}{2} \, ds
\le {L}^{2}\int_{0}^{t\wedge {\tau }_{N}} {e}^{-r(s)} {|U(s)|}_{H}^{2} \, ds  \nonumber \\
& & 
+2 \int_{0}^{t\wedge {\tau }_{N}} {e}^{-r(s)} \dual{G({u}_{1}(s))-G({u}_{2}(s))}{U(s)\, dW(s)}{} , \quad t \in [0,T].
\end{eqnarray}
Let us choose $\eps >0 $  such that $2-\eps - {L}^{2} >0$.
Then by (\ref{E:2D_uniq_est_0}), in particular, we have
\begin{eqnarray} \label{E:2D_uniq_est}
&&  {e}^{-r(t\wedge {\tau }_{N})}|U(t\wedge {\tau }_{N}){|}_{H}^{2}
\le {L}^{2}\int_{0}^{t\wedge {\tau }_{N}} {e}^{-r(s)} {|U(s)|}_{H}^{2} \, ds  \nonumber \\
&& + 2 \int_{0}^{t\wedge {\tau }_{N}} {e}^{-r(s)} \dual{G({u}_{1}(s))-G({u}_{2}(s))}{U(s)\, dW(s)}{} , \quad t \in [0,T] .
\end{eqnarray}
Let us denote
\[
{\mu } (t):= \int_{0}^{t} {e}^{-r(s)} \dual{G({u}_{1}(s))-G({u}_{2}(s))}{U(s)\, dW(s)}{} , \qquad t \in [0,T]
\]
and
\[
    {\mu }_{N} (t) := \mu (t \wedge {\tau }_{N}), \qquad t \in [0,T], \quad N \in \nat  .
\]
Since by \eqref{E:2D_uniq_Lipsch} and \eqref{E:2D_uniq_stopping}
\begin{eqnarray*}
 \hat{\e} \Bigl[ \int_{0}^{T} \ind{[0,{\tau }_{N}]} {|U(s)|}_{H}^{2}\norm{G({u}_{1}(s))-G({u}_{2}(s))}{\lhs (Y,H)}{2} \, ds \Bigr]
 \le N {L}^{2} \hat{\e} \Bigl[ \int_{0}^{T} \norm{{u}_{1}(s)-{u}_{2}(s)}{V}{2} \, ds \Bigr] < \infty ,
\end{eqnarray*}
for each $N \in \nat $ the process $\{ {\mu }_{N} (t) \} $ is a martingale.
Hence   by \eqref{E:2D_uniq_est_0}
  and the martingale property we obtain for all $t \in [0,T]$
\begin{eqnarray}
   \hat{\e } \bigl[ {e}^{-r(t\wedge {\tau }_{N})}{|U(t\wedge {\tau }_{N})|}_{H}^{2} \bigr]
  +    (2- \eps - {L}^{2} ) \int_{0}^{t\wedge {\tau }_{N}} \hat{\e } \bigl[{e}^{-r(s)} \norm{U(s)}{}{2} \bigr] \, ds 
 \le {L}^{2} \int_{0}^{t\wedge {\tau }_{N}}  \hat{\e } \bigl[{e}^{-r(s)} {|U(s)|}_{H}^{2} \bigr] \, ds. 
 \label{E:2D_uniq_Ito_final}
\end{eqnarray}
In particular, 
for all $t \in [0,T]$
\begin{eqnarray*}
  \hat{\e } \bigl[ {e}^{-r(t\wedge {\tau }_{N})}{|U(t\wedge {\tau }_{N})|}_{H}^{2} \bigr]
 \le {L}^{2} \int_{0}^{t\wedge {\tau }_{N}}  \hat{\e } \bigl[{e}^{-r(s)} {|U(s)|}_{H}^{2} \bigr] \, ds. 
\end{eqnarray*}
By the Gronwall lemma we infer that for all $t \in [0,T]$,
\begin{eqnarray}
\hat{\e } \bigl[ {e}^{-r(t\wedge {\tau }_{N})}{|U(t\wedge {\tau }_{N})|}_{H}^{2} \bigr] =0 
\label{E:2D_uniq_value=0}
\end{eqnarray} 
and hence, in particular,
$\int_{0}^{t\wedge {\tau }_{N}}  \hat{\e } \bigl[{e}^{-r(s)} {|U(s)|}_{H}^{2} \bigr] \, ds =0$
By \eqref{E:2D_uniq_Ito_final} and \eqref{E:2D_uniq_value=0}, we infer that 
for all $t \in [0,T]$
\begin{equation} \label{E:2D_uniq_int=0}
 \hat{\e } \Bigl[ \int_{0}^{t\wedge {\tau }_{N}} {e}^{-r(s)} \norm{U(s)}{}{2} \, ds  \Bigr] =0 .
\end{equation}
By the Schwarz, the Burkholder-Davis-Gundy inequalities  \eqref{E:2D_uniq_value=0} and \eqref{E:2D_uniq_int=0}, we obtain
\begin{eqnarray}
& &  \hat{\e } \bigl[ \sup_{t \in [0,T]} |{\mu }_{N} (t)| \bigr]
 \le C \hat{\e } \Bigl[ \Bigl( \int_{0}^{T\wedge {\tau }_{N}} {e}^{-2r(s)}|U(s){|}_{H}^{2} \norm{G({u}_{1}(s))-G({u}_{2}(s))}{\lhs (Y,H)}{2} \, ds {\Bigr) }^{\frac{1}{2}} \Bigr] \nonumber \\
 & & \le C L\hat{\e } \Bigl[ \Bigl( \sup_{s \in [0,T]}{e}^{-r(s\wedge {\tau }_{N})}{|U(s\wedge {\tau }_{N})|}_{H}^{2}
 \int_{0}^{T\wedge {\tau }_{N}} {e}^{-r(s)} \norm{U(s)}{V}{2} \, ds {\Bigr) }^{\frac{1}{2}} \Bigr]
  \nonumber \\
& & \le \frac{1}{2}   \hat{\e }\bigl[ \sup_{t \in [0,T]} {e}^{-r(t\wedge {\tau }_{N})}|U(t\wedge {\tau }_{N}){|}_{H}^{2} + \tilde{C} \hat{\e } \Bigl[
 \int_{0}^{T\wedge {\tau }_{N}} {e}^{-r(s)} \norm{U(s)}{V}{2} \, ds  \Bigr] \nonumber \\
& &  = \frac{1}{2}   \hat{\e }\bigl[ \sup_{t \in [0,T]} {e}^{-r(t\wedge {\tau }_{N})}|U(t\wedge {\tau }_{N}){|}_{H}^{2}  .
 \label{E:2D_uniq_BDG_est}
\end{eqnarray}
By \eqref{E:2D_uniq_est} and \eqref{E:2D_uniq_BDG_est}, we obtain
\[
  \hat{\e } \bigl[ \sup_{t \in [0,T]} {e}^{-r(t\wedge {\tau }_{N})}|U(t\wedge {\tau }_{N}){|}_{H}^{2} =0.
\]
Since $\hat{\p }$-a.s.  $\lim_{N \to \infty } {\tau }_{N} = T $ and   $\hat{\e }\bigl[ \int_{0}^{T}
 \norm{{u}_{2}(t)}{}{} \, dt \bigr] < \infty $ we infer that
$\hat{\p }$-a.s. for all $t \in [0,T]$, $U(t)=0$.
The proof of the lemma is thus complete.
\end{proof}

\bigskip
\begin{definition}  \rm
We say that problem  (\ref{E:NS})
has a \bf strong solution \rm iff for every stochastic basis $\bigl( \Omega , \fcal , \mathbb{F} , \p  \bigr) $ with a filtration $\mathbb{F}={\{ \ft \} }_{t \ge 0}$ and every cylindrical  Wiener process $ W(t)$ in a separable Hilbert space $Y$ defined on this stochastic basis
there exists
a progressively measurable process
$u: [0,T] \times \Omega \to H$ with $\p  $ - a.e. paths
$$
  u(\cdot , \omega ) \in \ccal \bigl( [0,T], {H}_{w} \bigr)
   \cap {L}^{2}(0,T;V )
$$
such that for all $ t \in [0,T] $ and all $v \in \vcal $:
\begin{eqnarray*}
 \ilsk{u(t)}{v}{H} +  \int_{0}^{t} \dual{\acal u(s)}{v}{} \, ds
+ \int_{0}^{t} \dual{B(u(s),u(s))}{v}{} \, ds & & \nonumber \\
  = \ilsk{{u}_{0}}{v}{H} +  \int_{0}^{t} \dual{f(s)}{v}{} \, ds
 + \Dual{\int_{0}^{t} G(u(s))\, d\hat{W}(s)}{v}{}  & &
\end{eqnarray*}
the identity holds $\hat{\p }$ - a.s.
\end{definition}

\bigskip \noindent
Let us recall two basic concepts of uniqueness of the solution, i.e. pathwise uniqueness and uniqueness in law, see \cite{Ikeda+Watanabe_1981},
\cite{Ondrejat_2004}.

\bigskip
\begin{definition}  \rm
We say that solutions of problem  (\ref{E:NS})
are \bf pathwise unique \rm iff the following condition holds:
\begin{eqnarray*}
& & \mbox{\it if } (\Omega ,\fcal , \fmath ,\p ,W,{u}^{i}), \, \,  i=1,2, \mbox{ \it  are  such solutions
  of problem  (\ref{E:NS})
  that } {u}^{i}(0) = {u}_{0}, \, \, i =1,2, \\
& &  \mbox{\it then $\p $-a.s. for all } t \in [0,T], \, \, \, {u}^{1}(t) = {u}^{2}(t).
\end{eqnarray*}
\end{definition}

\bigskip
\begin{definition}  \rm
We say that solutions of problem  (\ref{E:NS})
are \bf  unique in law \rm iff the following condition holds:
\begin{eqnarray*}
& & \mbox{\it if } ({\Omega }^{i},{\fcal }^{i}, {\fmath }^{i}, {\p }^{i},{W}^{i},{u}^{i}), \, \,  i=1,2, \mbox{ \it  are  such solutions
  of problem  (\ref{E:NS})
  that }  \\
& & {u}^{i}(0) = {u}_{0}, \, \, i =1,2, \mbox{ \it then  } {Law}_{{\p }^{1}}({u}^{1}) = {Law}_{{\p }^{2}}({u}^{2})
\end{eqnarray*}
where ${Law}_{{\p }^{1}}({u}^{i})$, $i=1,2$, are probability measures on $\ccal ([0,T],{H}_{w})\cap {L}^{2}(0,T;V)$.
\end{definition}

\bigskip
\begin{cor}
Let $d=2$ and assume that conditions (A.1)-(A.3) are satisfied.
Moreover, assume that the Lipschitz constant $L$ in (\ref{E:2D_uniq_Lipsch}) satisfies condition $L <\sqrt{2}$.
\begin{itemize}
\item[1)] There exists a pathwise unique strong solution of problem  (\ref{E:NS})
\item [2)] Moreover, if $\bigl( \Omega , \fcal , \mathbb{F} , \p  , W, u\bigr) $ is a strong solution
of problem  (\ref{E:NS})
then for $\p $-almost all $\omega \in \Omega $ the trajectory $u(\cdot , \omega )$ is equal almost everywhere to a continuous $H$-valued function defined on $[0,T]$.
\item[3)] The martingale solution  of problem  (\ref{E:NS}) is unique in law.
\end{itemize}
\end{cor}

\bigskip
\begin{proof}
Since by Theorem  \ref{T:existence}
there exists a martingale solution and by Lemma \ref{L:2D_pathwise_uniqueness} it is pathwise unique,
assertions 1) and 3) follow from Theorems 2 and 12.1 in \cite{Ondrejat_2004}.
Assertion 2) is a direct consequence of Lemma \ref{L:2D_NS_regularity}.
\end{proof}

%%%%%%%%%%%%%%%%%%%%%%%%%%%%%%%%%%%%%%%%%%%%%%%%%%%%%%%%%%%%%%%%%%%%%%%%%%%%%%%%%%%%%%%%%%%%%%%
%%%%%%%%%%%%%%%%%%%%%%%%%%%%%%%%%%%%%%%%%%%%%%%%%%%%%%%%%%%%%%%%%%%%%%%%%%%%%%%%%%%%%%%%%%%%%%%
\bigskip
\appendix{\bf Appendix A}
%%%%%%%%%%%%%%%%%%%%%%%%%%%%%%%%%%%%%%%%%%%%%%%%%%%%%%%%%%%%%%%%%%%%%%%%%%%%%%%%%%%%%%%%%%%%%%%
%%%%%%%%%%%%%%%%%%%%%%%%%%%

\bigskip  \noindent
Let us assume that the mapping $G: V \to \lhs (Y,H)$ satisfies the following inequality
\begin{equation}
    2\dual{\acal u}{u}{}-\norm{G(u)}{\lhs (Y,H)}{2}
  \ge \eta \norm{u}{}{2} -{\lambda }_{0}{|u|}_{H}^{2} - \rho ,\qquad u\in V   \tag{G}
 \end{equation}
for some constants $\eta \in (0,2]$, ${\lambda }_{0} \geq 0$ and $\rho \in\mathbb{R} $.

\bigskip  \noindent
 Since $\norm{u}{}{}:= \norm{\nabla u}{{L}^{2}}{}$ and  $\dual{\acal u}{u}{} = \dirilsk{u}{u}{}:= \ilsk{\nabla u}{\nabla u}{{L}^{2}}$,
$$2\dual{\acal u}{u}{}-\eta \norm{u}{}{2}=(2-\eta ) \norm{u}{}{2}.$$
Hence  inequality (G) can be written equivalently in the following form
\begin{equation} \label{A:G_2'}
 \norm{G(u )}{\lhs (Y,H)}{2}
  \le (2-\eta ) \norm{u}{}{2}+{\lambda }_{0}{|u|}_{H}^{2}+\rho , \qquad u \in V . \tag{G$^\prime$}
\end{equation}
The following proof of Lemma \ref{L:Galerkin_estimates } is standard, see \cite{Flandoli+Gatarek_1995}. However, we provide all details  to indicate the importance of the assumption (\ref{eqn-p_cond}) on $p$.

\bigskip  \noindent
\bf Proof of estimates (\ref{E:H_estimate}), (\ref{E:HV_estimate}) and (\ref{E:V_estimate})
in Lemma \ref{L:Galerkin_estimates } under the assumption (G).\rm

\bigskip  \noindent
Let $p $ satisfy condition (\ref{eqn-p_cond}), i.e.
\begin{eqnarray*}
& & p\in  \bigl[ 2, 2+ \frac{\eta }{2-\eta } \bigr) \quad \mbox{ if $\quad \eta \in(0,2)$  } \\
& & p\in [2, \infty )  \quad \mbox{ if $\quad \eta =2 $}.
\end{eqnarray*}
We apply the It\^{o} formula to the function $F(x) = {|x|}^{p}:= |x{|}_{H}^{p}$, $x \in H$. Since
$$
 F^\prime(x)=d_xF= \frac{\partial F}{\partial x} = p \cdot {|x|}^{p-2} \cdot x  ,
\qquad \|  F^{\prime\prime}(x) \|= \bigl\|  d^2_xF \bigr\|= \biggl\| \frac{{\partial }^{2} F}{\partial  {x}^{2}} \biggr\| \le  p (p-1) \cdot {|x|}^{p-2} ,
\qquad x \in H,
$$
we have the following equalities
\begin{eqnarray*}
& & d \bigl[ {|\un (t)|}^{p}  \bigr]
  =  \Bigl[ p \,  {|\un (t)|}^{p-2} \dual{\un (t)}{- \acal \un (t)
       -  \Bn \bigl( \un (t) \bigr) + \Pn f(t)}{}  \\
 & & + \frac{1}{2} \tr
 \bigl[ F^{\prime\prime}( \un (t)\bigl(\Pn G (\un (t)),\Pn G (\un (t)) \bigr ) \bigr]
\Bigr]
     \, dt
  + p \,  {|\un (t)|}^{p-2} \dual{\un (t)}{ G ( \un (t) ) \, d W(t) }{} \\
& & =  \Bigl[ - p \, {|\un (t)|}^{p-2} \norm{\un (t)}{}{2}
   + p \, {|\un (t)|}^{p-2} \dual{\un (t)}{  f(t)}{}
   + \frac{1}{2} \tr \bigl[
   F^{\prime\prime}( \un (t)\bigl(\Pn G (\un (t)),\Pn G (\un (t)) \bigr )
          \bigr] \Bigr] \, dt \\
& & + p \,  {|\un (t)|}^{p-2} \dual{\un (t)}{ G ( \un (t) ) \, d W(t) }{} , \quad t \in [0,T].
\end{eqnarray*}
Thus
\begin{eqnarray*}
& & d \bigl[ {|\un (t)|}^{p}  \bigr]  + p \, {|\un (t)|}^{p-2} \norm{\un (t)}{}{2}
  \, \, \le   \, \, p \, {|\un (t)|}^{p-2} \dual{f}{\un (t)}{} \, dt  \\
& & \qquad  + \frac{1}{2} p(p-1) \, {|\un (t)|}^{p-2} \cdot \norm{\Pn G(\un (t))}{\lhs (Y,H)}{2}   \, dt \\
& & \qquad + p \,  {|\un (t)|}^{p-2} \dual{\un (t)}{ G ( \un (t) ) \, d W(t) }{} \quad t \in [0,T].
\end{eqnarray*}
By (\ref{A:G_2'})
$$
  \norm{\Pn G(\un (t))}{\lhs (Y,H)}{2}   \le (2- \eta ) \, \norm{\un (t)}{}{2}
  + {\lambda }_{0} {|\un (t)|}^{2} + \rho .
$$
Thus
\begin{eqnarray*}
& & d \bigl[ {|\un (t)|}^{p}  \bigr]  + p \, {|\un (t)|}^{p-2} \norm{\un (t)}{}{2}
  \, \, \le   \, \, p \, {|\un (t)|}^{p-2} \dual{f(t)}{\un (t)}{} \, dt  \\
& & \qquad  + \frac{1}{2} p(p-1) \, {|\un (t)|}^{p-2} \cdot \bigl[ (2- \eta ) \, \norm{\un (t)}{}{2}
  + {\lambda }_{0} {|\un (t)|}^{2} + \rho  \bigr]  \, dt \\
& & \qquad + p \,  {|\un (t)|}^{p-2} \dual{\un (t)}{ G ( \un (t) ) \, d W(t) }{} .
\end{eqnarray*}
Moreover, by (\ref{E:norm_V}) and the  Young  inequality with exponents $2,\frac{2p}{p-2}$ and $p$, we obtain
\begin{eqnarray*}
  \,  {|\un (t)|}^{p-2} \dual{f(t)}{\un (t)}{}
 &\leq &   \,  {|\un (t)|}^{p-2} \norm{\un (t)}{V}{}  \, {|f(t)|}_{V^\prime } \\
 & = & {|\un (t)|}^{p-2}   {({|\un (t)|}^{2}+\norm{\un (t)}{}{2})}^{\frac{1}{2}}
   \, {|f(t)|}_{V^\prime }
 \\
 &\leq &
 \frac{\eps}{2} ({|\un (t)|}^{2}+\norm{\un (t)}{}{2})\,  {|\un (t)|^{p-2}+(\frac12-\frac1p){|\un (t)|}^{p}+\frac{\eps^{-p/2}}{p}\, {|f(t)|}^p_{V^\prime } }
\\
 &\leq &
  \frac{\eps}{2}  \norm{\un (t)}{}{2}\,  {|\un (t)|^{p-2}+(\frac{1+\eps}{2}-\frac1p){|\un (t)|}^{p}+\frac{\eps^{-p/2}}{p}\, {|f(t)|}^p_{V^\prime } }
\end{eqnarray*}
Hence
\begin{eqnarray*}
 d \bigl[ {|\un (t)|}^{p}  \bigr] &+& \bigl[ p - p \frac{\eps}{2} - \frac{1}{2} p(p-1) (2- \eta ) \,   \bigr]
  \, {|\un (t)|}^{p-2} \norm{\un (t)}{}{2} \\
&    \le &   (\frac{p(1+\eps)}2-1){|\un (t)|}^{p}+\eps^{-p/2}\, {|f(t)|}^p_{V^\prime }
   + \frac{1}{2} p(p-1) \, {|\un (t)|}^{p-2} \cdot \bigl[
  {\lambda }_{0} {|\un (t)|}^{2} + \rho  \bigr]  \, dt \\
 & &\hspace{+2truecm}\lefteqn{ + p \,  {|\un (t)|}^{p-2} \dual{\un (t)}{ G ( \un (t) ) \, d W(t) }{} }
\\
&=&\big(\frac{\lambda_0}{2}p(p-1)+\frac{p(1+\eps)}2-1\big) \, {|\un (t)|}^{p} + \frac{\rho}{2} p(p-1)\, {|\un (t)|}^{p-2} +\eps^{-p/2}\, {|f(t)|}^p_{V^\prime }
\\
&&\hspace{+2truecm}\lefteqn{  + p \,  {|\un (t)|}^{p-2} \dual{\un (t)}{ G ( \un (t) ) \, d W(t) }{} }
\end{eqnarray*}
Let us choose $\eps \in (0,1) $ such that  $ \delta=\delta(p,\eta):=p - p \frac{\eps}2 - \frac{1}{2} p(p-1) (2- \eta ) >0 $,
or equivalently,
$$
  \eps <  1 \wedge [ 2 -  (p-1) (2-\eta )] .
$$
Notice that under  condition (\ref{eqn-p_cond}) such $\eps $ exists. Denote also 
\[
K_p({\lambda }_{0},\rho ):= \frac{{\lambda}_{0}}{2}p(p-1)+p-1+\rho p(1-\frac{2}{p})\frac{p-1}{2}
=\frac{p-1}{2} [{\lambda }_{0}p+2 +\rho (p-2)] .
\]
Thus, since by Young inequality $x^{p-2}\leq(1-\frac2p)x^p+\frac2p1^{p/2}$ for $x\geq 0$,   we obtain
\begin{equation}
\begin{array}{rcl}
  {|\un (t)|}^{p}   &+&\delta
  \, \int_{0}^{t}{|\un (s)|}^{p-2} \norm{\un (s)}{}{2} \, ds \\
  && \\
 & \le&    {|\un (0)|}^{p} +K_p({\lambda }_{0},\rho)
\int_{0}^{t} \, {|\un (s)|}^{p} \, ds
  +    \rho (p-1) t \,
  + \eps^{-p/2} \int_{0}^{t} {|f(t)|}_{V^{\prime }}^{p}\, ds
  \\
  && \\
& +& p \,  \int_{0}^{t}{|\un (s)|}^{p-2} \dual{\un (s)}{ G ( \un (s) ) \, d W(s) }{} ,
\quad t \in [0,T].
\end{array}  \tag{A.1}   \label{E:apriori}
\end{equation}

\bigskip  \noindent
By Lemma \ref{L:Galerkin_existence}, we infer that the process
$$
 {\mu }_{n}(t):= \int_{0}^{t} {|\un (s)|}^{p-2} \dual{\un (s)}{G(\un (s)) \, d W(s) }{} ,
 \qquad t \in [0,T]
$$
is a square integrable martingale and that $\e [{\mu }_{n} (t) ] = 0 $. Thus
\begin{equation}
\begin{array}{cc}
  \e \bigl[ {|\un (t)|}^{p} \bigr]
  \, \, &\le   \, \, \e \bigl[ {|\un (0)|}^{p} \bigr] +
K_p({\lambda }_{0},\rho)\int_{0}^{t} \,
 \e \bigl[ {|\un (s)|}^{p} \bigr] \, ds
 \nonumber \\
 & \nonumber \\
  &+   \rho (p-1) t
  + \eps^{-p/2} \e \int_{0}^{t} {|f(s)|}_{V^{\prime }}^{p}\, ds, \quad t \in [0,T].
  \end{array}  \tag{A.2} \label{E:apriori'}
\end{equation}
Hence by the Gronwall Lemma there exists a constant $C=C_p(  |u_0 |^p, \int_{0}^{T} {|f(s)|}_{V^{\prime }}^{p}\, ds )>0$ such that
\begin{equation} \label{E:app_H_est}
  \e \bigl[ {|\un (t)|}^{p} \bigr]  \le C , \qquad \forall \, t \in [0,T] \quad \forall \, n \ge 1
  \tag{A.3}
\end{equation}
Using this bound in (\ref{E:apriori}) we also obtain
\begin{equation} \label{E:app_HV_est}
 \e \biggl[ \int_{0}^{T}{|\un (s)|}^{p-2} \norm{\un (s)}{}{2} \, ds  \biggr]
  \le C , \qquad n \ge 1  \tag{A.4}
\end{equation}
for a new constant  $C=\tilde C_p(\eta,  |u_0 |^p, \int_{0}^{T} {|f(s)|}_{V^{\prime }}^{p}\, ds )>0$. This completes the proof of estimates (\ref{E:HV_estimate}) and (\ref{E:V_estimate}).

\bigskip  \noindent
By the Burkholder-Davis-Gundy inequality, see \cite{DaPrato_Zabczyk_Erg}, inequality (\ref{A:G_2'}) and estimates (\ref{E:app_H_est}) and (\ref{E:app_HV_est}), we have the following inequalities (with still another constant $C>0$)
\begin{eqnarray*}
 & & \e \biggl[ \sup_{0 \le s \le t}
 \biggl| \int_{0}^{s} p \, {|\un (\sigma )|}^{p-2} \dual{\un (\sigma )}{\Pn G ( \un (\sigma ) ) \, d W(\sigma ) }{}
  \biggr| \biggr] \\
& &\le {C}_{p} \cdot \e \biggl[
 {\biggl( \int_{0}^{t} \, { |\un (\sigma )|}^{2p-2} \cdot
 \norm{  G ( \un (\sigma ) ) }{\lhs (Y,H)}{2}
   \, d\sigma
  \biggr) }^{\frac{1}{2}} \biggr]   \\
& &\le {C}_{p} \cdot \e \biggl[ \sup_{0\le \sigma \le t } {|\un (\sigma )|}^{\frac{p}{2}}
 {\biggl( \int_{0}^{t} \,  { |\un (\sigma )|}^{p-2} \cdot
 \norm{  G ( \un (\sigma ) ) }{\lhs (Y,H)}{2}
   \, d\sigma
  \biggr) }^{\frac{1}{2}} \biggr]   \\
& &\le {C}_{p} \cdot \e \biggl[ \sup_{0\le \sigma \le t } {|\un (\sigma )|}^{\frac{p}{2}}
 {\biggl( \int_{0}^{t} \,  { |\un (\sigma )|}^{p-2} \cdot
 \bigl[ {\lambda }_{0} \,  {|\un (\sigma )|}^{2} + \rho  + (2- \eta ) \norm{\un (\sigma )}{}{2} \bigr]
   \, d\sigma
  \biggr) }^{\frac{1}{2}} \biggr]   \\
& &\le \frac{1}{2} \e \bigl[ \sup_{0\le s \le t } {|\un (s )|}^{p} \bigr]
 + \frac{1}{2} {C}^{2}_{p} \biggl( {\lambda }_{0} +\rho \bigl( 1-\frac{2}{p}\bigr)\biggr)    \cdot
\e \biggl[ \int_{0}^{t} \, \sup_{0 \le s \le \sigma }{ |\un (s )|}^{p} \, d\sigma \biggr]
 + C
\end{eqnarray*}
Thus, by (\ref{E:apriori'}) we have
\begin{eqnarray*}
& & \e \bigl[ \sup_{0 \le s \le t }{|\un (s)|}^{p} \bigr]
   \le    \e \bigl[ {|\un (0)|}^{p} \bigr] +
K_p({\lambda }_{0},\rho) \int_{0}^{t} \,
 \e \bigl[ {\sup_{0 \le r \le s}|\un (r)|}^{p} \bigr] \, ds \\
& & +
   \rho (p-1) t
  + \eps^{-p/2} \e \int_{0}^{t} {|f(s)|}_{V^{\prime }}^{p}\, ds
 + \frac{1}{2} \e \bigl[ \sup_{0\le s \le t } {|\un (s )|}^{p} \bigr] \\
& & + \frac{1}{2} {C}^{2}_{p} \biggl( {\lambda }_{0} +\rho \bigl( 1-\frac{2}{p}\bigr)\biggr)    \cdot
 \int_{0}^{t} \, \e \bigl[\sup_{0 \le s \le \sigma }{ |\un (s )|}^{p}\bigr] \, d\sigma
 + C  , \quad t \in [0,T].
\end{eqnarray*}
By the Gronwall Lemma, we obtain (\ref{E:H_estimate}), i.e.
$$
  \e \bigl[ \sup_{0 \le s \le t }{|\un (s)|}^{p} \bigr]   \le {C}_{1}(p)
$$
for some constant ${C}_{1}(p)>0$.
This completes the proof of estimates (\ref{E:H_estimate}), (\ref{E:HV_estimate}) and (\ref{E:V_estimate})
in Lemma \ref{L:Galerkin_estimates }. \qed

\bigskip
%%%%%%%%%%%%%%%%%%%%%%%%%%%%%%%%%%%%%%%%%%%%%%%%%%%%%%%%%%%%%%%%%%%%%%%%%%%%%%%%%%%%%%%%%%%%%%%
%%%%%%%%%%%%%%%%%%%%%%%%%%%%%%%%%%%%%%%%%%%%%%%%%%%%%%%%%%%%%%%%%%%%%%%%%%%%%%%%%%%%%%%%%%%%%%%
\appendix{\bf Appendix B}
%%%%%%%%%%%%%%%%%%%%%%%%%%%%%%%%%%%%%%%%%%%%%%%%%%%%%%%%%%%%%%%%%%%%%%%%%%%%%%%%%%%%%%%%%%%%%%%
%%%%%%%%%%%%%%%%%%%%%%%%%%%%%%%%%%%%%%%%%%%%%%%%%%%%%%%%%%%%%%%%%%%%%%%%%%%%%%%%%%%%%%%%%%%%%%%

\bigskip \noindent
The following lemma is a generalization of the result contained in Corollary 5.3 in \cite{Brz+Li_2006} for 2D domains bounded in some direction \footnote{A domain $\ocal \subset \rd $ is bounded in some direction if there exists a nonzero vector $h \in \rd $ such that $\ocal \cap (h+\ocal )= \emptyset $. Here $h+\ocal := \{ h + x , \, x \in \ocal  \} $.}. Here, we consider any 2D or 3D domain, possibly unbounded, with smooth boundary. Moreover, we impose weaker assumptions on the sequence ${(\un )}_{n} $.

\bigskip  \noindent
\bf Lemma B.1 \it %\label{L:B_conv_aux_V_gamma}
Let $u \in {L}^{2}(0,T;H)$ and let  ${(\un )}_{n} $ be a bounded sequence in ${L}^{2}(0,T;H)$ such that $\un \to u $ in ${L}^{2}(0,T;{H}_{loc})$. Let $\gamma > \frac{d}{2}+1$.
Then for all $t \in [0,T]$ and all $\psi \in {V}_{\gamma }$:
\begin{equation} \label{E:B_conv_aux_V_gamma}
  \nlim \int_{0}^{t}  \dual{ B \bigl( \un (s) ,\un (s) \bigr)  }{ \psi }{}\, ds
  = \int_{0}^{t}  \dual{ B \bigl( u (s),u(s)  \bigr)  }{\psi }{}\, ds .  \tag{B.1}
\end{equation}
Here $\dual{\cdot }{\cdot }{}$ denotes the dual pairing between ${V}_{\gamma }$ and ${V}_{\gamma }^{\prime }$.
\rm

\bigskip
\begin{proof}
 Assume first that $\psi \in \vcal $. Then there exists $R>0$ such that $\supp \psi $ is a compact subset of ${\ocal }_{R}$.
Then, using the integration by parts formula, we infer that for every $v ,w \in H $
\begin{equation} \label{E:estimate_B(O_R)_ext}
 | \dual{B(v,w)}{\psi }{} | = \biggl| \int_{{\ocal }_{R}} ( v  \cdot \nabla \psi ) w \, dx \biggr|
  \le \norm{u}{{L}^{2}({\ocal }_{R})}{}  \norm{w}{{L}^{2}({\ocal }_{R})}{}
  \norm{\nabla \psi }{{L}^{\infty }({\ocal }_{R})}{}
  \le C |u{|}_{{H}_{{\ocal }_{R}}}  |w{|}_{{H}_{{\ocal }_{R}}} \norm{\psi }{{V}_{\gamma }}{} .
 \tag{B.2}
\end{equation}
We have
$ B(\un , \un ) - B(u,u) =  B(\un -u , \un ) + B(u,\un -u) $.
Thus, using the estimate (\ref{E:estimate_B(O_R)_ext}) and the H\"{o}lder inequality, we obtain
\begin{eqnarray*}
& &\Bigl| \int_{0}^{t} \dual{ B \bigl( \un (s) ,\un (s) \bigr) }{\psi }{} \, ds
- \int_{0}^{t} \dual{B \bigl( u(s), u(s)\bigr)  }{\psi }{} \, ds\Bigr|  \\
& &\le \Bigl| \int_{0}^{t} \dual{ B \bigl( \un (s) - u(s) ,\un (s) \bigr) \ }{\psi }{} , ds \Bigr|
  + \Bigl| \int_{0}^{t} \dual{ B \bigl( u (s) ,\un (s) - u(s) \bigr)  }{\psi }{} \, ds \Bigr| \\
& &\le \Bigl( \int_{0}^{t}   |\un (s) -u(s) {|}_{{H}_{{\ocal }_{R}}}  |\un (s){|}_{{H}_{{\ocal }_{R}}} \, ds
   + \int_{0}^{t}   |u(s) {|}_{{H}_{{\ocal }_{R}}}  |\un (s) - u(s){|}_{{H}_{{\ocal }_{R}}} \, ds  \Bigr)
 \cdot \norm{\psi }{{V}_{\gamma }}{}  \\
& & \le C
\norm{\un - u }{{L}^{2}(0,T;{H}_{{\ocal }_{R}})}{} \bigl( \norm{\un  }{{L}^{2}(0,T;{H}_{{\ocal }_{R}})}{}
+\norm{  u }{{L}^{2}(0,T;{H}_{{\ocal }_{R}})}{} \bigr)
\norm{\psi }{{V}_{\gamma }}{} ,
\end{eqnarray*}
where $C$ stands for some positive constant.
Since $\un \to u $ in ${L}^{2}(0,T;{H}_{loc})$, we infer that (\ref{E:B_conv_aux_V_gamma}) holds for every $\psi \in \vcal $.

\bigskip  \noindent
If $\psi \in {V}_{\gamma }$ then for every  $\eps > 0 $ there exists ${\psi }_{\eps } \in \vcal $
such that $\norm{\psi - {\psi }_{\eps }}{{V}_{\gamma }}{} \le \eps $.
Then
\begin{eqnarray*}
& &\bigl| \dual{ B\bigl( \un (\sigma ) ,\un (\sigma ) \bigr) - B\bigl( u(\sigma ) ,u (\sigma ) \bigr) }{\psi }{} \bigr| \\
 & & \le  \bigl| \dual{ B\bigl( \un (\sigma ) ,\un (\sigma )\bigr) - B\bigl( u(\sigma ) ,u (\sigma ) \bigr) }{
\psi -{\psi }_{\eps }}{} \bigr|
 + \bigl| \dual{ B\bigl( \un (\sigma ) ,\un (\sigma )\bigr) - B\bigl( u(\sigma ) ,u (\sigma )\bigr)  }{{\psi }_{\eps } }{} \bigr| \\
& & \le  \bigl( \bigl| B\bigl( \un (\sigma ) ,\un (\sigma ) \bigr){\bigr| }_{{V}_{\gamma }^{\prime }}
  + \bigl| B\bigl( u (\sigma ) ,u (\sigma ) \bigr) {\bigr| }_{{V}_{\gamma }^{\prime }} \bigr)
   \norm{\psi -{\psi }_{\eps }}{{V}_{\gamma }}{}  \\
 & &  + \bigl| \dual{ B\bigl( \un (\sigma ) ,\un (\sigma )\bigr) - B\bigl( u(\sigma ) ,u (\sigma )\bigr)  }{{\psi }_{\eps } }{} \bigr| \\
 & & \le \eps  \bigl( |\un (\sigma ){|}_{H}^{2}
   +|u (\sigma ){|}_{H}^{2}\bigr)
 + \bigl| \dual{ B\bigl( \un (\sigma ) ,\un (\sigma )\bigr)
   - B\bigl( u(\sigma ) ,u (\sigma )\bigr)  }{{\psi }_{\eps } }{} \bigr| .
\end{eqnarray*}
Hence
\begin{eqnarray*}
& &\bigl| \int_{s}^{t} \dual{ B\bigl( \un (\sigma ) ,\un (\sigma ) \bigr) - B\bigl( u(\sigma ) ,u (\sigma ) \bigr) }{\psi }{}  \, d\sigma \bigr|
\\
& & \le  \eps \int_{0}^{t}\bigl(  |\un (\sigma ){|}_{H}^{2} +  |u (\sigma ){|}_{H}^{2}\bigr) \, d \sigma
 + \bigl| \int_{s}^{t}\dual{ B\bigl( \un (\sigma ) ,\un (\sigma )\bigr) - B\bigl( u(\sigma ) ,u (\sigma )\bigr)  }{{\psi }_{\eps } }{} \, d\sigma \bigr| \\
& & \le \eps \cdot \bigl( \sup_{n\ge 1}\norm{\un }{{L}^{2}(0,T;H)}{2} +\norm{u }{{L}^{2}(0,T;H)}{2} \bigr)
+ \bigl| \int_{s}^{t}\dual{ B\bigl( \un (\sigma ) ,\un (\sigma )\bigr) - B\bigl( u(\sigma ) ,u (\sigma )\bigr)  }{{\psi }_{\eps } }{} \, d\sigma \bigr|
.
\end{eqnarray*}
Passing to the upper limit as $n \to \infty $, we obtain
$$
 \limsup_{n \to \infty }
\bigl| \int_{s}^{t} \dual{ B\bigl( \un (\sigma ) ,\un (\sigma ) \bigr) - B\bigl( u(\sigma ) ,u (\sigma ) \bigr) }{\psi }{}  \, d\sigma \bigr|
 \le   M \eps ,
$$
where $M:=\sup_{n\ge 1}\norm{\un }{{L}^{2}(0,T;H)}{2} +\norm{u }{{L}^{2}(0,T;H)}{2}<\infty $.
Since $\eps >0$ is arbitrary, we infer that
$$
 \lim_{n \to \infty }
 \int_{s}^{t} \dual{ B\bigl( \un (\sigma) ,\un (\sigma ) \bigr) - B\bigl( u(\sigma ) ,u (\sigma ) \bigr) }{\psi }{}  \, d\sigma =0.
$$
This completes the proof.
\end{proof}

\bigskip  \noindent
In the following corollary we state a result which is used  in the proof of Lemma \ref{L:pointwise_conv}.

\bigskip  \noindent
\bf Corollary B.2 \it %\label{L:B_conv_aux}
Let $u \in {L}^{2}(0,T;H)$ and let  ${(\un )}_{n} $ be a bounded sequence in ${L}^{2}(0,T;H)$ such that $\un \to u $ in ${L}^{2}(0,T;{H}_{loc})$.
Then for all $t \in [0,T]$ and all $\psi \in U$:
\begin{equation} \label{E:B_conv_aux}
  \nlim \int_{0}^{t} \dual{ B \bigl( \un (s) ,\un (s) \bigr)  }{\Pn \psi }{} \, ds
  = \int_{0}^{t} \dual{ B \bigl( u (s),u(s)  \bigr)  }{\psi }{}\, ds .  \tag{B.3}
\end{equation}

\bigskip
\begin{proof}
Let $t \in [0,T]$ and let $\psi \in U$. We have
\begin{eqnarray*}
& &\int_{0}^{t}  \dual{ B \bigl( \un (s),\un (s)  \bigr) }{\Pn \psi }{} \, ds
 \\
& & = \int_{0}^{t} \dual{ B \bigl( \un (s),\un (s)  \bigr)  }{\Pn  \psi -\psi }{} \, ds
 + \int_{0}^{t} \dual{ B \bigl( \un (s),\un (s) \bigr) }{\psi }{} \, ds
 =: {I}_{1}(n) + {I}_{2}(n)  .
\end{eqnarray*}
Let $\gamma > \frac{d}{2}+1$.
Let us consider  first the term ${I}_{1}(n)$.
By (\ref{E:estimate_B_ext}), we have the following inequalities
\begin{eqnarray*}
& & |{I}_{1}(n)|
 \le \int_{0}^{t} \bigl| B \bigl( \un (s),\un (s)  \bigr)
 {\bigr| }_{{V}_{\gamma }^{\prime } }\, ds
 \cdot {\| \Pn \psi -\psi \| }_{{V}_{\gamma }}
 \le c\int_{0}^{T} |\un (s){|}_{H}^{2}\, ds  \cdot {\|\Pn \psi -\psi \|}_{{V}_{\gamma }}.
\end{eqnarray*}
Since the sequence ${(\un )}_{n \ge 1 } $ is bounded in ${L}^{2}(0,T;H)$, by (ii) in Lemma \ref{L:P_n|U} (c) we infer that
$$
 \lim_{n \to \infty } {I}_{1}(n) = 0 .
$$
Since $U \subset {V}_{\gamma }$, by  Lemma B.1 we infer that
$$
   \lim_{n \to \infty } {I}_{2}(n)  = \int_{0}^{t} \dual{ B \bigl( u (s),u(s)  \bigr)  }{\psi }{}\, ds.
$$
The proof of the lemma is thus complete.
\end{proof}

\bigskip
%%%%%%%%%%%%%%%%%%%%%%%%%%%%%%%%%%%%%%%%%%%%%%%%%%%%%%%%%%%%%%%%%%%%%%%%%%%%%%%%%%%%%%%%%%%%%%%%%%55
\appendix{\bf Appendix C: Some auxiliary results from functional analysis}
%%%%%%%%%%%%%%%%%%%%%%%%%%%%%%%%%%%%%%%%%%%%%%%%%%%%%%%%%%%%%%%%%%%%%%%%%%%%%%%%%%%%%%%%%%%%%%%%%%%%%%%

\bigskip
\noindent  \rm
The following result can be found in Holly and Wiciak, \cite{Holly_Wiciak_1995}.

\bigskip \noindent
\bf Lemma C.1 \rm \label{L:2_5_Holly_Wiciak} \rm (see Lemma 2.5, p.99 in \cite{Holly_Wiciak_1995}) \it
Consider a separable Banach space $\Phi $ having the following property
\begin{equation} \label{E:2_6_Holly_Wiciak}
  \mbox{there exists a Hilbert space $H$ such that $\Phi \subset H$ continuously. }
  \tag{C.1}
\end{equation}
Then there exists a Hilbert space $\bigl( \hcal , (\cdot |\cdot {)}_{\hcal } \bigr) $ such that
$\hcal \subset \Phi$, $\hcal $ is dense in $\Phi $ and the embedding $\hcal \hookrightarrow \Phi $ is compact.  \rm

\begin{proof}
Without loss of generality we can assume that $\dim \Phi = \infty $ and $\Phi $ is dense in $H$.
Since $\Phi $ is separable, there exists a sequence $({\varphi }_{n}{)}_{n \in \nat } \subset \Phi $
linearly dense in $\Phi $.
Since $\Phi $ is dense in $H$ and the embedding $\Phi \hookrightarrow H$ is continuous,
the subspace
 $span \{ {\varphi }_{1}, {\varphi }_{2},... \}  $  is dense in $H$.
After the orthonormalization of $({\varphi }_{n})$ in the Hilbert space
$\bigl( H, (\cdot |\cdot {)}_{H} \bigr) $ we obtain an orthonormal basis $({h }_{n})$ of this space.
Furthermore, the sequence $({h }_{n})$ is linearly dense in $\Phi $.
Since the natural embedding $\iota : \Phi \hookrightarrow H$ is continuous, we infer that
$$
  1 = |{h}_{n} {|}_{H} = |\iota ({h}_{n}){|}_{H} \le |\iota | \cdot |{h}_{n} {|}_{\Phi }
$$
and $ \frac{1}{|{h}_{n} {|}_{\Phi }} \le |\iota |  \mbox{ for all } n \in \nat .$

\bigskip  \noindent
Let us take ${\eta }_{0} \in (0,1)$ and define inductively a sequence $({\eta }_{n}{)}_{n \in \nat }$ by
$$
   {\eta }_{n} : = \frac{{\eta }_{n-1} +1}{2} , \qquad n=1,2,...
$$
The sequence $({\eta }_{n} )$ is strongly increasing and $\nlim {\eta }_{n} =1$.
Let us define a sequence $({r}_{n}{)}_{n \in \nat }$ by
$$
   {r}_{n} := \frac{1- {\eta }_{n}}{2 |{h}_{n}{|}_{\Phi }} > 0 ,  \qquad n=1,2,...
$$
Obviously $\nlim {r}_{n} =0$.
Let us consider the set
$$
 \hcal := \biggl\{ x \in H : \quad \sum_{n=1}^{\infty }
  \frac{1}{{r}_{n}^{2}} \cdot |(x|{h}_{n}{)}_{H}{|}^{2} < \infty  \biggr\}
$$
and the Hilbert space ${L}_{\mu }^{2} ({\nat }^{\ast } ,\mathbb{K} )$, where
$\mu : {2}^{{\nat }^{\ast }} \to [0,\infty ]$ is the measure given by the formula
$$
    \mu (M) := \sum_{n \in M } \frac{1}{{r}_{n}^{2}}, \qquad M \subset {\nat }^{\ast } .
$$
The linear operator
$$
   l : {L}_{\mu }^{2} ({\nat }^{\ast }, \mathbb{K} )  \ni \xi \mapsto
    \sum_{n=1}^{\infty } {\xi }_{n} {h}_{n} \in H
$$
is well defined. Moreover, $l$ is an injection and hence we may introduce the following inner product
$$
 (\cdot |\cdot {)}_{\hcal } := (\cdot |\cdot {)}_{{L}^{2}} \, \underline{\circ } \,  {l}^{-1}
  : \hcal \times \hcal \ni (x,y) \mapsto ({l}^{-1}x|{l}^{-1}y{)}_{{L}^{2}} \in \mathbb{K} .
$$
Now, $l$ is an isometry onto the pre-Hilbert  space $(\hcal , (\cdot |\cdot {)}_{\hcal }) $
and consequently $\hcal $ is $(\cdot |\cdot {)}_{\hcal }$-complete.
Let us notice that for all $x,y\in \hcal $
$$
   (x|y {)}_{\hcal } = \sum_{n=1}^{\infty }  \frac{1}{{r}_{n}^{2}} \cdot (x|{h}_{n}{)}_{H} (y|{h}_{n}{)}_{H},
   \qquad |x{|}_{\hcal }^{2} = \sum_{n=1}^{\infty } \frac{1}{{r}_{n}^{2}} \cdot |(x|{h}_{n}{)}_{H}{|}^{2}
$$
We will show that $\hcal \subset \Phi $ continuously.
Indeed, let $x\in \hcal $, $|x{|}_{\hcal } \le 1 $. Then
$$
 |(x|{h}_{i}){h}_{i}{|}_{\Phi } = |(x|{h}_{i})| \cdot |{h}_{i}{|}_{\Phi }
 \le {r}_{i} |{h}_{i}{|}_{\Phi } = \frac{1-{\eta }_{i-1}}{2|{h}_{i}{|}_{\Phi } } |{h}_{i}{|}_{\Phi }
 = \frac{1-{\eta }_{i-1}}{2}  = {\eta }_{i} - {\eta }_{i-1}, \qquad i \in \nat .
$$
Thus, for any $k,n \in \nat $, $k< n$, we have the following estimate
$$
 \Bigl| \sum_{i=k+1}^{n}(x|{h}_{i}){h}_{i}{\Bigr| }_{\Phi }
 \le \sum_{i=k+1}^{n} ({\eta }_{i} - {\eta }_{i-1}) = {\eta }_{n} - {\eta }_{k}.
$$
Since in particular, the sequence $\bigl( {s}_{n}:= \sum_{i=1}^{n}(x|{h}_{i}){h}_{i} \bigr) $
is Cauchy in the Banach space $(\Phi , |\cdot {|}_{\Phi })$,
there exists $\varphi \in \Phi $ such that $\nlim | {s}_{n}-\varphi  {|}_{\Phi }=0$.
On the other hand, ${s}_{n} = \sum_{i=1}^{n}(x|{h}_{i}){h}_{i} \to x $ in $H$.
Thus by the uniqueness of the limit $\varphi =x \in \Phi $ and
$$
  \sum_{i=1}^{n}(x|{h}_{i}){h}_{i} \to x \quad \mbox{ in $\Phi $ } .
$$
Moreover,
$$
 |x{|}_{\Phi } \underset{\infty \leftarrow n}{\longleftarrow } |{s}_{n}{|}_{\Phi }
 \le {\eta }_{n} - {\eta }_{0}  \underset{n \to \infty }{\longrightarrow } 1 - {\eta }_{0} .
$$
Thus $\hcal \subset \Phi $ continuously (with the norm of the embedding not  exceeding $1 - {\eta }_{0}$).

\bigskip \noindent
We will show that  the embedding $j: \hcal \hookrightarrow \Phi $ is compact.
It is sufficient to prove that the ball $Z:= \{ x \in \hcal :|x{|}_{\hcal } \le 1 \} $
is relatively compact in $(\Phi , |\cdot {|}_{\Phi }) $.
According to the Hausdorff Theorem  it is sufficient to find (for every  fixed $\eps $)
an $\eps $-net of the set $j(Z)$.

\bigskip  \noindent
Since $\lim_{n \to \infty }{\eta }_{n}=1$, there exists $n \in \nat $ such that $1 - {\eta }_{n} \le \frac{\eps }{2}$.
The linear operator
$$
  {S}_{n} : \hcal \ni x \mapsto \sum_{i=1}^{n}(x|{h}_{i}){h}_{i}  \in \Phi
$$
being finite-dimensional is compact. Therefore ${S}_{n}(Z)$ is relatively compact in $(\Phi , |\cdot {|}_{\Phi }) $
and consequently there is a finite subset $F \subset \Phi $ such that
${S}_{n}(Z) \subset \bigcup_{\varphi \in Z } {\ball }_{\Phi }(\varphi , \frac{\eps }{2})$.

\bigskip  \noindent
We will show that the set $F$ is the $\eps $-net for $j(Z)$. Indeed, let $x \in Z$.
Then ${S}_{N}(x) \to x $ in $(\Phi , |\cdot {|}_{\Phi }) $ and
$$
 |x - {S}_{n}(x){|}_{\Phi } \underset{\infty \leftarrow N}{\longleftarrow } |{S}_{N}(x)- {S}_{n}(x){|}_{\Phi }
 \le {\eta }_{N} - {\eta }_{n}  \underset{N \to \infty }{\longrightarrow } 1 - {\eta }_{n} \le \frac{\eps }{2} .
$$
On the other hand, ${S}_{n}(x) \in {S}_{n}(Z)$, so, there is $\varphi \in F $ such that
${S}_{n}(x) \in {\ball }_{\Phi }(\varphi , \frac{\eps }{2})$.
Finally,
$$
 |x - \varphi {|}_{\Phi } \le |x - {S}_{n}(x){|}_{\Phi } + |{S}_{n}(x) - \varphi {|}_{\Phi }
 \le \frac{\eps }{2} + \frac{\eps }{2} = \eps ,
$$
i.e. $x \in {\ball }_{\Phi }(\varphi ,\eps )$. Thus
$$
 Z \subset \bigcup_{\varphi \in \Phi } {\ball }_{\Phi }(\varphi ,\eps ) .
$$
The proof is thus complete.
\end{proof}

\bigskip
%%%%%%%%%%%%%%%%%%%%%%%%%%%%%%%%%%%%%%%%%%%%%%%%%%%%%%%%%%%%%%%%%%%%%%%%%%%%%%%%%%%%%%%%%%%%%%%555
\appendix{\bf Appendix D: Auxiliary deterministic result}
%%%%%%%%%%%%%%%%%%%%%%%%%%%%%%%%%%%%%%%%%%%%%%%%%%%%%%%%%%%%%%%%%%%%%%%%%%%%%%%%%%%%%%%%%%%%%%%%%

\bigskip  \noindent
Assume that $d=2$.
Let $z \in {L}^{2}(0,T,V) \cap \ccal ([0,T];H)$, $f\in {L}^{2}(0,T;V^{\prime })$ and ${u}_{0}\in H $ be given.  Let us consider the following problem
\begin{equation}  \label{E:appendix_2D_NS_shift}
\begin{cases}
   & \frac{dv(t)}{dt} =   - A v(t) +v(t)+z(t) -  B (v(t)+z(t))  + f(t) , \qquad t \in (0,T) ,   \\
   &  v(0) = {u}_{0}.  % \label{E:appendix_2D_init_cond}
\end{cases}   \tag{D.1}
\end{equation}

\bigskip  \noindent
\bf Definition D.1 \rm
We say that $v \in {L}^{2}(0,T,V) \cap \ccal ([0,T];{H}_{w})$ is a \bf weak solution \rm of problem
(\ref{E:appendix_2D_NS_shift}) iff
for  all $t\in [0,T] $ and all $\phi \in V $ the following equality holds
\begin{equation}
\begin{array}{rrr}
  \ilsk{v(t)}{\phi }{H} = \ilsk{{u}_{0}}{\phi }{H} - \int_{0}^{t} \ilsk{v(s)}{\phi }{V} \, ds
   + \int_{0}^{t} \ilsk{v(s)+z(s)}{\phi }{H} \, ds  & & \nonumber \\
   & & \nonumber \\
 - \int_{0}^{t} \dual{B (v(s)+z(s))}{\phi }{} \, ds
   + \int_{0}^{t} \dual{f(s)}{\phi }{} \, ds . & &  \tag{D.2}\label{E:appendix_2D_NS_shift_identity}
\end{array}
\end{equation}

\bigskip  \noindent
Using the Galerkin method and the compactness criterion contained in Lemma \ref{L:comp}
we will prove the existence and uniqueness of the weak solutions.
Moreover, using Lemma III.1.2 in \cite{Temam79} we show that the weak solution is almost everywhere equal to a continuous $H$-valued function.

\bigskip  \noindent
\bf Theorem D.2 \it  \label{T:App_D_existence_unique}
Assume that $d=2$.
Let $z \in {L}^{2}(0,T,V) \cap \ccal ([0,T];H)$, $f\in {L}^{2}(0,T;V^{\prime })$ and ${u}_{0}\in H $.
\begin{itemize}
\item[(a)] There exists a unique weak solution  of problem (\ref{E:appendix_2D_NS_shift}).
\item[(b)] The weak solution is almost everywhere equal to a continuous $H$-valued function defined on $[0,T]$.
\end{itemize}
\rm

\begin{proof}
\bigskip  \noindent
Let $\{ {e}_{i} {\} }_{i =1}^{\infty  }$ be the orthonormal basis in $H$ composed of eigenvectors of the operator $L$ defined by
(\ref{E:op_L}).
Let ${H}_{n}:= span \{ {e}_{1},  ..., {e}_{n} \} $ be the subspace with the norm inherited from $H$ and
let $\Pn : U^{\prime } \to {H}_{n} $  be defined by
(\ref{E:tP_n}).
Let us consider the Faedo-Galerkin approximation in the space ${H}_{n}$
\begin{equation} \label{E:2D_NS_shift_Galerkin}
\begin{cases}
 &   \frac{d \vn (t)}{dt} = \Pn  [- A \vn (t) +\vn (t)+ z(t)-B(\vn (t)+z(t)) +f(t)],%\qquad t \in (0,T) ,
\\
 &   \vn (0) = \Pn {u}_{0}.  %\label{E:appendix_2D_init_cond_Galerkin}
\end{cases} \tag{D.3}
\end{equation}
In a standard way we infer that there exists a unique absolutely continuous local solution of problem (\ref{E:2D_NS_shift_Galerkin}). Moreover, by Lemma
\ref{L:A_acal_rel}
(a), we obtain %for all $t \in (0,T)$
\begin{equation} \label{E:app_2D_derivative}
 \frac{d}{dt} |\vn (t){|}_{H}^{2} = - 2 \norm{\vn (t)}{}{2} + 2 \ilsk{z(t)}{\vn (t)}{H}
   - 2\dual{B(\vn (t) + z(t))}{\vn (t)}{} + 2 \dual{f(t)}{\vn (t)}{} .  \tag{D.4}
\end{equation}
\it Step 1. \rm We establish some estimates on the Galerkin solutions $\vn $.
Let us fix  $n \in \nat $ and $t\in (0,T)$.
We have the following inequalities
\begin{equation}
\begin{array}{lll}
 & & 2 \ilsk{z(t)}{\vn (t)}{H} \le 2 |z(t){|}_{H}^{} |\vn (t){|}_{H}
  \le |z(t){|}_{H}^{2} + |\vn (t){|}_{H}^{2} , \tag{D.5} \label{E:app_2D_z_vn_est}
\end{array}
\end{equation}
and
\begin{equation}
 2 \dual{f(t)}{\vn (t)}{} \le 2 |f(t){|}_{V^\prime } \norm{\vn (t)}{V}{}
     \le |f(t){|}_{V^\prime }^{2} + |\vn (t){|}_{H}^{2}+ \norm{\vn (t)}{}{2} .
  \tag{D.6} \label{E:app_2D_f_vn_est}
\end{equation}
By (\ref{E:wirowosc_b}) and (\ref{E:antisymmetry_b}), we obtain
\begin{eqnarray*}
 \dual{B(\vn (t) + z(t))}{\vn (t)}{} = \dual{B(\vn (t) + z(t), z(t))}{\vn (t)}{}
 = - \dual{B(\vn (t) + z(t), \vn (t))}{z(t)}{}.
\end{eqnarray*}
We claim that for every $\eps > 0$ there exists a constant ${C}_{\eps }>0$ such that
\begin{equation} \label{E:2D_b_L^4}
 \bigl| \dual{B(\vn (t) + z(t), \vn (t))}{z(t)}{} \bigr| \le \eps \norm{\vn (t)}{}{2} +
  {C}_{\eps } (|\vn (t){|}_{H}^{2} +1 ) \norm{z(t)}{{L}^{4}}{4}.  \tag{D.7}
\end{equation}
Indeed, %by (\ref{E:b_estimate_L^4})
we have
\begin{equation}
\begin{array}{llll}
& \bigl| \dual{B(\vn (t) + z(t), \vn (t))}{z(t)}{} \bigr|  & \le &
\norm{\vn (t)+z(t)}{{L}^{4}}{}  \norm{z(t)}{{L}^{4}}{} \norm{\vn (t)}{}{} \nonumber \\
& & \le & \frac{\eps }{2} \norm{\vn (t)}{}{2}
+ {\tilde{C}}_{\eps } \norm{\vn (t)+z(t)}{{L}^{4}}{2}  \norm{z(t)}{{L}^{4}}{2} \tag{D.8}\label{E:2D_b_L^4_1}
\end{array}
\end{equation}
for some constant ${\tilde{C}}_{\eps }>0$. Moreover, by (\ref{E:2D_norm_L^4_est}) we obtain
\begin{equation}
\begin{array}{llll}
{\tilde{C}}_{\eps }  \norm{\vn (t)+z(t)}{{L}^{4}}{2}  \norm{z(t)}{{L}^{4}}{2}
& & \le &  2  {\tilde{C}}_{\eps } \bigl(  \norm{\vn (t)}{{L}^{4}}{2}  +  \norm{ z(t)}{{L}^{4}}{2}  \bigr)
   \norm{z(t)}{{L}^{4}}{2} \nonumber \\
 & & \le & 2 \sqrt{2}{\tilde{C}}_{\eps }  |\vn (t){|}_{H} \norm{\vn (t)}{}{} \cdot \norm{z(t)}{{L}^{4}}{2}
    + 2 {\tilde{C}}_{\eps } \norm{ z(t)}{{L}^{4}}{4} \nonumber \\
 & & \le & \frac{\eps }{2}  \norm{\vn (t)}{}{2} +  {\tilde{\tilde{C}}}_{\eps } |\vn (t){|}_{H}^{2}
  \norm{z(t)}{{L}^{4}}{4}
    + 2 {\tilde{C}}_{\eps } \norm{ z(t)}{{L}^{4}}{4} . \tag{D.9} \label{E:2D_b_L^4_2}
\end{array}
\end{equation}
By (\ref{E:2D_b_L^4_1}) and (\ref{E:2D_b_L^4_2})  the proof of inequality (\ref{E:2D_b_L^4}) is complete.

\bigskip  \noindent
Using inequalities (\ref{E:app_2D_z_vn_est}), (\ref{E:app_2D_f_vn_est}) and inequality
(\ref{E:2D_b_L^4}) with $\eps := \frac{1}{2}$ in (\ref{E:app_2D_derivative}) we infer that
\begin{equation} \label{E:app_2D_derivative_est}
  \frac{d}{dt} |\vn (t){|}_{H}^{2} + \frac{1}{2} \norm{\vn (t)}{}{2}
\le a(t) + \theta (t) |\vn (t){|}_{H}^{2} , %\quad t \in (0,T),
\tag{D.10}
\end{equation}
where $a(t) =|z(t){|}_{H}^{2} + |f(t){|}_{V^\prime }^{2} + {C}_{\frac{1}{2} }\norm{z(t)}{{L}^{4}}{4}$
and $\theta (t) =2+{C}_{\frac{1}{2} }\norm{z(t)}{{L}^{4}}{4}$.
Since $z \in {L}^{2}(0,T,V) \cap \ccal ([0,T];H)$, by inequality (\ref{E:2D_norm_L^4_est}) we have
$ \norm{z(t)}{{L}^{4}}{4} \le 2 {|z(t)|}_{H}^{2} \norm{z(t)}{}{2}$, $t \in [0,T]$,
and hence
\begin{eqnarray*}
  \int_{0}^{T}\norm{z(t)}{{L}^{4}}{4} \, dt \le 2 \sup_{t\in [0,T]} {|z(t)|}_{H}^{2}
   \int_{0}^{T} \norm{z(t)}{}{2}  \, dt < \infty .
\end{eqnarray*}
Notice that the functions $a$ and $ \theta $ are nonnegative and integrable.

\bigskip  \noindent
By (\ref{E:app_2D_derivative_est}), we obtain %for all $t \in [0,T]$
\begin{equation} \label{E:app_2D_derivative_est'}
   |\vn (t){|}_{H}^{2} + \frac{1}{2}\int_{0}^{t} \norm{\vn (s)}{}{2} \, ds
\le \int_{0}^{t}a(s) \, ds  + \int_{0}^{t} \theta (s)|\vn (s){|}_{H}^{2} \, ds. \tag{D.11}
\end{equation}
In particular, %for all $t\in [0,T]$
\begin{eqnarray*}
& &   |\vn (t){|}_{H}^{2}
\le \int_{0}^{t}a(s) \, ds  + \int_{0}^{t} \theta (s)|\vn (s){|}_{H}^{2} \, ds
\end{eqnarray*}
and hence by the Gronwall Lemma, see \cite[page 18]{Temam95}, for all $t\in dom(\vn )$
$$
   |\vn (t){|}_{H}^{2} \le |\vn (0){|}_{H}^{2}
   \exp \Bigl( \int_{0}^{t} \theta (s) \, ds  \Bigr)
  + \int_{0}^{t} a(s) \exp \Bigl( \int_{s}^{t} \theta (r) \, dr  \Bigr)\, ds  \nonumber
$$
\begin{equation}
  \le |{u}_{0}{|}_{H}^{2}
   \exp \Bigl( \int_{0}^{T} \theta (s) \, ds  \Bigr)
  + \exp \Bigl( \int_{0}^{T} \theta (r) \, dr  \Bigr) \int_{0}^{T} a(s) \, ds =:{K}_{1}.
 \tag{D.12} \label{E:app_2D_derivative_est''}
\end{equation}
Since inequality (\ref{E:app_2D_derivative_est''}) holds for all $t \in dom (\vn )$, we infer that
\begin{equation} \label{E:app_vn_est_L^infty_dom_vn}
   \sup_{t \in dom (\vn )} |\vn (t){|}_{H}^{2}  \le {K}_{1} . \tag{D.13}
\end{equation}
By (\ref{E:app_2D_derivative_est'}) and (\ref{E:app_vn_est_L^infty_dom_vn}) we have  also the following inequality
\begin{equation*} \label{E:app_vn_est_L^2_dom_vn}
   \int_{dom (\vn )} \norm{\vn (s)}{}{2} \, ds \le {K}_{2}
\end{equation*}
for some constant ${K}_{2}>0$.
Hence $dom (\vn ) = [0,T]$, i.e. the Galerkin solutions are defined on the whole interval $[0,T]$ and satisfy the following inequalities
\begin{equation} \label{E:app_vn_est_L^infty}
  \sup_{n\in \nat } \sup_{t \in [0,T]} |\vn (t){|}_{H}^{2}  \le {K}_{1} . \tag{D.14}
\end{equation}
and
\begin{equation} \label{E:app_vn_est_L^2}
  \sup_{n\in \nat } \int_{0}^{T} \norm{\vn (s)}{}{2} \, ds \le {K}_{2} \tag{D.15}
\end{equation}

\bigskip  \noindent
\it Step 2. \rm
Let us consider the sequence $(\vn )$ of  the Galerkin solutions.
Using Lemma  \ref{L:comp},
we will show that $(\vn )$ is relatively compact in the space $\zcal $ defined by
(\ref{E:Z}),
i.e.
\begin{equation*}
  \zcal =\ccal ([0,T];U^{\prime })\cap {L}_{w}^{2}(0,T;V)\cap {L}^{2}(0,T;{H}_{loc}) \cap \ccal ([0,T];{H}_{w}) .
\end{equation*}
Indeed,
since $(\vn )$ satisfies inequalities (\ref{E:app_vn_est_L^infty}) and (\ref{E:app_vn_est_L^2}),
it is sufficient to check that $(\vn )$ satisfies condition (c) in
Lemma \ref{L:comp}.

\bigskip  \noindent
By (\ref{E:2D_NS_shift_Galerkin}) and Lemma \ref{L:A_acal_rel} (a) we have
for all $s,t \in [0,T] $ such that $s \le t $
$$
 \vn (t) - \vn (s) = - \int_{s}^{t} \Pn \acal \vn (r) \,dr + \int_{s}^{t} \Pn z(r) \, dr
$$
 \begin{equation}
 \qquad  - \int_{s}^{t} \Pn B(\vn (r) + z(r)) \, dr + \int_{s}^{t} \Pn f(r) \, dr .
  \tag{D.16} \label{E:app_mod_cont}
\end{equation}
Let us recall that we have the following continuous embeddings
$$
U \subset {V}_{\gamma } \subset V \subset H \cong H^\prime \subset U^\prime ,
$$
where $\gamma  > \frac{d}{2} +1$.
Let us fix $\delta >0 $ and assume that $|t-s|\le \delta $.
We will estimate  each term on the right-hand side of equality (\ref{E:app_mod_cont}).

\bigskip  \noindent
By continuity of the embedding $U \subset V$, (\ref{E:Acal_V'_norm}), the H\"{o}lder inequality,  Lemma \ref{L:P_n|U} (c) and (\ref{E:app_vn_est_L^2})  we have
\begin{equation} \label{E:app_mod_cont_est_1}
 \Bigl| \int_{s}^{t} \Pn \acal \vn (r) \,dr {\Bigr| }_{U^\prime }
 \le c{\delta }^{\frac{1}{2}} \Bigl( \int_{s}^{t} \norm{\vn (r)}{}{2} \, dr {\Bigr) }^{\frac{1}{2}}
 \le {c}_{1} {\delta }^{\frac{1}{2}}.  \tag{D.17}
\end{equation}
Since the embedding $U \subset H$ is continuous and the restriction of $\Pn $ to $U$ is an orthogonal projection and  $z \in \ccal ([0,T];H)$, we have
\begin{equation} \label{E:app_mod_cont_est_2}
 \Bigl| \int_{s}^{t} \Pn  z(r) \,dr {\Bigr| }_{U^\prime } \le c \sup_{r \in [0,T]} |z(r){|}_{H} \cdot \delta
  = {c}_{2} \delta  \tag{D.18}
\end{equation}
By the continuity of the embedding $U \subset {V}_{\gamma }$, Lemma  \ref{L:P_n|U} (c) and  inequalities
 (\ref{E:estimate_B_ext}) and (\ref{E:app_vn_est_L^infty}) we obtain
\begin{equation} \label{E:app_mod_cont_est_3}
 \Bigl| \int_{s}^{t} \Pn  B(\vn (r) + z(r)) \, dr {\Bigr| }_{U^\prime}
 \le C \int_{s}^{t}  |\vn (r) + z(r){|}_{H}^{2} \, dr
 \le C \delta \sup_{r \in [0,T]} |\vn (r) + z(r){|}_{H}^{2}
 \le {c}_{3} \delta .   \tag{D.19}
\end{equation}
Since the embedding $U \subset V$ is continuous and $f \in {L}^{2}(0,T;V^\prime )$,
by the H\"{o}lder inequality and  Lemma \ref{L:P_n|U} (c) we obtain
\begin{equation} \label{E:app_mod_cont_est_4}
 \Bigl| \int_{s}^{t} \Pn  f(r) \,dr {\Bigr| }_{U^\prime }
 \le c{\delta }^{\frac{1}{2}} \Bigl( \int_{s}^{t} |f(r){|}_{V^\prime }^{2} \, dr {\Bigr) }^{\frac{1}{2}}
 \le {c}_{4} {\delta }^{\frac{1}{2}}.  \tag{D.20}
\end{equation}
By (\ref{E:app_mod_cont}) and inequalities (\ref{E:app_mod_cont_est_1})-(\ref{E:app_mod_cont_est_4})
we infer that
$$
  \lim_{\delta \to 0} \sup_{n \in \nat }  \sup_{\overset{s,t \in[0,T]}{|t-s|\le \delta }}
 {|\vn (t)-\vn (s)|}_{U^{\prime }} =0 .
$$
Hence by Lemma \ref{L:comp}
the sequence $(\vn )$ contains a subsequence, still denoted by $(\vn )$,
convergent in the space $\zcal $ to some $v \in \zcal $. In particular,
$v \in {L}^{2}(0,T;V) \cap \ccal ([0,T];{H}_{w})$.

\bigskip \noindent
\it Step 3. \rm  We will prove that $v$ satisfies equality (\ref{E:appendix_2D_NS_shift_identity}).
By (\ref{E:2D_NS_shift_Galerkin}), Lemma \ref{L:A_acal_rel} (a) and equalities (\ref{E:Acal_ilsk_Dir})
and (\ref{E:tP_n-P_n}), we infer that for all  $\phi \in U $
$$
  \ilsk{\vn (t)}{\phi }{H} = \ilsk{{u}_{0}}{\Pn \phi }{H}
  - \int_{0}^{t} \dirilsk{\vn (s)}{\Pn \phi }{} \, ds
   + \int_{0}^{t} \ilsk{z(s)}{\Pn \phi }{H} \, ds
$$
\begin{equation}
  \qquad \qquad  - \int_{0}^{t} \dual{ B (\vn (s)+z(s)) }{\Pn \phi }{} \, ds
   +  \int_{0}^{t} \dual{ f(s) }{\Pn \phi }{} \, ds . \tag{D.21}  \label{E:2D_NS_shift_Galerkin_int}
\end{equation}
 We will pass to the limit as $n \to \infty $ in equality (\ref{E:2D_NS_shift_Galerkin_int}).
Since $\vn \to v$ in the space $\zcal $, by Lemmas \ref{L:P_n|U} (c)  and B.2  we find that for all  $t \in [0,T]$ and $\phi \in U$
$$
  \ilsk{v (t)}{\phi }{H} = \ilsk{{u}_{0}}{ \phi }{H}
  - \int_{0}^{t} \dirilsk{ v (s)}{\phi }{} \, ds
  + \int_{0}^{t} \ilsk{z(s)}{ \phi }{H} \, ds
$$
\begin{equation}
 \qquad \qquad  - \int_{0}^{t} \dual{ B (v (s)+z(s)) }{ \phi }{} \, ds
   +  \int_{0}^{t} \dual{ f(s) }{ \phi }{} \, ds . \tag{D.22} \label{E:2D_NS_int}
 \end{equation}
Since $U$ is dense in the space $V$, equality holds for all $\phi \in V$, as well.
Hence by (\ref{E:V_il_sk}),
$v$ satisfies equality (\ref{E:appendix_2D_NS_shift_identity}), i.e. $v$ is a weak solution of problem (\ref{E:appendix_2D_NS_shift}).

\bigskip  \noindent
\it Step 4. \rm  (Uniqueness)
Let ${v}_{1},{v}_{2}$ be two weak solutions of problem (\ref{E:appendix_2D_NS_shift}).
Let $w:= {v}_{1}-{v}_{2}$. Since
$$
  B({v}_{1}+z) - B({v}_{2}+z) = B({v}_{1}+z,w) + B(w,{v}_{2}+z) ,
$$
$w$ satisfies the following equation
\begin{equation*}
\begin{cases}
  & \frac{dw}{dt} = - \acal w - B({v}_{1}+z,w) - B(w,{v}_{2}+z) , \\
  & w(0) =0 .
\end{cases}
\end{equation*}
By Lemma \ref{L:III.3.4_Temam'79}, we infer that $w^{\prime } \in {L}^{2}(0,T;V^{\prime })$ and hence
$$
   \frac{1}{2} \frac{d}{dt} |w(t){|}_{H}^{2} = - \norm{w(t)}{}{2} - \dual{B(w,{v}_{2}+z)}{w}{} .
$$
Moreover, by (\ref{E:III.3.53})
\begin{eqnarray*}
  \bigl| - \dual{B(w,{v}_{2}+z)}{w}{} \bigr| = \bigl|  \dual{B(w)}{{v}_{2}+z}{} \bigr|
  \le \sqrt{2} |w{|}_{H} \norm{w}{}{} \norm{{v}_{2}+z}{}{}
   \le \frac{1}{2} \norm{w}{}{2} +  |w{|}_{H}^{2} \norm{{v}_{2}+z}{}{2} .
\end{eqnarray*}
Thus
$$
  \frac{d}{dt} |w(t){|}_{H}^{2} + \norm{w(t)}{}{2} \le 2 |w(t){|}_{H}^{2} \norm{{v(t)}_{2}+z(t)}{}{2}
$$
and, in particular,
$$
    |w(t){|}_{H}^{2} \le |w(0){|}_{H}^{2} +2 \int_{0}^{t} |w(s){|}_{H}^{2} \norm{{v(s)}_{2}+z(s)}{}{2} \, ds , \qquad t \in [0,T].
$$
Since $w(0)=0$, by the Gronwall Lemma we infer that $w=0$, i.e. ${v}_{1}={v}_{2}$.

\bigskip  \noindent
Let us move to the proof of (b).
%We will prove that $v$ is almost everywhere equal to a $H$-valued continuous function defined on $[0,T]$.
Let us write equality (\ref{E:2D_NS_int}) in the following form
$$
  \frac{dv}{dt} = - \acal v +z - B(v+z) +f .
$$
Since $v\in {L}^{2}(0,T;V)$, by Lemma III.1.2 in \cite{Temam79} it is sufficient to show that $v^{\prime } \in {L}^{2}(0,T;V^{\prime })$.
The most difficulty appears in the nonlinear term. However, since $v,z \in {L}^{2}(0,T;V) \cap {L}^{\infty }(0,T;H)$, Lemma \ref{L:III.3.4_Temam'79} yields that $B(v+z) \in {L}^{2}(0,T;V^{\prime })$.
The proof of the theorem is thus complete.
\end{proof}

\bibliographystyle{model1a-num-names}
\bibliography{<your-bib-database>}

\end{document}